\documentclass[a4paper]{article}

\usepackage{hyperref}
\usepackage{amsmath,amsthm,amssymb}
\usepackage{listings}
\usepackage{graphicx}
\usepackage{epstopdf}

\usepackage{algorithm,algpseudocodex}

\usepackage[english]{babel}
\usepackage{authblk}

\DeclareMathOperator{\sgn}{sgn}

\newcommand{\lx}{\boldsymbol\lambda_x}
\newcommand{\ly}{\boldsymbol\lambda_y}
\newcommand{\LL}{\boldsymbol\Lambda}
\newcommand{\Lx}{\boldsymbol\Lambda_x}
\newcommand{\Ly}{\boldsymbol\Lambda_y}
\newcommand{\Lxy}{\boldsymbol\Lambda_{xy}}
\newcommand{\lj}{\boldsymbol\lambda_j}
\newcommand{\Lj}{\boldsymbol\Lambda_j}
\newcommand{\C}{\mathbf C}
\newcommand{\CC}{\mathbb C}
\newcommand{\D}{\mathbf D}
\newcommand{\Dx}{\mathbf D_x}
\newcommand{\Dxj}{\D_{x_jx_j}}
\newcommand{\Dxx}{\mathbf D_{xx}}
\newcommand{\Dyy}{\mathbf D_{yy}}
\newcommand{\diag}{\operatorname{diag}}

\newcommand{\I}{\mathbf I}
\newcommand{\Jn}{\mathbf 1_{N_1\times\ldots\times N_n}}
\newcommand{\Jx}{\mathbf 1_{N_x}}
\newcommand{\Jy}{\mathbf 1_{N_y}}
\newcommand{\Jxy}{\mathbf 1_{N_x\times N_y}}

\newcommand{\PP}{\mathbf P}
\newcommand{\Px}{\mathbf P_x}
\newcommand{\Py}{\mathbf P_y}
\newcommand{\R}{\mathbb R}
\newcommand{\U}{\mathbf U}
\newcommand{\tilU}{\widetilde\U}
\newcommand{\x}{\mathbf x}
\newcommand{\vx}{\vec x}
\newcommand{\vr}{\vec r}
\newcommand{\xxii}{\boldsymbol{\xi}}
\newcommand{\V}{\mathbf V}
\newcommand{\tilV}{\widetilde\V}

\newcommand{\y}{\mathbf y}
\newcommand{\vy}{\vec y}

\newtheorem{theorem}{Theorem}[section]

\newtheorem{lemma}[theorem]{Lemma}
\newtheorem{corollary}[theorem]{Corollary}
\theoremstyle{remark}
\newtheorem*{remark}{Remark}

\author[1,2]{Lo\"ic Constantin}

\author[1]{Carlota M. Cuesta}

\author[1]{Francisco de la Hoz\thanks{Corresponding author: Francisco de la Hoz, francisco.delahoz@ehu.eus}}

\affil[1]{Department of Mathematics, University of the Basque Country, Barrio Sarriena S/N, 48940 Leioa, Spain}
\affil[2]{Universit\'e de Pau et des Pays de l'Adour (UPPA), Av. de l'Universit\'e, 64000 Pau, France}

\title{A matrix-based spectral method for the numerical approximation of the fractional Laplacian and the fractional $p$-Laplacian of functions defined on $\R^n$}

\begin{document}

	\maketitle
	
	\begin{abstract}
Given a function $u$ defined on $\R^n$, its fractional $p$-Laplacian is given by
\begin{equation*}
(-\Delta)_p^s u(\vx) = C_1(n,s,p)\int_{\R^n}\frac{|u(\vx)-u(\vy)|^{p-2}(u(\vx)-u(\vy))}{\|\vx-\vy\|_2^{n+sp}}d\vy,\quad \vx\in\R^n,
\end{equation*}
where the integral is understood in the principal value sense, $p \in (1,\infty)$, $s \in (0,1)$, and $C_1(n,s,p)$ is a normalization constant depending on $n$, $s$, and $p$. A formally equivalent nonlinear Balakrishnan formulation is given by
\begin{equation*}
(-\Delta)_p^s u(\vx) = C_4(n,s,p)\int_0^\infty\Delta(t - \Delta)^{-1}\left[\Phi_p(u(\vx) - u(\cdot))\right](\vx)\frac{dt}{t^{1-sp/2}},
\end{equation*}
where $C_4(n,s,p)$ is another normalization constant, and $\Phi_p(t) = |t|^{p-2}t$.

In this paper, we present a matrix-based spectral method for the numerical approximation of both the fractional Laplacian (i.e., the linear case, when $p = 2$) and the fractional $p$-Laplacian for functions defined on $\R^n$. Our approach builds on the Balakrishnan representation, where we discretize the second-order derivatives in $\Delta$ using spectrally accurate differentiation matrices. A key advantage is that these matrices can be diagonalized in a well-conditioned manner, enabling a stable and robust numerical scheme that naturally extends to arbitrary spatial dimensions $n$. In particular, this diagonalization allows the fractional operator to act directly on the eigenvalue spectrum, effectively reducing the Balakrishnan integral to an analytical evaluation at the spectral level and thereby avoiding costly multidimensional quadrature. The resulting method also avoids domain truncation and variational formulations, making it both computationally efficient and conceptually straightforward.

As a practical application, we simulate numerically the evolution of
\begin{equation*}
\hspace{22em}
\frac{\partial u}{\partial t} + (-\Delta)^s_p u = 0,
\end{equation*}
in one and two spatial dimensions, being able to capture the self-similar solutions that arise as $t\to\infty$.


\end{abstract}

\maketitle

\section{Introduction}

The aim of this paper is to develop a pseudospectral method to approximate numerically the fractional Laplacian and the fractional $p$-Laplacian of $u$ on $\R^n$. The latter operator is a generalization of the former, and is defined (see, e.g., \cite{teso2021}) by
\begin{equation}
\label{e:Dpa}
(-\Delta)_p^s u(\vx) = C_1(n,s,p)\int_{\R^n}\frac{|u(\vx)-u(\vy)|^{p-2}(u(\vx)-u(\vy))}{\|\vx-\vy\|_2^{n+sp}}d\vy,\quad \vx\in\R^n,
\end{equation}
where $p\in(1,\infty)$, $s\in(0,1)$, with
\begin{equation*}
C_1(n,s,p) = \frac{sp(1-s)2^{2s-2}}{\pi^{(n-1)/2}}\frac{\Gamma((n+sp)/2)}{\Gamma((p+1)/2)\Gamma(2-s)} > 0,
\end{equation*}
where $\Gamma(\cdot)$ is Euler's Gamma function:
\begin{equation*}
\Gamma(z) = \int_0^\infty y^{z-1}e^{-y}dy.
\end{equation*}
Note that \eqref{e:Dpa} and all the other integrals appearing in this paper must be understood in the principal value sense, whenever applicable.
When $p = 2$, the fractional $p$-Laplacian is linear, and its definition coincides with that of the integral definition of the fractional Laplacian (see, e.g., \cite{kwasnicki}). Indeed, denoting $s = \alpha/2$, with $\alpha\in(0,2)$, \eqref{e:Dpa} becomes
\begin{align*}
	(-\Delta)_2^{\alpha/2}u(\vx) & = \frac{\alpha2^{\alpha-1}}{\pi^{n/2}}\frac{\Gamma((n+\alpha)/2)}{\Gamma(1-\alpha/2)}\int_{\R^n}\frac{u(\vx)-u(\vy)}{\|\vx-\vy\|_2^{n+\alpha}}d\vy 
	\cr
	& = \frac{\alpha2^{\alpha-1}}{\pi^{n/2}}\frac{\Gamma((n+\alpha)/2)}{\Gamma(1-\alpha/2)}\int_{\R^n}\frac{u(\vx)-u(\vx + \vy)}{\|\vy\|_2^{n+\alpha}}d\vy = (-\Delta)^{\alpha/2}u(\vx).
\end{align*}
In particular, when $p = 2$, the nonlinear factor in the integrand in \eqref{e:Dpa} equals one, which allows rewriting the numerator as an integral.

For example, in the one-dimensional case ($n = 1$):
\begin{equation}
\label{e:D2a}
(-\Delta)^{\alpha/2}u(x) = c_\alpha\int_{\R}\frac{u(x)-u(y)}{|x-y|^{1+\alpha}}dy = c_\alpha\int_{\R}\frac{u(x)-u(x+y)}{|y|^{1+\alpha}}dy,
\end{equation}
where
\begin{equation*}
c_\alpha = \frac{\alpha2^{\alpha-1}}{\pi^{1/2}}\frac{\Gamma((1+\alpha)/2)}{\Gamma(1-\alpha/2)}.
\end{equation*}
The formulation allows deriving other expressions for $(-\Delta)^{\alpha/2}u(x)$. Indeed, in \cite[Lemma 2.1]{cayamacuestadelahoz2021} the following result was proved:
\begin{lemma}
	Consider the twice continuous bounded function $u\in\mathcal C_b^2(\R)$, and let $\alpha\in(0,2)$. If $\alpha\in(0,1)$, assume additionally that $\lim_{x\to\pm \infty}u_x(x)=0$. Then, \eqref{e:D2a} can be expressed as
	\begin{equation}
		\label{e:D2aalternative}
		(-\Delta)^{\alpha/2}u(x) = \left\{
		\begin{aligned}
			& \frac{1}{\pi}\int_{-\infty}^\infty \frac{u_{x}(y)}{x-y}dy, & \alpha = 1, \\
			& \frac{c_{\alpha}}{\alpha(1-\alpha)}\int_{-\infty}^\infty \frac{u_{xx}(y)}{|x-y|^{\alpha - 1}}dy, & \alpha \neq 1.
		\end{aligned}
		\right.
	\end{equation}
\end{lemma}
Working with \eqref{e:D2aalternative} instead of with \eqref{e:D2a} is particularly advantageous from a numerical point of view, and particularly well-suited for the implementation of a pseudospectral method, thanks to its convolutional form. In this regard, we can mention the series of papers \cite{cayamacuestadelahoz2020,cayamacuestadelahoz2021,cuestadelahozgirona2024,cayamacuestadelahozgarciacervera2025}, where, after transforming the integration domain from $\R$ to $(0,\pi)$ by means of the change of variable $x = L\cot(\xi)$, with $L > 0$, and defining the operator $(-\Delta)_\xi^{\alpha/2} u(\xi) \equiv [(-\Delta)^{\alpha/2} u](L\cot(\xi))$, a Fourier-series expansion of the function $u(L\cot(\xi))$ is performed, and the problem of approximating numerically $(-\Delta)_\xi^{\alpha/2} u(\xi)$ is reduced to computing $(-\Delta)_\xi^{\alpha/2} e^{ik\xi}$, with $k\in\mathbb Z$ (note that, in this paper, the subscript in $(-\Delta)^s_p$ refers exclusively to the fractional $p$-Laplacian, so there is no risk of confusion). This allows the construction of a high-accuracy operational matrix, as in \cite{cayamacuestadelahoz2021}, but, additionally, when $k$ is even, it is possible to express explicitly $(-\Delta)_\xi^{\alpha/2} e^{ik\xi}$ for $\alpha\in(0,1)\cup(1,2)$ in terms of hypergeometric Gaussian functions ${}_2F_1$, as in \cite{cayamacuestadelahoz2020}, although their numerical application requires the use of arbitrary precision. On top of that, the important case with $\alpha = 1$, which corresponds to the so-called half-Laplacian, allows computing $(-\Delta)_\xi^{1/2} e^{ik\xi}$ explicitly also when $k$ is odd, and the resulting expression can be again expressed in terms of ${}_2F_1$, although its numerical implementation does not require the use of arbitrary precision, and allows the use of a very fast FFT-based convolution algorithm. Note that the use of an FFT-based convolution algorithm is also possible when $\alpha\in(0,1)\cup(1,2)$, at the expense of losing some accuracy, as in \cite{cayamacuestadelahozgarciacervera2025}, where a second-order modification of the midpoint rule is considered to approximate  $(-\Delta)_\xi^{\alpha/2} e^{ik\xi}$. 

Even if it is possible to find formulations similar to \eqref{e:D2aalternative} in higher dimensions (see, e.g., \cite[Appendix A]{cayamacuestadelahoz2021}), it is still not clear how to use them numerically. Therefore, in the literature, other approaches are followed when approximating numerically the fractional Laplacian in multiple dimensions without truncating the domain, which is $\R^n$. For instance, in \cite{sheng2020}, they use a variational form corresponding to the fractional Laplacian formulated by means of the Dunford-Taylor formulation. More precisely, for any $u, v\in H^s(\R^n)$, with $s\in(0,1)$,
\begin{equation}
\label{e:DunfordTaylor}
\langle(-\Delta)^{s/2}u,(-\Delta)^{s/2}v\rangle_{L^2(\R^n)} = \frac{2\sin(\pi s)}{\pi}\int_0^\infty t^{1-2s}\int_{\R^n}\left((-\Delta)(\mathbb I - t^2\Delta)^{-1}u\right)(\vx)v(\vx)d\vx dt,
\end{equation}
where $\mathbb I$ denotes the identity operator. Then, a fast spectral-Galerkin method is applied to \eqref{e:DunfordTaylor}, using mapped Chebyshev functions as basis functions.

In the literature, it is also common to talk about the Balakrishnan formulation of the fractional Laplacian. In fact, given an operator, the Dunford-Taylor formulation and the Balakrishnan formulation are related ways to define fractional powers of it, and they coincide in the case of the fractional Laplacian (see \cite{kwasnicki}). In this regard, in this paper, we follow \cite{teso2021}, which, among other things, offers a generalization of the Balakrishnan formula for the linear fractional Laplacian, given by
\begin{equation}
\label{e:balakrishnan}
(-\Delta)^s u(\vx) = \frac{\sin(\pi s)}{\pi}\int_0^\infty(-\Delta)(t - \Delta)^{-1}[u](\vx)\frac{dt}{t^{1-s}},
\end{equation}
to the nonlinear Balakrishnan formula for the fractional $p$-Laplacian, for $p\neq2$:
\begin{equation}
\label{e:balakrishnanp}
(-\Delta)_p^s u(\vx) = C_4(n,s,p)\int_0^\infty\Delta(t - \Delta)^{-1}\left[\Phi_p(u(\vx) - u(\cdot))\right](\vx)\frac{dt}{t^{1-sp/2}},
\end{equation}
where
\begin{equation*}
C_4(n, s, p) = \frac{\pi^{n/2}}{2^{sp}\Gamma((2+sp)/2)\Gamma((n+sp)/2)}C_1(n,s,p),
\end{equation*}
and
\begin{equation}
\label{e:Phip}
\Phi_p(t) = |t|^{p-2}t.
\end{equation}
Note that $\Delta(t - \Delta)^{-1}$ is applied to the variables denoted by a dot in $\Phi_p(u(\vx) - u(\cdot))$, after fixing $\vx$, then the result is a function that must be evaluated precisely at $\vx$, so it only depends on $t$.

Let us mention that there are other representations of the fractional Laplacian, besides the Bala\-krishnan formulation. Indeed, in \cite{teso2021}, the Bochner formula for the fractional Laplacian, given by
\begin{equation}
\label{e:bochner}
(-\Delta)^s u(\vx) = \frac{1}{\Gamma(-s)}\int_0^\infty(e^{t\Delta}[u](\vx) - u(\vx))\frac{dt}{t^{1+s}},
\end{equation}
is generalized to
\begin{equation}
\label{e:bochnerp}
(-\Delta)_p^s u(\vx) = C_2(n,s,p)\int_0^\infty e^{t\Delta}\left[\Phi_p(u(\vx) - u(\cdot))\right](\vx)\frac{dt}{t^{1+sp/2}},
\end{equation}
where
\begin{equation*}
C_2(n, s, p) = \frac{\pi^{n/2}}{2^{sp}\Gamma((n+sp)/2)}C_1(n,s,p).
\end{equation*}
Note that the representations \eqref{e:balakrishnanp} and \eqref{e:bochnerp} are formally valid (see \cite{teso2021}) when $p\in(1, \infty)$, $s\in(0,1)$ and $sp < 2$. However, the relationship between those formulas and \eqref{e:Dpa} is not completely clear, when, e.g., $sp \ge 2$.

To date, only a very limited number of studies have proposed concrete numerical discretizations or approximation strategies specifically for the fractional p-Laplacian. While \cite{tesoetal2018II} and \cite{tesoetal2019I} and the more recent \cite{tesoetal2024} established foundational monotone finite difference discretizations for nonlocal operators, and \cite{Vazquez2020} provided early numerical profiles of self-similar solutions, it should be noted that these methods were primarily implemented for the study of evolution equations. In contrast, the pseudospectral method developed in this paper offers a significant shift in methodology. Specifically, whereas the primary aim of these earlier works was to manage the mathematical challenges posed by singularities and regularity—ensuring convergence even for nonsmooth distributional solutions and handling the "singular range" where the derivatives of the nonlinearity become problematic, the current approach is engineered specifically for high spectral accuracy and computational precision. By leveraging the nonlinear Balakrishnan formulation with well-conditioned, spectrally accurate differentiation matrices, this method achieves streamlined performance across multiple dimensions without requiring domain truncation or the heavy computational burden of multidimensional quadrature. Indeed, unlike traditional methods that rely on domain truncation or heavy multidimensional quadrature, our approach utilizes a domain transformation to map the unbounded nature of $\R^n$ onto a finite domain while maintaining spectral precision.


In this paper, building on the integral representations \eqref{e:balakrishnan} and \eqref{e:balakrishnanp}, we use a pseudospectral method to approximate those operators numerically. As in \cite{cayamacuestadelahoz2020,cayamacuestadelahoz2021,cuestadelahozgirona2024,cayamacuestadelahozgarciacervera2025}, we transform in the one-dimensional case the integration domain from $\R$ to $(0,\pi)$ by means of the change of variable $x = L\cot(\xi)$, with $L > 0$. Then, we discretize the second-order partial derivatives appearing in the Laplacian operator $\Delta$ by means of spectrally accurate differentiation matrices, which are diagonalized afterwards to evaluate fractional powers efficiently. A key feature of the proposed method is that the diagonalization of the resulting second-order differentiation matrices is well conditioned, which enables the construction of a sound and robust numerical scheme that extends in a natural way to arbitrary spatial dimensions $n$ (by applying a similar transformation in each Cartesian coordinate). Moreover, this diagonalization allows the fractional operator, within the Balakrishnan framework, to act directly on the eigenvalue spectrum (see Theorem~\ref{t:fraclap2D}), effectively reducing the integral representation to an analytical evaluation at the spectral level and thereby avoiding costly multidimensional quadrature. The resulting method avoids domain truncation and variational formulations, making it both computationally efficient and conceptually straightforward. In particular, to the authors' knowledge, this is the first pseudospectral method for the fractional $p$-Laplacian on unbounded domains that naturally extends to multiple spatial dimensions.

The structure of this paper is as follows: in Section~\ref{s:matrices}, we construct second-order differentiation matrices associated to the following nodes:
\begin{equation*}
x_j = L\cot\left(\frac{\pi(2j-1)}{2N}\right), \quad 1 \le j \le N,
\end{equation*}
which are precisely the mapped Chebyshev points of the first kind.

In Section~\ref{s:fraclap}, after diagonalizing the matrices developed in Section~\ref{s:matrices}, and introducing them into \eqref{e:balakrishnan}, we develop a numerical method to approximate the fractional Laplacian in $n$ dimensions. We also show that it is equivalent to work with \eqref{e:bochner}. Note that, in order to facilitate the understanding of the method, we have found it convenient to start with the $n = 2$ case, because the $n = 1$ case follows trivially, and it is straightforward to generalize it to $n \ge 3$. However, in our opinion, proposing directly the method for any $n\in\mathbb N$ would make the exposition less clear.

In Section~\ref{s:fracplap}, we apply the ideas in Section~\ref{s:fraclap} to \eqref{e:balakrishnanp} to approximate numerically the fractional $p$-Laplacian in $n$ dimensions.

In Section~\ref{s:numericalfracLap}, we carry out the numerical experiments for the fractional Laplacian, and in Section~\ref{s:numericalfracpLap}, those for the fractional $p$-Laplacian. Finally, in Section~\ref{s:evolution}, we consider an evolution equation involving the fractional $p$-Laplacian.

In all cases, our programming language of choice has been MATLAB \cite{matlab}, which proves particularly suitable in those problems that imply a heavy use of arrays. In fact, we have tried to exploit the possibilities of MATLAB's syntax as much as possible, in order to develop compact and efficient implementations. In order to facilitate the understanding of the methods, we offer fully working codes.

Those experiments where we indicate the execution times have been run on an Apple MacBook Pro (13-inch, 2020, 2.3 GHz Quad-Core Intel Core i7, 32 GB).

\section{Construction of second-order differentiation matrices}

\label{s:matrices}

Suppose that $u(x)$ is a bounded, regular function defined on $\R$. As in \cite{cayamacuestadelahoz2020,cayamacuestadelahoz2021,cuestadelahozgirona2024,cayamacuestadelahozgarciacervera2025}, we perform the change of variable $x = L\cot(\xi)$, that maps $\R$ into $(0, \pi)$, then discretize $\xi\in(0,\pi)$ in $N$ equally spaced inner nodes 
\begin{equation}
\label{e:xi}
\xi_j = \frac\pi N\left(j-\frac12\right) = \frac{\pi(2j-1)}{2N}, \quad 1 \le j \le N.
\end{equation}
Therefore,
\begin{equation*}
x_j = L\cot\left(\frac{\pi(2j-1)}{2N}\right), \quad 1 \le j \le N;
\end{equation*}
note that the nodes $x_j$ are indeed the first kind Chebyshev points $y_j = \cos(\pi (j - 1/2) / N)$, for $1\le j\le N$, which are precisely the $N$ roots of the $N$th degree Chebyshev polynomial $T_N(x) = \cos(N\arccos(x))$, mapped under map $x = Ly(1 - y^2)^{-1/2}$, with $y\in(-1, 1)$. In fact, this map gives rise to the underlying rational Chebyshev functions $TB_m(x)$ (see, e.g., \cite[Section 17.7]{boyd2001}), where $TB_m(x) = TB_m(Ly(1 - y^2)^{-1/2}) = T_m(x)$, being $TB_m(x)$ and $T_m(x)$ the $m$th rational Chebyshev function and the $m$th Chebyshev polynomial, respectively.

Define $U(\xi) \equiv u(L\cot(\xi))$, then,
\begin{align*}
U_\xi(\xi) & = -\frac{L}{\sin^2(\xi)}u_x(L\cot(\xi)) \Longleftrightarrow u_x(x) = -\frac{\sin^2(\xi)}{L}U_\xi(\xi),
	\cr
U_{\xi\xi}(\xi) & = \frac{2L\cos(\xi)}{\sin^3(\xi)}u_x(L\cot(\xi)) + \frac{L^2}{\sin^4(\xi)}u_{xx}(L\cot(\xi)) 
	\cr
& \qquad \Longleftrightarrow  u_{xx}(L\cot(\xi)) = \frac{\sin^4(\xi)}{L^2}U_{\xi\xi}(\xi) - \frac{2\sin(\xi)\cos(\xi)}{L}u_x(L\cot(\xi))
	\cr
& \qquad \Longleftrightarrow  u_{xx}(x) = \frac{\sin^4(\xi)}{L^2}U_{\xi\xi}(\xi) + \frac{\sin^2(\xi)\sin(2\xi)}{L^2}U_\xi(\xi).
\end{align*}
Therefore, 
\begin{equation}
\label{e:uxuxx}
u_x(x_j) = -\frac{\sin^2(\xi_j)}{L}U_\xi(\xi_j), \qquad u_{xx}(x_j) = \frac{\sin^4(\xi_j)}{L^2}U_{\xi\xi}(\xi_j) + \frac{\sin^2(\xi_j)\sin(2\xi_j)}{L^2}U_\xi(\xi_j),
\end{equation}
so the numerical approximation of $\{u_{x}(x_j)\}$ can be reduced to that of $\{U_\xi(\xi_j)\}$, and the numerical approximation of $\{u_{xx}(x_j)\}$, to that of $\{U_\xi(\xi_j)\}$ and $\{U_{\xi\xi}(\xi_j)\}$, which can be computed by means of the differentiation matrices associated to periodic grids developed in \cite[Chapter 3]{trefethen2000}. Note, however, that $U(\xi)$ may not be periodic, i.e., it may happen that $U(0) \neq U(\pi)$. Therefore, we perform an even extension at $\xi = \pi$, i.e., $U(\pi + \xi) \equiv U(\pi - \xi)$, from which it follows that
\begin{equation*}
U(\xi_j) =
\begin{cases}
U(\xi_j), & 1 \le j \le N,
	\cr
U(\xi_{2N+1 - j}), & N+1 \le j \le 2N.
\end{cases}
\end{equation*}
In matrix form:
\begin{equation}
\label{e:Itilde}
\begin{pmatrix}
U(\xi_1) \\ \vdots \\ U(\xi_{2N})
\end{pmatrix}
=
\tilde\I
\cdot
\begin{pmatrix}
U(\xi_1) \\ \vdots \\ U(\xi_{N})
\end{pmatrix},
\end{equation}
where $\tilde\I\in\R^{(2N)\times N}$ denotes the block matrix consisting of the identity matrix $\I_N$ of order $N$ stacked above $\I_N$ with its rows in reverse order:
\begin{equation*}
\tilde\I =
\left(
\begin{array}{c}
	\I_N
	\\
	\hline
	\I_N(N:-1:1, :)
\end{array}
\right).
\end{equation*}
Even if in this paper we will only use an even extension, which is enough for our purposes, sometimes it is possible to consider also other types of extensions, like an odd extension at $\xi = \pi$, $U(\pi + \xi) \equiv -U(\pi - \xi)$, or simply a $\pi$-periodic extension, $U(\pi + \xi) \equiv U(\xi)$, in which case $\tilde\I$ would be defined as
\begin{equation*}
\tilde\I =
\left(
\begin{array}{c}
	\I_N
	\\
	\hline	
	-\I_N(N:-1:1, :)
\end{array}
\right)
\text{ (odd extension) or }
\tilde\I =
\left(
\begin{array}{c}
	\I_N
	\\
	\hline
	\I_N
\end{array}
\right)
\text{ ($\pi$-periodic extension)}.
\end{equation*}
Then, in order to approximate $U_\xi(\xi_j)$ and $U_{\xi\xi}(\xi_j)$ in \eqref{e:uxuxx}, for $1\le j\le N$, we use the spectral differentiation matrices defined in \cite[Chapter 3]{trefethen2000}, but taking $2N$ instead of $N$; note that, although we apply them to $U(\xi_j)$, for $1\le j\le 2N$, we only keep the first $N$ rows, and discard the last $N$ rows, which are not necessary, because we do not need $U_\xi(\xi_j)$ and $U_{\xi\xi}(\xi_j)$, for $N+1\le j\le 2N$. More precisely, we have
\begin{equation}
\label{e:UDxi}
\begin{pmatrix}
U_\xi(\xi_1) \\ \vdots \\ U_\xi(\xi_{N})
\end{pmatrix}
\approx
\D_\xi\cdot
\begin{pmatrix}
U(\xi_1) \\ \vdots \\ U(\xi_{2N})
\end{pmatrix},
\qquad
\begin{pmatrix}
U_{\xi\xi}(\xi_1) \\ \vdots \\ U_{\xi\xi}(\xi_{N})
\end{pmatrix}
\approx
\D_{\xi\xi}\cdot
\begin{pmatrix}
U(\xi_1) \\ \vdots \\ U(\xi_{2N})
\end{pmatrix},
\end{equation}
where $\D_\xi\in\R^{N\times(2N)}$ is given by
\begin{equation}
\label{e:Dxiij}
[\D_\xi]_{ij} =
\left\{
\begin{aligned}
	& 0, & & i = j,
	\cr
	& \frac12(-1)^{i-j}\cot\left(\frac{\pi(i-j)}{2N}\right), & & i \neq j,
\end{aligned}
\right.
\end{equation}
and $\Dxx$ by
\begin{equation}
\label{e:Dxixiij}
[\D_{\xi\xi}]_{ij} =
\left\{
\begin{aligned}
	& {-}\frac{2N^2+1}{6}, & & i = j,
	\cr
	& {-}\frac12(-1)^{i-j}\sin^{-2}\left(\frac{\pi(i-j)}{2N}\right), & & i \neq j;
\end{aligned}
\right.
\end{equation}
where we have taken $1\le i\le N$ and $1\le j\le 2N$ in both \eqref{e:Dxiij} and \eqref{e:Dxixiij}.

Let us denote $\xxii = (\xi_1, \ldots, \xi_N)^T$. Taking into account \eqref{e:uxuxx}, \eqref{e:Itilde} and \eqref{e:UDxi}, we define the following differentiation matrices:
\begin{align}
	\label{e:Dx}
	\Dx & = -\diag(\sin^2(\xxii))\cdot\D_\xi\cdot\tilde\I,
	\\
	\label{e:Dxx}
	\Dxx & = \diag(\sin^4(\xxii))\cdot\D_{\xi\xi}\cdot\tilde\I + \diag(\sin^2(\xxii)\sin(2\xxii))\cdot\D_\xi\cdot\tilde\I
	\cr
	& = \diag(\sin^4(\xxii))\cdot\D_{\xi\xi}\cdot\tilde\I - \diag(\sin(2\xxii))\cdot\Dx,
\end{align}
where $\diag(\cdot)$ denotes the diagonal matrix whose entries are those of the vector given as the argument. Although in this paper we are only interested in $\Dxx$, the computational cost of generating $\Dx$ or both $\Dx$ and $\Dxx$ is virtually the same, so we generate both, and use the second expression of $\Dxx$, that uses $\Dx$ explicitly. Then, bearing in mind that $u(x_j) \equiv U(\xi_j)$,
\begin{equation*}
	\begin{pmatrix}
		u_x(x_1) \\ \vdots \\ u_x(x_{N})
	\end{pmatrix}
	\approx
	\frac1L\Dx
	\begin{pmatrix}
		u(x_1) \\ \vdots \\ u(x_{N})
	\end{pmatrix},
	\qquad
	\begin{pmatrix}
		u_{xx}(x_1) \\ \vdots \\ u_{xx}(x_{N})
	\end{pmatrix}
	\approx
	\frac1{L^2}\Dxx
	\begin{pmatrix}
		u(x_1) \\ \vdots \\ u(x_{N})
	\end{pmatrix}.
\end{equation*}
Furthermore, from an implementational point of view, it is enough to generate the first $\lceil N/2\rceil$ rows of $\Dx$ and $\Dxx$. Indeed, from \eqref{e:Dxiij} and \eqref{e:Dx}, and from \eqref{e:Dxixiij} and \eqref{e:Dxx}, it follows respectively that
\begin{equation*}
[\Dx]_{N+1-i,N+1-j} = -[\Dx]_{i,j}, \quad 1 \le i, j\le N, \qquad [\Dxx]_{N+1-i,N+1-j} = [\Dxx]_{i,j}, \quad 1 \le i, j\le N,
\end{equation*}
and, likewise, $x_{N+1-j} = -x_j$, for $1 \le j\le N$. Therefore, we also only need the first $\lceil N/2\rceil$ rows of $\D_\xi$ and $\D_{\xi\xi}$, which are Toeplitz matrices; in fact, $\D_\xi$ and $\D_{\xi\xi}$ would be circulant matrices, should we consider $2N$ rows in their respective definitions \eqref{e:Dxiij} and \eqref{e:Dxixiij}. More precisely, their first $\lceil N/2\rceil$ rows have the structure of a matrix $\C\in\R^{\lceil N/2\rceil\times(2N)}$ defined as
\begin{equation*}
\C =
\left(
\begin{array}{cccc|ccccc}
	c_{2N+1} & c_{2N+2} & \ldots & c_{3N} & c_{N+1} & c_{N+2} & \ldots & c_{2N}
	\\
	c_{2N} & c_{2N+1} & \ldots & c_{3N-1} & c_{N} & c_{N+1} & \ldots & c_{2N-1}
	\\
	\vdots & \vdots & \ddots & \vdots & \vdots & \vdots & \ddots & \vdots
	\\
	c_{N+2+\lfloor N/2\rfloor} & c_{N+3+\lfloor N/2\rfloor} & \ldots & c_{2N+1+\lfloor N/2\rfloor} & c_{2+\lfloor N/2\rfloor} & c_{3+\lfloor N/2\rfloor} & \ldots & c_{N+1+\lfloor N/2\rfloor}
\end{array}
\right),
\end{equation*}
where $c_{2N+j} \equiv c_j$. The reason why we write the subindices in this way, i.e., $c_{2N+1}$ instead of $c_1$, etc., is because the generation of matrices of the type of $\C$, i.e., with subscripts ordered monotonically both row-wise and column-wise becomes extremely efficient in MATLAB. On the other hand, to further optimize the code, we compute directly $\D_\xi\cdot\tilde\I$ and $\D_{\xi\xi}\cdot\tilde\I$, which have the following structure:
\begin{equation}
\label{e:CtildeI}
\C \cdot\tilde\I =
	\begin{pmatrix}
		c_{2N+1} & c_{2N+2} & \ldots & c_{3N}
		\\
		c_{2N} & c_{2N+1} & \ldots & c_{3N-1}
		\\
		\vdots & \vdots & \ddots & \vdots
		\\
		c_{N+2+\lfloor N/2\rfloor} & c_{N+3+\lfloor N/2\rfloor} & \ldots & c_{2N+1+\lfloor N/2\rfloor}
	\end{pmatrix}
	+
\begin{pmatrix}
	c_{2N} & c_{2N-1} & \ldots & c_{N+1}
	\\
	c_{2N-1} & c_{2N-2} & \ldots & c_N
	\\
	\vdots & \vdots & \ddots & \vdots
	\\
	c_{N+1+\lfloor N/2\rfloor} & c_{N+\lfloor N/2\rfloor} & \ldots & c_{2+\lfloor N/2\rfloor}
\end{pmatrix}.
\end{equation}
Therefore, if we have a matrix having the shape of $\C$, we take its first row, which is formed by the elements $c_j$, for $1 \le j \le 2N$, extend it to $1 \le j \le 3N$, by defining $c_{2N+j} \equiv c_j$, then generate the indices corresponding to the two matrices on the right-hand side of \eqref{e:CtildeI}. In MATLAB, this is done by typing:
\begin{verbatim}
	index1=(2*N+1:3*N)-(0:ceil(N/2)-1)';
	index2=(2*N:-1:N+1)-(0:ceil(N/2)-1)';
\end{verbatim}
Then, if \verb|c| stores $c_j$, for $1\le j \le3N$, $\C\cdot\tilde\I$ is obtained by simply typing \verb|c(index1)+c(index2);|.  Furthermore, should we want an odd extension of $U(\xi)$ at $\xi=\pi$, i.e., $U(\pi+\xi) = -U(\pi - \xi)$, then, we would type \verb|c(index1)-c(index2);|. Finally, should a $\pi$-periodic extension of $U(\xi)$ be required, i.e., $U(\pi+\xi) = U(\xi)$, then, we would flip the columns of the second matrix on the right-hand side of \eqref{e:CtildeI}, which would imply defining \verb|index2=(N+1:2*N)-(0:ceil(N/2)-1)';|, followed again by \verb|c(index1)+c(index2);| In all cases, note that \verb|index1| and \verb|index2| can be used in the generation of the first $\lceil N/2\rceil$ rows of both $\D_\xi\cdot\tilde\I$ and $\D_{\xi\xi}\cdot\tilde\I$, which produces a faster code that, e.g., by using the MATLAB command \verb|toeplitz|.

Coming back to \eqref{e:Dxiij} and \eqref{e:Dxixiij}, a relevant fact is that, given an infinitesimal value $0 < \epsilon \ll 1$, MATLAB computes more accurate $-\cot(\epsilon)$ and $\sin^{-2}(\epsilon)$ than $\cot(\pi - \epsilon)$ and $\sin^{-2}(\pi - \epsilon)$. For instance, \verb|-cot(1e-9)| and \verb|cot(pi-1e-8)| return \verb|-9.999999999999999e+08| and \verb|-9.999997947949913e+08|, respectively, and \verb|sin(1e-9)^-2| and \verb|sin(pi-1e-9)^-2| return \verb|9.999999999999999e+17| and \verb|9.999995895900246e+17|, respectively. Hence, we use the formulas $\cot(\pi-\epsilon) = -\cot(\epsilon)$ and $\sin^{-2}(\pi-\epsilon) = \sin^{-2}(\epsilon)$, and operate with $\epsilon\in(0,\pi/2)$. Bearing in mind this, we define the vector $\mathbf c^{(1)} = (c_1^{(1)}, \ldots, c_{3N}^{(1)})$ corresponding to the first row of $\D_\xi$, by taking $i = 1$ in \eqref{e:Dxiij}, for $1\le j\le N$, and making $c_{2N+j}^{(1)} = c_j^{(1)}$, for $1 \le j \le N$, getting:
\begin{equation*}
c_j^{(1)} = 
\left\{
\begin{aligned}
	& 0, & & j \in \{1, N+1, 2N+1\},
	\cr
	& \frac12(-1)^j\cot\left(\frac{\pi(j-1)}{2N}\right), & & 2 \le j \le N,
	\cr
	& {-}\frac12(-1)^j\cot\left(\frac{\pi(2N - j + 1)}{2N}\right) = -c_{2N-j+2}^{(1)}, & & N+2 \le j \le 2N,
	\cr
	& \frac12(-1)^j\cot\left(\frac{\pi(j-1-2N)}{2N}\right) = c_{j - 2N}^{(1)}, & & 2N+2 \le j \le 3N.
\end{aligned}
\right.
\end{equation*}
Likewise, we define the vector $\mathbf c^{(2)} = (c_1^{(2)}, \ldots, c_{3N}^{(2)})$ corresponding to the first row of $\D_{\xi\xi}$, by taking $i = 1$ in \eqref{e:Dxixiij}, for $1\le j\le N$, and making $c_{2N+j}^{(2)} = c_j^{(2)}$, for $1 \le j \le N$, getting:
\begin{equation*}
	c_j ^{(2)} = 
	\left\{
	\begin{aligned}
		& {-}\frac{2N^2+1}{6}, & & j \in \{1, 2N+1\},
		\cr
		& \frac12(-1)^j\sin^{-2}\left(\frac{\pi(j-1)}{2N}\right), & & 2 \le j \le N,
		\cr
		& {-}\frac12(-1)^N, & & j = N + 1,
		\cr
		& \frac12(-1)^j\sin^{-2}\left(\frac{\pi(2N - j + 1)}{2N}\right) = c_{2N-j+2}^{(2)}, & & N+2 \le j \le 2N,
		\cr
		& \frac12(-1)^j\sin^{-2}\left(\frac{\pi(j-1-2N)}{2N}\right) = c_{j-2N}^{(2)}, & & 2N+2 \le j \le 3N.
	\end{aligned}
	\right.
\end{equation*}
$\mathbf c^{(1)}$ and $\mathbf c^{(2)}$ enable us to create $\D_\xi$ and $\D_{\xi\xi}$, respectively, from which, after applying \eqref{e:Dx} and \eqref{e:Dxx}, $\Dx$ and $\Dxx$ follow. Bearing in mind all the previous arguments, we offer in Listing~\ref{code:createDxDxx} the MATLAB function \verb|create_Dx_Dxx|, that generates $\Dx$ and $\Dxx$, as well as a vector $\mathbf x$ containing the nodes $\{x_j\}$; note that $\Dxx$ is returned first, because it is the only matrix that we will need in the rest of the paper. Note that, in MATLAB, if \verb|v| is a column vector of compatible side, \verb|diag(v)*A| and \verb|v.*A| produce the same result in MATLAB, but the latter is more efficient, so we do not use explicitly \verb|diag|.

\begin{lstlisting}[label=code:createDxDxx, language=MATLAB, basicstyle=\footnotesize, caption = {Function \texttt{create\_Dx\_Dxx}, to generate the matrices $\Dxx$, $\Dx$ and the nodes $\{x_j\}$}]
% Function that returns $\mathbf D_{xx}$, $\mathbf D_x$ and $\mathbf x$
function [Dxx,Dx,x]=create_Dx_Dxx(N,L)
xi=pi*((1:ceil(N/2))-1/2).'/N; % $\boldsymbol{\xi}$
x=cot(xi);
index1=(2*N+1:3*N)-(0:ceil(N/2)-1)';
index2=(2*N:-1:N+1)-(0:ceil(N/2)-1)';
aux=(-1).^(2:N).*cot(pi*(1:N-1)/(2*N))/2;
c1=[0,aux,0,-aux(N-1:-1:1),0,aux]; % $\mathbf c^{(1)}$
Dx=-sin(xi).^2.*(c1(index1)+c1(index2));
aux=-(-1).^(1:N-1)./sin(pi*(1:N-1)/(2*N)).^2/2;
c2=[-(2*N^2+1)/6,aux,-(-1)^N/2,aux(N-1:-1:1),-(2*N^2+1)/6,aux]; % $\mathbf c^{(2)}$
Dxx=sin(xi).^4.*(c2(index1)+c2(index2))-sin(2*xi).*Dx;
x=L*[x;-x(floor(N/2):-1:1)];
Dx=[Dx;-Dx(floor(N/2):-1:1,N:-1:1)]/L;
Dxx=[Dxx;Dxx(floor(N/2):-1:1,N:-1:1)]/L^2;
\end{lstlisting}

With the method thus developed, we find particularly remarkable that, besides the stability of the whole process, very high accuracies are still possible even for very large values of $N$. To illustrate this, we offer in Listing~\ref{code:testDxDxx} a simple numerical test that approximates the first and second derivatives of $u(x)=1-\exp(-x^2)$ and $u(x)=\arctan(x)$, and returns the discrete $\ell^\infty$ of the errors. Observe that these functions have been chosen, because, they do not tend to zero at infinity, and have different asymptotic properties. More precisely, $u(x)=1-\exp(-x^2)$ tends to $1$ at a Gaussian rate, as $x\to\pm\infty$, whereas $u(x)=\arctan(x)$ tends quadratically (i.e., algebraically) fast to $\pm\pi/2$, as $x\to\pm\infty$. Therefore, an even extension suits well in both cases. We have taken $N=20000$, and a value of $L$ that yields good results for both functions, i.e., $L=400$. Then, the error in the numerical approximation of the first and second derivatives are respectively of the order of $\mathcal O(10^{-12})$ and $\mathcal O(10^{-11})$. Let us also remark that, even if \verb|Dx| and \verb|Dxx| require each $3.2\times10^9$ bytes of storage, i.e. over 3 Gb, the elapsed time was approximately of 13 seconds.

\begin{lstlisting}[label=code:testDxDxx, language=MATLAB, basicstyle=\footnotesize, caption = {Program \texttt{test\_Dx\_Dxx.m}, that test $\Dx$ and $\Dxx$}]
clear
tic
N=20000; % Number of nodes $N$
L=400; % Scale factor $L$
[Dxx,Dx,x]=create_Dx_Dxx(N);
Dx=Dx/L; % Scale $\mathbf D_x$
Dxx=Dxx/L^2; % Scale $\mathbf D_{xx}$
x=L*x; % Scale $\mathbf x$
norm(Dx*(1-exp(-x.^2))-(2*x.*exp(-x.^2)),inf) % 1.808997396324230e-12
norm(Dxx*(1-exp(-x.^2))-((2-4*x.^2).*exp(-x.^2)),inf) % 7.556577585887680e-11
norm(Dx*atan(x)-1./(1+x.^2),inf) % 1.758038159493935e-12
norm(Dxx*atan(x)-(-2*x./(1+x.^2).^2),inf) % 6.811984309962327e-11
toc
\end{lstlisting}

\section{Numerical approximation of the fractional Laplacian in $n$ dimensions}

\label{s:fraclap}

After creating a second-order differentiation matrix $\Dxx$, we will use it to approximate numerically the fractional Laplacian by using \eqref{e:balakrishnan}. A crucial point here is that $\Dxx$ is diagonalizable, i.e., it can be factorized as $\Dxx = \Px\cdot\Lx\cdot\Px^{-1}$ (note that we use $\Lx$ rather than $\D$ to denote the diagonal eigenvalue matrix, in order to avoid any confusions with differentiation matrices), and its eigenvector matrix $\PP$ is well-conditioned; this does not happen, e.g., in the Chebyshev differentiation matrices defined on $[-1,1]$ (see \cite[Chapter 6]{trefethen2000}). To illustrate this, we have diagonalized $\Dxx$ for $100 \le N \le 5000$, by typing \verb|[P,~]=eig(Dxx);| in MATLAB, and have computed the condition number $\kappa(\PP)$ associated to the Euclidean norm of $\PP$ by typing \verb|cond(P)|. A least squares regression analysis reveals that $\log_{10}(\kappa(\PP)) \approx 0.7478\log_{10}(N) - 0.1149$, being Pearson's correlation coefficient $r = 1 - 5.0022\times10^{-7}$. Therefore, $\kappa(\PP)$ grows sublinearly with respect to $N$, i.e., $\kappa(\PP)\approx 0.7676N^{0.7478}$, and, indeed, when $N = 5000$, $\kappa(\PP) = 448.75$. Thanks to this fact, it is numerically safe to work with the diagonalization of $\Dxx$. Observe also that, for a given $N$, it is immediate to check that the corresponding $\Dxx$ has $N - 1$ negative eigenvalues, and one eigenvalue that would be exactly zero, should exact arithmetic be used, but in practice, due to rounding errors, it takes an infinitesimally small value, that can be positive or negative. For instance, when $N = 5000$, the theoretically zero eigenvalue is $8.4098\times10^{-10}$ (for the sake of comparison, note that the eigenvalue with largest modulus is $-2.4826\times10^7$). Since the fact that the eigenvalue with the smallest absolute value not being strictly equal to zero can occasionally have a negative impact on the accuracy, we have systematically rounded it to zero, and have replaced the corresponding column in $\PP$ by a constant vector with entries equal to $N^{-1/2}$, in coherence with the result of \verb|eig|, which returns $\PP$ in such a way that its columns have unitary Euclidean length.

In what follows, we consider first the two-dimensional case, with $n = 2$, which is immediately generalizable to higher dimensions, as well as to the one-dimensional case, with $n = 1$. More precisely, we prove the following theorem.

\begin{theorem}
	
\label{t:fraclap2D}
	
Let $s\in(0,1)$. Let $\vx = (x, y)\in\R^2$. Let $N_x\in\mathbb N$ and $N_y\in\mathbb N$ denote respectively the number of nodes $x_i$, with $1 \le i\le N_x$, and $y_j$ corresponding to the $x$-axis and $y$-axis, and $\x = (x_1, \ldots, x_{N_x})^T$ and $\y = (y_1, \ldots, y_{N_y})^T$ the vectors whose entries are these nodes. Let $\Dxx\in\R^{N_x\times N_x}$ and $\Dyy\in\R^{N_y\times N_y}$ be diagonalizable and negative semidefinite differentiation matrices associated respectively to $\x$ and $\y$ that approximate numerically $\partial_{xx}$ and $\partial_{yy}$. Let $\Dxx = \Px\cdot\Lx\cdot\Px^{-1}$ and $\Dyy = \Py\cdot\Ly\cdot\Py^{-1}$, where $\Px\in\R^{N_x\times N_x}$ and $\Py\in\R^{N_y\times N_y}$ are the eigenvector matrices, and $\Lx = \diag(\lx)\in\R^{N_x\times N_x}$ and $\Ly = \diag(\ly)\in\R^{N_y\times N_y}$ are the diagonal eigenvalue matrices, with $\lx = (({\lambda_{x,1}, \ldots, \lambda_{x,N_x}}))^T$ and $\ly = (({\lambda_{y,1}, \ldots, \lambda_{y,N_y}}))^T$ being the vectors whose entries are the eigenvalues $\lambda_{x,i} \le 0$ and $\lambda_{y,j} \le 0$ of $\Dxx$ and $\Dyy$, respectively. Let $\U\in\CC^{N_x\times N_y}$ be the matrix such that $[\U]_{ij} = u(x_i, y_j)$. Then, the numerical approximation of both \eqref{e:balakrishnan} and \eqref{e:bochner} for $n = 2$ is given by
\begin{equation}
	\label{e:approxfraclap2D}
	(-\Delta)^s u(\x,\y) \approx \Px\cdot\left[(-\Lxy)^{\odot s}\odot\left[\Px^{-1}\cdot\U\cdot(\Py^T)^{-1}\right]\right]\cdot\Py^T,
\end{equation}
where, with some abuse of notation, $[(-\Delta)^s u(\x,\y)]_{ij} = (-\Delta)^s u(x_i,y_j)$; $\Lxy= \lx\cdot\Jy^T + \Jx\cdot\ly^T\in\R^{N_x\times N_y}$, with $\Jx\in\R^{N_x}$ and $\Jy\in\R^{N_y}$ being the all-one column vectors of length $N_x$ and $N_y$ respectively, i.e., $[\Lxy]_{ij} = \lambda_{x,i} + \lambda_{y,j}$; and $(\cdot)^{\odot s}$ and $\odot$ denote respectively the Hadamard or entry-wise power of order $s$ and product, respectively.

\end{theorem}

\begin{proof}

Let us approximate first \eqref{e:balakrishnan}, for $n = 2$. In order to do it, we have to compute $(-\Delta)(t - \Delta)^{-1}[u](x,y)$, for a given function $u(x,y)$. Recall that, given $t > 0$, $v(x,y) = (t - \Delta)^{-1}[u](x,y)$ must be understood as the solution of
\begin{equation}
\label{e:tdeltav}
(t - \Delta)v(x,y) = u(x,y).
\end{equation}
We use $\Dxx$ and $\Dyy$ to discretize $\Delta$ in \eqref{e:tdeltav}. More precisely, after denoting by $\U \approx u(\x,\y)$  and $\V \approx v(\x,\y)$ with some abuse of notation the matrices such that $U_{ij} = u(x_i, y_j)$ and $V_{ij}\approx v(x_i, y_j)$, \eqref{e:tdeltav} becomes
\begin{equation}
\label{e:VU}
t\V - \Dxx\cdot\V - \V\cdot\Dyy^T = \U,
\end{equation}
where $\Dxx\cdot\V \approx \partial_{xx}v(\x,\y)$ and $\V\cdot\Dyy^T \approx \partial_{yy}v(\x,\y)$, i.e., $[\Dxx\cdot\V]_{ij}\approx\partial_{xx}v(x_i,y_j)$ and $[\V\cdot\Dyy^T]_{ij}\approx\partial_{yy}v(x_i,y_j)$; note that $\Dyy$ appears transposed, because it is applied to the rows of $\U$. Then, introducing the factorizations of $\Dxx$ and $\Dyy$ into \eqref{e:VU},
\begin{equation*}
t\V -  \Px\cdot\Lx\cdot\Px^{-1}\cdot\V - \V\cdot(\Py\cdot\Ly\cdot\Py^{-1})^T = t\V -  \Px\cdot\Lx\cdot\Px^{-1}\cdot\V - \V\cdot(\Py^T)^{-1}\cdot\Ly\cdot\Py^T =  \U,
\end{equation*}
or, equivalently,
\begin{equation*}
t\Px^{-1}\cdot\V\cdot(\Py^T)^{-1} -  \Lx\cdot\Px^{-1}\cdot\V\cdot(\Py^T)^{-1} - \Px^{-1}\cdot\V\cdot(\Py^T)^{-1}\cdot\Ly =  \Px^{-1}\cdot\U\cdot(\Py^T)^{-1}.
\end{equation*}
Defining $\tilU = \Px^{-1}\cdot\U\cdot(\Py^T)^{-1}$ and $\tilV = \Px^{-1}\cdot\V\cdot(\Py^T)^{-1}$,
\begin{equation*}
t\tilV -  \Lx\cdot\tilV - \tilV\cdot\Ly =  (t\Jxy - \Lxy)\odot\tilV = \tilU,
\end{equation*}
where $\Jxy\in\R^{N_x\times N_y}$ is the all-one matrix of size $N_x\times N_y$, $\odot$ denotes the Hadamard or entry-wise product, and $[\Lxy]_{ij} = \lambda_{x,i} + \lambda_{y,j}$, i.e., $\Lxy= \lx\cdot\Jy^T + \Jx\cdot\ly^T$, where $\Jx\in\R^{N_x}$ and $\Jy\in\R^{N_y}$ are the all-one column vectors of length $N_x$ and $N_y$, respectively. Therefore,
\begin{equation*}
\tilV = \tilU \oslash (t\Jxy - \Lxy) \Longleftrightarrow \V = \Px\cdot\left[\tilU \oslash (t\Jxy - \Lxy)\right]\cdot\Py^T,
\end{equation*}
where $\oslash$ denotes the Hadamard or entry-wise division. Hence,
\begin{equation*}
\V = \Px\cdot\left[\tilU \oslash (t\Jxy - \Lxy)\right]\cdot\Py^T,
\end{equation*}
which is the numerical approximation of $v(x,y) = (t - \Delta)^{-1}[u](x,y)$. Note that, since $\lambda_{x,i} \le 0$ and $\lambda_{y,j} \le 0$, $\tilV$, and hence $\V$, are uniquely defined for $t > 0$. In order to approximate numerically $(-\Delta)(t - \Delta)^{-1}[u](x,y) = (-\Delta)v(x,y)$, we discretize again $\Delta$ by $\Dxx$ and $\Dyy$, and apply them to $\V$:
\begin{align}
\label{e:V}
(-\Delta)v(\x,\y) & \approx -\Dxx\cdot\V - \V\cdot\Dyy^T = -\Px\cdot\Lx\cdot\Px^{-1}\cdot\V - \V\cdot(\Py^T)^{-1}\cdot\Ly\cdot\Py^T
	\cr
& = -\Px\cdot\Lx\cdot\Px^{-1}\cdot\V - \V\cdot(\Py^T)^{-1}\cdot\Ly\cdot\Py^T
	\cr
& = -\Px\cdot\Lx\cdot\left[\tilU \oslash (t\Jxy - \Lxy)\right]\cdot\Py^T - \Px\cdot\left[\tilU \oslash (t\Jxy - \Lxy)\right]\cdot\Ly\cdot\Py^T
	\cr
& = -\Px\cdot\left[\Lx\cdot\left[\tilU \oslash (t\Jxy - \Lxy)\right] + \left[\tilU \oslash (t\Jxy - \Lxy)\right]\cdot\Ly\right]\cdot\Py^T
	\cr
& = -\Px\cdot\left[\left[\Lxy\oslash (t\Jxy - \Lxy)\right]\odot\tilU\right]\cdot\Py^T,
\end{align}
where $[(-\Delta)v(\x,\y)]_{ij} = (-\Delta)v(x_i,y_j)$, and
\begin{equation}
\label{e:LLij}
\left[\Lxy\oslash (t\Jxy - \Lxy)\right]_{ij} = \frac{\lambda_{x,i} + \lambda_{y,j}}{t - \lambda_{x,i} - \lambda_{y,j}}.
\end{equation}
Introducing \eqref{e:V} into \eqref{e:balakrishnan}:
\begin{align}
\label{e:int0infmatrix0}
(-\Delta)^s u(\x,\y) & \approx \frac{\sin(\pi s)}{\pi}\int_0^\infty\left[-\Px\cdot\left[\left[\Lxy\oslash (t\Jxy - \Lxy)\right]\odot\tilU\right]\cdot\Py^T\right]\frac{dt}{t^{1-s}}
	\cr
& = -\frac{\sin(\pi s)}{\pi}\Px\cdot\left[\left(\int_0^\infty[\Lxy\oslash (t\Jxy - \Lxy)]\frac{dt}{t^{1-s}}\right)\odot\tilU\right]\cdot\Py^T,
\end{align}
where $[(-\Delta)^s u(\x,\y)]_{ij} = (-\Delta)^s u(x_i,y_j)$. Then, bearing in mind \eqref{e:LLij}, we have
\begin{equation}
\label{e:int0infmatrix}
\left[\int_0^\infty[\Lxy\oslash (t\Jxy - \Lxy)]\frac{dt}{t^{1-s}}\right]_{ij} = \int_0^\infty\left(\frac{\lambda_{x,i} + \lambda_{y,j}}{t - \lambda_{x,i} - \lambda_{y,j}}\right)\frac{dt}{t^{1-s}} = -\pi\csc(\pi s)(-\lambda_{x,i} - \lambda_{y,j})^s,
\end{equation}
where we have used Lemma~\ref{lem:I1} in Appendix~\ref{a:aux}. Therefore, \eqref{e:int0infmatrix0} becomes
\begin{align*}
(-\Delta)^s u(\x,\y) & \approx -\frac{\sin(\pi s)}{\pi}\Px\cdot\left(\left[-\pi\csc(\pi s)(-\Lxy)^{\odot s}\right]\odot\tilU\right)\cdot\Py^T
	\cr
& = \Px\cdot\left(\left[(-\Lxy)^{\odot s}\right]\odot\tilU\right)\cdot\Py^T,
\end{align*}
where $(-\Lxy)^{\odot s}$ denotes the Hadamard or entry-wise power of order $s$ of $-\Lxy$. Expanding $\tilU$, we get \eqref{e:approxfraclap2D}.

Let us approximate now the Bochner formula \eqref{e:bochner}. Reasoning as above,
\begin{align}
\label{e:etuu}
e^{t\Delta}[u](\x,\y) - u(\x,\y) & \approx e^{t\Dxx}\cdot\U\cdot e^{t\Dyy^T} - \U = \Px\cdot\left[e^{\odot t\Lxy}\odot\tilU - \tilU\right]\cdot\Py^T
	\cr
& = \Px\cdot\left[\left(e^{\odot t\Lxy} - \Jxy\right)\odot\tilU\right]\cdot\Py^T,
\end{align}
where $[e^{t\Delta}[u](\x,\y)]_{ij} \approx e^{t\Delta}[u](x_i,y_j)$, $e^{t\Dxx}$ and $e^{t\Dyy^T}$ are respectively the exponential matrices of $t\Dxx$ and $t\Dyy^T$, and $e^{\odot t\Lxy}$ is the Hadamard or entry-wise exponential of $t\Lxy$; hence,
\begin{equation}
	\label{e:eLLij}
	\left[e^{\odot t\Lxy} - \Jxy\right]_{ij} = e^{t(\lambda_{x,i} + \lambda_{y,j})} - 1.
\end{equation}
Introducing \eqref{e:etuu} into \eqref{e:bochner},
\begin{align}
\label{e:int0infmatrix1}
(-\Delta)^s u(\x,\y) & \approx \frac{1}{\Gamma(-s)}\int_0^\infty\left[\Px\cdot\left[\left(e^{\odot t\Lxy} - \Jxy\right)\odot\tilU\right]\cdot\Py^T\right]\frac{dt}{t^{1+s}}
	\cr
& = \frac{1}{\Gamma(-s)}\Px\cdot\left[\left(\int_0^\infty\left[e^{\odot t\Lxy} - \Jxy\right]\frac{dt}{t^{1+s}}\right)\odot\tilU\right]\cdot\Py^T.
\end{align}
Then, bearing in mind \eqref{e:eLLij},
\begin{equation*}
\left[\int_0^\infty\left[e^{\odot t\Lxy} - \Jxy\right]\frac{dt}{t^{1+s}}\right]_{ij} = \int_0^\infty\frac{e^{t(\lambda_{x,i} + \lambda_{y,j})} - 1}{t^{1+s}}dt = \Gamma(-s)(-\lambda_{x,i} - \lambda_{y,j})^s,
\end{equation*}
where we have used Lemma~\ref{lem:I2} in Appendix~\ref{a:aux}. Therefore, \eqref{e:int0infmatrix1} becomes
\begin{equation*}
(-\Delta)^s u(\x,\y) \approx \frac{1}{\Gamma(-s)}\Px\cdot\left[\Gamma(-s)(-\Lxy)^{\odot s}\odot\tilU\right]\cdot\Py^T = \Px\cdot\left[(-\Lxy)^{\odot s}\odot\tilU\right]\cdot\Py^T.
\end{equation*}
Expanding $\tilU$, we get again \eqref{e:approxfraclap2D}.

\end{proof}

\begin{remark} The quality of the approximation \eqref{e:approxfraclap2D} depends on the quality of the approximation of the partial derivatives by means of $\Dxx$ and $\Dyy$. In this regard, the differentiation matrices developed in Section~\ref{s:matrices} are particularly well-suited, because they can be efficiently generated, provide spectral accuracy, and the eigenvector matrices appearing in their diagonalization are well conditioned; however, other types of differentiation matrices might be possible, provided that they satisfy these properties. On the other hand, \eqref{e:approxfraclap2D} holds if $\Dxx$ and $\Dyy$ are scaled, i.e., if the actual differentiation matrices are $(1/L_x^2)\Dxx$ and $(1/L_y^2)\Dyy$, for $L_x > 0$ and $L_y > 0$. More precisely, if $\Dxx = \Px\cdot\diag(\lx)\cdot\Px^{-1}$ and $\Dyy = \Py\cdot\diag(\ly)\cdot\Py^{-1}$, then  $(1/L_x^2)\Dxx = \Px\cdot\diag((1/L_x^2)\lx)\cdot\Px^{-1}$ and $(1/L_y^2)\Dyy = \Py\cdot\diag((1/L_y^2)\ly)\cdot\Py^{-1}$, so it is enough to diagonalize $\Dxx$ and $\Dyy$, regardless of the values of the scaling factors $L_x$ and $L_y$. Then,  $\Lxy= (1/L_x^2)\lx\cdot\Jy^T + (1/L_y^2)\Jx\cdot\ly^T$, i.e., $[\Lxy]_{ij} = \lambda_{x,i}/L_x^2 + \lambda_{y,j}/L_y^2$. Therefore, in the particular case where the scaling factors are the same for $x$ and $y$, i.e., $L = L_x = L_y$,  \eqref{e:approxfraclap2D} becomes
\begin{equation*}
(-\Delta)^s u(\x,\y) \approx \frac{1}{L^{2s}}\Px\cdot\left[(-\Lxy)^{\odot s}\odot\left[\Px^{-1}\cdot\U\cdot(\Py^T)^{-1}\right]\right]\cdot\Py^T,
\end{equation*}
and the scaling does not intervene in the definition of $\Lxy= \lx\cdot\Jy^T + \Jx\cdot\ly^T$.

\end{remark}

It is straightforward to generalize Theorem~\ref{t:fraclap2D} to $n$ dimensions. In fact, the generalization to $n = 1$, which is particularly simple, and to $n\in\mathbb N$, can be stated as corollaries of this theorem. Let us state first the $n = 1$ case.

\begin{corollary}
	
\label{coro:1}
	
Let $s\in(0,1)$. Let $N_x \in\mathbb N$ denote the number of nodes $x_i$, and $\x = (x_1, \ldots, x_{N_x})^T$ the vector whose entries are these nodes. Let $\Dxx\in\R^{N_x\times N_x}$ be a diagonalizable and negative semidefinite differentiation matrix associated to $\x$ that approximates numerically $d^2/dx^2$. Let $\Dxx = \Px\cdot\Lx\cdot\Px^{-1}$, where $\Px\in\R^{N_x\times N_x}$ is the eigenvector matrix, and $\Lx = \diag(\lx)\in\R^{N_x\times N_x}$ is the diagonal eigenvalue matrix, with $\lx = (({\lambda_{x,1}, \ldots, \lambda_{x,N_x}}))^T$ being the vector whose entries are the eigenvalues $\lambda_{x,i} \le 0$ of $\Dxx$. Let $\U = (u(x_1), \ldots, u(x_{N_x}))^T\in\CC^{N_x}$. Then, the numerical approximation of both \eqref{e:balakrishnan} and \eqref{e:bochner} for $n = 1$ is given by
\begin{equation}
	\label{e:approxfraclap1D}
	(-\Delta)^s u(\x) \approx \Px\cdot\left[(-\boldsymbol\lambda_x)^{\odot s}\odot\left[\Px^{-1}\cdot\U\right]\right],
\end{equation}
where $(-\Delta)^su(\x) = ((-\Delta)^su(x_1), \ldots, (-\Delta)^su(x_{N_x}))^T$; and $(\cdot)^{\odot s}$ and $\odot$ denote respectively the Hadamard or entry-wise power of order $s$ and product, respectively.
\end{corollary}

\begin{proof}
Since the variable $y$ does not intervene, the proof follows trivially from the statement of Theorem~\ref{t:fraclap2D}, by taking $N_y = 1$, and defining the one-element matrices $\Dyy = (0)$, $\Py = (1)$ and $\Ly = (0)$, and one-element vectors $\y = (0)$ and $\ly = (0)$.
\end{proof}

\begin{remark}
Since $\Lx = \diag(\boldsymbol\lambda_x)$, it is equivalent to write
\begin{equation*}
	(-\Delta)^s u(\x) \approx \Px\cdot(-\Lx)^{\odot s}\cdot\Px^{-1}\cdot\U,
\end{equation*}
although \eqref{e:approxfraclap1D} is faster from an implementational point of view.
\end{remark}

On the other hand, in order to generalize \eqref{e:approxfraclap2D} to $n\in\mathbb N$ dimensions, we need to adjust the notation. More precisely, $\vx = (x_1, \ldots, x_n)$, and $\Dxj$, for $1 \le j\le n$, is the differentiation matrix used to approximate $\partial_{x_jx_j}$, which can be diagonalized as $\Dxj = \PP_j\cdot\Lj\cdot\PP_j^{-1}$, where $\PP_j$ is the eigenvector matrix, and $\Lj= \diag(\lj)$ is the diagonal eigenvalue matrix, with $\lj = ({\lambda_{j,1}, \ldots, \lambda_{j,N_j}})^T$ the vector containing the eigenvalue of $\Dxj$. Each coordinate $x_j$ is discretized in $N_j$ nodes that we store in a vector $\x_j = (x_{j,1}, \ldots, x_{j,N_j})^T$. Then, with some abuse of notation, $\U \approx u(\x_1, \ldots, \x_n)$, where $\U\in\CC^{N_1\times\ldots\times N_n}$, and $[\U]_{i_1, \ldots, i_n} \approx u(x_{1,i_1}, \ldots, x_{n,i_n})$, for $1 \le i_j \le N_j$.

In order to approximate the second-order partial derivatives of $u(x_1, \ldots, x_n)$, we observe that, in the two-dimensional case,
\begin{align}
\label{e:dxx}
\partial_{xx}u(x_i, y_j) & = [\Dxx\cdot\U]_{ij} = \sum_{k = 1}^{N_x}[\Dxx]_{ik}[\U]_{kj},
	\\
\label{e:dyy}
\partial_{yy}u(x_i, y_j) & = [\U\cdot\Dyy^T]_{ij} = \sum_{k = 1}^{N_y}[\U]_{ik}[\Dyy^T]_{kj} = \sum_{k = 1}^{N_y}[\Dyy]_{jk}[\U]_{ik},
\end{align}
i.e., in $\Dxx\cdot\U$, we perform the sum along the first dimension of $\U$, whereas in $\U\cdot\Dyy^T$, we perform the sum along the second dimension of $\U$. In order to generalize \eqref{e:dxx} and \eqref{e:dyy}, we introduce the operator $\square_j$ that denotes the product between a matrix $\mathbf A\in\CC^{M\times N_j}$ and an $n$-dimensional array $\U\in\CC^{N_1\times\cdots\times N_n}$ along the $j$th dimension, i.e., $\mathbf A\square_j\U\in\CC^{N_1\times \ldots \times N_{j-1}\times M\times N_{j+1}\times \ldots N_n}$:
\begin{equation}
\label{e:squarej}
[\mathbf A\square_j\mathbf U ]_{i_1\ldots i_n} = \sum_{k = 1}^{N_j}[\mathbf A]_{i_jk}\cdot[\U]_{i_1\ldots i_{j-1}k i_{j+1}\ldots i_n}.
\end{equation}
Therefore, given $\U\in \CC^{N_1\times\ldots\times N_n}$, $\square_j$ enables us to generalize \eqref{e:dxx} and \eqref{e:dyy} to the $n$-dimensional case:
\begin{equation*} 
\partial_{x_j, x_j}u(x_{i_1}, \ldots, x_{i_n}) = [\Dxj\square_j\mathbf U ]_{i_1\ldots i_n} = \sum_{k = 1}^{N_j}[\Dxj]_{i_jk}\cdot[\U]_{i_1\ldots i_{j-1}k i_{j+1}\ldots i_N}, \quad 1 \le j \le n.
\end{equation*} 
Note that, with this notation, in the two-dimensional case, we have the equivalences $\Dxx\square_1\U\equiv \Dxx\cdot\U$ and $\Dyy\square_2\U\equiv \U\cdot\Dyy^T$.

Bearing in mind the previous arguments, we can adapt Theorem~\ref{t:fraclap2D} to the $n$-dimensional case by means of the following corollary.

\begin{corollary}
\label{coro:n}
Let $s\in(0,1)$. Let $\vx = (x_1, \ldots, x_n)\in\R^n$, with $n\in\mathbb N$. Let $N_j\in\mathbb N$, with $1\le j\le n$, denote the number of nodes corresponding to the $x_j$-axis, and $\x_j= (x_{j,1}, \ldots, x_{j,N_j})^T$ the vector whose entries are these nodes. Let $\Dxj\in\R^{N_j\times N_j}$ be a diagonalizable and negative semidefinite differentiation matrix associated to $\x_j$ that approximates numerically $\partial_{x_jx_j}$. Let $\Dxj = \PP_j\cdot\Lj\cdot\PP_j^{-1}$, where $\PP_j\in\R^{N_j\times N_j}$ is the eigenvector matrix, and $\Lj = \diag(\lj)\in\R^{N_j\times N_j}$ is the diagonal eigenvalue matrix, with $\lj = ({\lambda_{j,1}, \ldots, \lambda_{j,N_j}})^T$ being the vector whose entries are the eigenvalues $\lambda_{j,i} \le 0$. Let $\U\in\CC^{N_1\times \ldots \times N_n}$ be the $n$-dimensional array such that $[\U]_{i_1\ldots i_n} = u(x_{1,i_1}, \ldots, x_{n,i_n})$, for $1\le i_j\le N_j$. Then, the numerical approximation of both \eqref{e:balakrishnan} and \eqref{e:bochner} for $n\in\mathbb N$ is given by
\begin{equation}
	\label{e:approxfraclapnD}
	(-\Delta)^s u(\x_1,\ldots,\x_n) \approx \PP_n\square_n[\ldots\square_2[\PP_1\square_1[(-\LL)^{\odot s}\odot\tilde\U]]],
\end{equation}
where
\begin{equation*}
\tilde\U = \PP_n^{-1}\square_n[\ldots\square_2[\PP_1^{-1}\square_1\U]].
\end{equation*}
Moreover, $[(-\Delta)^s u(\x_1, \ldots, \x_n)]_{i_1\ldots i_n} = (-\Delta)^s u(x_{1,i_1},\ldots, x_{n,i_n})$, for $1\le i_j\le N_j$; $\LL\in\R^{N_1\times\ldots\times N_n}$ is the $n$-dimensional array such that
$[\LL]_{i_1\ldots i_n} = \lambda_{1,i_1} + \ldots + \lambda_{n,i_n}$; $\square_j$ denotes the sum along the $j$th dimension, as defined in \eqref{e:squarej}, and $(\cdot)^{\odot s}$ and $\odot$ denote respectively the Hadamard or entry-wise power of $s$ and product, respectively.

\end{corollary}

\begin{proof} Identical to that of Theorem~\ref{t:fraclap2D}.
	
\end{proof}

\begin{remark}

In Corollary~\ref{coro:1} and Corollary~\ref{coro:n}, we have considered differentiation matrices without scaling, but it is straightforward to work with $\Dxj / L_j^2$, as explained in the remark after Theorem~\ref{t:fraclap2D}, for $n = 2$. Recall that $\Dxj = \PP_j\cdot\diag((1/L_j^2)\lj)\cdot\PP_j^{-1}$, so it is enough to diagonalize each differentiation matrix $\Dxj$ once, and update $\LL$ accordingly. Moreover, if all the matrices are scaled by the same factor $L^2$, then it is enough to divide the numerical approximation of $(-\Delta)^su(\vec{x})$ by $L^{2s}$.

\end{remark}

\subsection{Implementation in MATLAB}

In what regards the numerical implementation in MATLAB, we generate the differentiation matrices by means of the function \verb|create_Dx_Dxx| shown in Listing~\ref{code:createDxDxx}, and diagonalize them by means of the function \verb|eig|. With respect to $\square_j$, MATLAB does not implement natively the product of a matrix by a multidimensional array, so we use a minor generalization of the function \verb|multND| developed in \cite{cuestadelahoz2024}, that we offer in Listing~\ref{code:multND}, and which allows the multiplication of a vector or a matrix (not necessarily square) by an $n$-dimensional array along the $j$th dimension of the latter.

\begin{lstlisting}[label=code:multND, language=MATLAB, basicstyle=\footnotesize, caption = {Function \texttt{multND}, to multiply a matrix (not necessarily square) by a multidimensional array along a given dimension.}]
% compute $\mathbf A\square_j\mathbf X$ in MATLAB R2020b or newer
function X=multND(A,X,j)
if isscalar(A) % test whether A is a scalar
    X=A*X; % trivial case
    return
end
sizeA=size(A); % size of A
if length(sizeA) ~= 2, error('A must be a vector or a matrix!'); end
sizeX=size(X); % size of X
if sizeA(2) ~= sizeX(j), error('The dimensions of A and X do not match!'); end    
X=reshape(X,prod(sizeX(1:j-1)),sizeX(j),prod(sizeX(j+1:end))); % reshape X
X=pagemtimes(X,A.'); % compute the product
X=reshape(X,[sizeX(1:j-1),sizeA(1),sizeX(j+1:end)]); % reshape the product
\end{lstlisting}

Note that \verb|multND|  makes use of the function \verb|pagemtimes|, which was introduced in MATLAB R2020b. Therefore, in Listing~\ref{code:multNDalt}, we also offer a minor generalization of the function \verb|multNDalt| in \cite{cuestadelahoz2024}, which is slightly less efficient, but does not require \verb|pagemtimes|, and, hence, it can be executed in older versions of MATLAB.

\begin{lstlisting}[label=code:multNDalt, language=MATLAB, basicstyle=\footnotesize, caption = {Function \texttt{multNDalt}, to multiply a matrix (not necessarily square) by a multidimensional array along a given dimension. It works in older MATLAB versions.}]
% compute $\mathbf A\square_j\mathbf X$ in older MATLAB versions
function X=multNDalt(A,X,j)
if isscalar(A) % test whether A is a scalar
    X=A*X; % trivial case
    return
end
sizeA=size(A); % size of A
if length(sizeA) ~= 2, error('A must be a vector or a matrix!'); end
sizeX=size(X); % size of X
if sizeA(2) ~= sizeX(j), error('The dimensions of A and X do not match!'); end  
if j==1 % multiplication along the 1st dimension is simpler
    X=reshape(X,sizeX(j),[]); % reshape X
    X=A*X; % compute the product
    sizeX(1) = sizeA(1); % size of X, after multiplication
    X=reshape(X,sizeX); % reshape the product
else
    permutation=[j 2:j-1 1 j+1:length(sizeX)]; % permutation of the dimensions
    X=permute(X,permutation); % swap the first and jth dimensions
    sizeXpermuted=size(X); % size of X, after permutation
    sizeXpermuted(1)=sizeA(1); % size of X, after multiplication
    X=reshape(X,sizeX(j),[]); % reshape X
    X=A*X; % compute the product
    X=reshape(X,sizeXpermuted); % reshape the product
    X=permute(X,permutation); % swap the first and jth dimensions
end
\end{lstlisting}

In order to generate the $n$-dimensional mesh, it is possible to use always the function \verb|ndgrid|. For instance, if $n = 4$ and the column arrays \verb|x1|, \verb|x2|,  \verb|x3| and \verb|x4| correspond to $\x_1$, $\x_2$, $\x_3$ and $\x_4$, respectively, we type \verb|[X1,X2,X3,X4]=ndgrid(x1,x2,x3,x4);|. However, the MATLAB syntax makes often possible not to use \verb|ndgrid| at all: For example, if we want to compute \verb|X1+X2+X3+X4|, this is equivalent to typing \verb|x1+x2.'+permute(x3,[3 2 1])+permute(x4,[4 2 3 1]);|. Note that, instead of \verb|x2.'| it is possible to type \verb|permute(x2,[2,1]);|, or even \verb|permute(x2,[2,1,3]);| or \verb|permute(x2,[2,1,3,4]);|, etc. We have used this idea, which makes it possible to implement a code that works for any $n\in\mathbb N$.

Bearing in mind the previous arguments, we offer in Listing~\ref{code:fractionalLaplacianND} a code that shows how to approximate numerically, $(-\Delta)^se^{-x_1^2 - \ldots - x_n^2}$, whose exact expression (see \cite{sheng2020}) is given by
\begin{equation}
\label{e:Deltasexp}
	(-\Delta)^se^{-x_1^2 - \ldots - x_n^2} = \frac{2^{2s}\Gamma(s+n/2)}{\Gamma(n/2)}{}_1F_1\left(s+\frac n2; \frac n2; -x_1^2 - \ldots - x_n^2\right), \quad s\ge 0,\ n\in\mathbb N,
\end{equation}
We have taken $s = 0.13$, $n = 4$, $N_1= 39$, $N_2= 40$, $N_3= 41$, $N_4= 42$,  and $L_1 = 4.7$, $L_2 = 4.8$, $L_3 = 4.9$, $L_4 = 5$. The numerical error in max norm is $2.8839\times10^{-10}$. In what regards the execution time, while the code takes merely 0.25 seconds to execute, the computation of the exact result \eqref{e:Deltasexp}, which requires the evaluation of the function ${}_1F_1$, needs 4032.43 seconds. Note that the matrices $\PP_j^{-1}$ have to be computed only once, and it is safe to use them, because the matrices $\PP_j$ are well-conditioned.

\begin{lstlisting}[label=code:fractionalLaplacianND, language=MATLAB, basicstyle=\footnotesize, caption = {Program \texttt{fractionalLaplacianND.m}, that approximates numerically the fractional Laplacian of a function in $n$ dimensions.}]
% Approximate numerically the fractional Laplacian of $\exp(-x_1^2-\ldots-x_n^2)$
clear
tic
s=0.13; % $s$
Njall=[39,40,41,42]; % Number of nodes $\{N_1, \ldots, N_n\}$
Ljall=[4.7,4.8,4.9,5]; % Scale factors $\{L_1, \ldots, L_n\}$
if length(Njall)~=length(Ljall), error('Error in the data!'); end
n=length(Njall); % Number of dimensions $n$
Pjall=cell(1,n); % Eigenvector matrices $\mathbf P_j$
invPjall=cell(1,n); % $\mathbf P_j^{-1}$
for j=1:n
    Dxjxj=create_Dx_Dxx(Njall(j)); % $\mathbf D_{x_jx_j}$
    [Pj,Lambdaj]=eig(Dxjxj); % $\mathbf P_j$ and $\boldsymbol\Lambda_j$
    lambdaj=diag(Lambdaj)/Ljall(j)^2; % $\boldsymbol\lambda_j$
    aux=find(min(abs(lambdaj))==abs(lambdaj));
    lambdaj(aux)=0; % Set to zero the eigenvalue with the smallest modulus.
    Pj(:,aux)=1/sqrt(Njall(j)); % Eigenvector associated to $\lambda_j=0$.
    Pjall{j}=Pj; % Store $\mathbf P_j$.
    invPjall{j}=inv(Pj); % Store $\mathbf P_j^{-1}$.
    xj=Ljall(j)*cot(pi*(1/2+(0:Njall(j)-1)).'/ Njall(j)); % $\mathbf x_j$
    if j==1
        sumX2=xj.^2; % $x_1^2+\ldots+x_n^2$
        Lambda=lambdaj; % $\boldsymbol\Lambda$
    else
        P=[j 2:j-1 1 j+1:n]; % Permute the first and $m$th dimensions
        Lambda=Lambda+permute(lambdaj,P);
        sumX2=sumX2+permute(xj.^2,P);
    end
end
U=exp(-sumX2); % $u=\exp(-x_1^2-\ldots-x_n^2)$
tildeU=U;
for j=1:n, tildeU=multND(invPjall{j},tildeU,j); end
fraclap_num=(-Lambda).^s.*tildeU; % Numerical approximation of $(-\Delta)^su(\vec x)$
for j=1:n, fraclap_num=multND(Pjall{j},fraclap_num,j); end
% Exact value of the fractional Laplacian $(-\Delta)^su$
toc, tic
fraclap=2^(2*s)*gamma(s+n/2)/gamma(n/2)*hypergeom(s+n/2,n/2,-sumX2);
toc
norm(fraclap(:)-fraclap_num(:),inf) % Error in max-norm
\end{lstlisting}

\section{Numerical approximation of the fractional $p$-Laplacian in $n$ dimensions}

\label{s:fracplap}

With little modifications, Theorem~\ref{t:fraclap2D} and its corollaries allow us to develop a numerical method for the fractional $p$-Laplacian in $n$ dimensions, for which both \eqref{e:balakrishnanp} (which generalizes \eqref{e:balakrishnan}), and \eqref{e:bochnerp} (which generalizes \eqref{e:bochner}) can be used. Note that the fractional $p$-Laplacian exhibits two key differences with respect to the fractional Laplacian: when $p\neq2$, it is no longer a linear operator, and the function $\Phi_p$ intervening in its definition in \eqref{e:balakrishnanp} and \eqref{e:bochnerp} is not $C^\infty$. Moreover, if we take $\Phi_p(t) = |t|^{p-2}t$, as in \eqref{e:Phip}, it cannot be directly evaluated at $t = 0$, when $p < 2$. Therefore, in this paper, we use the following equivalent definition, based on the complex sign function, which enables us to consider directly $p\ge 1$:
\begin{equation*}
\Phi_p(t) = \sgn(t)|t|^{p-1}, \text{ where } \sgn(t) = 
\left\{
\begin{aligned}
	& \frac{t}{|t|}, & & t\neq0,
	\cr
	& 0, & & t = 0.
\end{aligned}
\right.
\end{equation*}
On the other hand, \eqref{e:balakrishnanp} and \eqref{e:bochnerp} need some clarification. Indeed, in $\Delta(t - \Delta)^{-1}\left[\Phi_p(u(\vx) - u(\cdot))\right](\vx)$ and $e^{t\Delta}\left[\Phi_p(u(\vx) - u(\cdot))\right](\vx)$, which appear respectively in these formulas, we fix $\vx$ and apply the operators $\Delta(t - \Delta)^{-1}$ and $e^{t\Delta}$ to the variables denoted by a dot in $\Phi_p(u(\vx) - u(\cdot))$ (which avoids having to write them explicitly), then evaluate the result at the same value of $\vx$. Therefore, we can understand the numerical approximation of the fractional $p$-Laplacian of a function $u\in\CC^{N_1\times\ldots\times N_n}$ as approximating numerically the fractional Laplacian of $\Phi_p(u(\vx) - u(\cdot))$ several times (i.e., once per each value of $\vx$), having replaced $s$ by $sp/2$ and multiplying the result by a constant, that we denote $C(n, s, p)$.

In order to determine $C(n, s, p)$, we observe that we have $\Delta(t - \Delta)^{-1}$ in the integrand in \eqref{e:balakrishnanp} (notation that we have kept from \cite{teso2021}), but $(-\Delta)(t - \Delta)^{-1}$ in \eqref{e:balakrishnan}. Hence, $C(n, s, p)$ comes from dividing the constant $-C_4(n, s, p)$ in \eqref{e:balakrishnanp} by $\sin(\pi sp/2)/\pi$ (which is the constant $\sin(\pi s)/\pi$ in \eqref{e:balakrishnan}, but having substituted $s$ by $sp/2$):
\begin{align}
\label{e:Cnsp}
C(n, s, p) & = \frac{-C_4(n, s, p)}{\sin(\pi sp/2)/\pi} = -\frac{\pi}{\sin(\pi sp/2)}\frac{\pi^{n/2}}{2^{sp}\Gamma((2+sp)/2)\Gamma((n+sp)/2)}C_1(n,s,p)
	\cr
& = -\frac{\pi^{n/2+1}}{\sin(\pi sp/2)2^{sp}\Gamma((2+sp)/2)\Gamma((n+sp)/2)}\frac{sp(1-s)2^{2s-2}}{\pi^{(n-1)/2}}\frac{\Gamma((n+sp)/2)}{\Gamma((p+1)/2)\Gamma(2-s)}
	\cr
& = -\frac{\pi^{1/2}2^{2s-sp-1}\Gamma(1-sp/2)}{\Gamma((p+1)/2)\Gamma(1-s)},
\end{align}
where we have used that $\Gamma(sp/2)\Gamma(1-sp/2) = \pi / \sin(\pi sp/2)$. Note that, if $p = 2$,
\begin{equation}
\label{e:Cnsp2}
C(n, s, 2) = \frac{-C_4(n, s, 2)}{\sin(\pi s)/\pi} = -\frac{\pi^{1/2}2^{2s-2s-1}\Gamma(1-2s/2)}{\Gamma((2+1)/2)\Gamma(1-s)} = -\frac{\pi^{1/2}2^{-1}}{\Gamma(3/2)} = -1.
\end{equation}
As a consequence, when $p = 2$, bearing in mind that $\Phi_2(t) = t$, \eqref{e:balakrishnanp} becomes
\begin{align}
(-\Delta)_2^s u(\vx) & = C_4(n,s,2)\int_0^\infty\Delta(t - \Delta)^{-1}\left[\Phi_2(u(\vx) - u(\cdot))\right](\vx)\frac{dt}{t^{1-s}}
	\cr
& = \frac{\sin(\pi s)}{\pi}\int_0^\infty\Delta(t - \Delta)^{-1}\left[u(\vx) - u(\cdot)\right](\vx)\frac{dt}{t^{1-s}}
	\cr
& = \frac{\sin(\pi s)}{\pi}\int_0^\infty\Delta(t - \Delta)^{-1}\left[- u(\cdot)\right](\vx)\frac{dt}{t^{1-s}}
	\cr
& = \frac{\sin(\pi s)}{\pi}\int_0^\infty(-\Delta)(t - \Delta)^{-1}u(\vx)\frac{dt}{t^{1-s}} = (-\Delta)^s u(\vx),
\end{align}
i.e., we recover \eqref{e:balakrishnan}. Hence, \eqref{e:balakrishnan} is equivalent to \eqref{e:balakrishnanp}, with $p = 2$.

Likewise, we observe that the constant of \eqref{e:bochnerp} is $C_2(n, s, p)$, whereas that of \eqref{e:bochner} is $1 / \Gamma(-s)$. Then, dividing the first constant by the second one (where we have replaced $s$ by $sp/2$), we recover the constant $C(n, s, p)$ in \eqref{e:Cnsp}:
\begin{align}
\label{e:C2Gamma}
\frac{C_2(n, s, p)}{1 / \Gamma(-sp/2)} & = \Gamma(-sp/2)\frac{\pi^{n/2}}{2^{sp}\Gamma((n+sp)/2)}C_1(n,s,p)
	\cr
& = \frac{\pi^{n/2}\Gamma(-sp/2)}{2^{sp}\Gamma((n+sp)/2)}\frac{sp(1-s)2^{2s-2}}{\pi^{(n-1)/2}}\frac{\Gamma((n+sp)/2)}{\Gamma((p+1)/2)\Gamma(2-s)} = -\frac{\pi^{1/2}2^{2s-sp-1}\Gamma(1-sp/2)}{\Gamma((p+1)/2)\Gamma(1-s)}
	\cr
& = C(n, s, p),
\end{align}
where we have used again that $\Gamma(sp/2)\Gamma(1-sp/2) = \pi / \sin(\pi sp/2)$, and also that $\Gamma(s)\Gamma(1-s) = \pi / \sin(\pi s)$. Note that, when $p = 2$, from \eqref{e:C2Gamma} and \eqref{e:Cnsp2},
\begin{equation*}
C(n, s, 2) = \frac{C_2(n, s, 2)}{1 / \Gamma(-s)} = -1 \Longrightarrow C_2(n, s, 2) = -\frac{1}{\Gamma(-s)}.
\end{equation*}
Therefore, when $p = 2$, \eqref{e:bochnerp} becomes
\begin{align*}
(-\Delta)_2^s u(\vx) & = C_2(n,s,2)\int_0^\infty e^{t\Delta}\left[\Phi_2(u(\vx) - u(\cdot))\right](\vx)\frac{dt}{t^{1+s}} = -\frac{1}{\Gamma(-s)}\int_0^\infty e^{t\Delta}\left[u(\vx) - u(\cdot)\right](\vx)\frac{dt}{t^{1+s}}
	\cr
& = -\frac{1}{\Gamma(-s)}\int_0^\infty \left(u(\vx) - e^{t\Delta}[- u(\cdot)](\vx)\right)\frac{dt}{t^{1+s}} = \frac{1}{\Gamma(-s)}\int_0^\infty(e^{t\Delta}[u](\vx) - u(\vx))\frac{dt}{t^{1+s}}
	\cr
& = (-\Delta)^s u(\vx) 
\end{align*}
i.e., we recover \eqref{e:bochner}. Hence, \eqref{e:bochner} is equivalent to \eqref{e:bochnerp}, with $p = 2$.

Bearing in mind the previous arguments, we propose the following theorem.
\begin{theorem}
	\label{teo:fracplap}
	Let $p \ge 1$, $s\in(0,1)$ and $sp/2\not\in\mathbb N$. Let $\vx = (x_1, \ldots, x_n)\in\R^n$, with $n\in\mathbb N$. Let $N_j$, with $1\le j\le n$, denote the number of nodes corresponding to the $x_j$-axis, and $\x_j= (x_{j,1}, \ldots, x_{j,N_j})^T$ the vector whose entries are these nodes. Let $\Dxj\in\R^{N_j\times N_j}$ be a diagonalizable and negative semidefinite differentiation matrix associated to $\x_j$ that approximates numerically $\partial_{x_jx_j}$. Let $\Dxj = \PP_j\cdot\Lj\cdot\PP_j^{-1}$, where $\PP_j\in\R^{N_j\times N_j}$ is the eigenvector matrix, and $\Lj = \diag(\lj)\in\R^{N_j\times N_j}$ is the diagonal eigenvalue matrix, with $\lj = ({\lambda_{j,1}, \ldots, \lambda_{j,N_j}})^T$ being the vector whose entries are the eigenvalues $\lambda_{j,i} \le 0$. Let $\U\in\CC^{N_1\times \ldots \times N_n}$ be the $n$-dimensional array such that $[\U]_{i_1\ldots i_n} = u(x_{1,i_1}, \ldots, x_{n,i_n})$, for $1\le i_j\le N_j$. Then, the numerical approximation of both \eqref{e:balakrishnanp} and \eqref{e:bochnerp} for $n\in\mathbb N$ is given by
	\begin{equation}
		\label{e:approxfracplapnD}
		(-\Delta)_p^s u(x_{1,i_1},\ldots, x_{n,i_n}) \approx C(n,s,p)\PP_n(i_n, :)\square_n[\ldots\square_2[\PP_1(i_1, :)\square_1[(-\LL)^{\odot sp/2}\odot\tilde\U(i_1, \ldots, i_n)]]],
	\end{equation}
	for $1\le i_j\le N_j$, where $\PP_j(i_j, :)$ denotes the $i_j$th row of $\PP_j$, and
	\begin{equation*}
	\tilde\U(i_1, \ldots, i_n) = \PP_n^{-1}\square_n[\ldots\square_2[\PP_1^{-1}\square_1\Phi_p([\U]_{i_1\ldots i_n}\Jn - \U)]],
\end{equation*}
	where $\Phi_p(\cdot)$ is evaluated point-wise. Moreover, $\LL\in\R^{N_1\times\ldots\times N_n}$ is the $n$-dimensional array such that
	$[\LL]_{i_1\ldots i_n} = \lambda_{1,i_1} + \ldots + \lambda_{n,i_n}$; $\square_j$ denotes the sum along the $j$th dimension, as defined in \eqref{e:squarej}, $\Jn\in\R^{N_1\times\ldots\times N_n}$ is the all-one array of size $N_1\times\ldots\times N_n$, and $(\cdot)^{\odot sp/2}$ and $\odot$ denote respectively the Hadamard or entry-wise power of $sp/2$ and product, respectively.

\end{theorem}

\begin{proof}

Let us fix $\vx = \vx_0$ and rewrite \eqref{e:balakrishnanp}:
\begin{align}
\label{e:vxvx0}
(-\Delta)_p^s u(\vx_0) & = \frac{-\pi C_4(n,s,p)}{\sin(\pi sp/2)}\left[\frac{\sin(\pi sp/2)}{\pi}\int_0^\infty(-\Delta)(t - \Delta)^{-1}\left[\Phi_p(u(\vx_0) - u(\cdot))\right](\vx)\frac{dt}{t^{1-sp/2}}\right]
	\cr
& = C(n,s,p)\left.\left[\frac{\sin(\pi sp/2)}{\pi}\int_0^\infty(-\Delta)(t - \Delta)^{-1}\left[\Phi_p(u(\vx_0) - u(\cdot))\right]\frac{dt}{t^{1-sp/2}}\right]\right|_{\vx = \vx_0},
\end{align}
where we have used \eqref{e:Cnsp}. Note that, in the first line of \eqref{e:vxvx0}, we evaluate the expression inside the integral at $\vx = \vx_0$, then we integrate with respect to $t$; however, in the second line, we integrate with respect to $t$ for each value of $\vx$, but it is not a matter of concern what happens when $\vx \neq \vx_0$, because we only keep the $\vx = \vx_0$ case after the integration, and all the $\vx\neq\vx_0$ cases are directly discarded, without even needing to know their existence. Therefore, the bracketed part in \eqref{e:vxvx0} matches \eqref{e:balakrishnan}, i.e., can be regarded as the fractional Laplacian of order $sp/2$ of $\Phi_p(u(\vx_0) - u(\cdot))$. Hence, taking $\vx_0 = (x_{1,i_1}, \ldots, x_{n,i_n})$ and applying Corollary~\ref{coro:n}, to \eqref{e:vxvx0}, we get
\begin{equation*}
(-\Delta)_p^s u(x_{1,i_1}, \ldots, x_{n,i_n}) \approx \left[\PP_n\square_n[\ldots\square_2[\PP_1\square_1[(-\LL)^{\odot sp/2}\odot\tilde\U]]]\right]_{i_1\ldots i_n},
\end{equation*}
where
\begin{equation*}
\tilde\U(i_1, \ldots, i_n) = \PP_n^{-1}\square_n[\ldots\square_2[\PP_1^{-1}\square_1\Phi_p([\U]_{i_1\ldots i_n}\Jn - \U)]],
\end{equation*}
i.e., we compute the $n$-dimensional array $\PP_n\square_n[\ldots\square_2[\PP_1\square_1[(-\LL)^{\odot sp/2}\odot\tilde\U]]\in\CC^{N_1\times\ldots\times N_n}$, then we keep the entry corresponding to $\vx_0$, which is equivalent to using only the $i_j$th row of $\PP_j$, for $1 \le j\le n$, and calculating $\PP_n(i_n, :)\square_n[\ldots\square_2[\PP_1(i_1,:)\square_1[(-\LL)^{\odot sp/2}\odot\tilde\U]]\in\CC$, from which \eqref{e:approxfracplapnD} follows.

On the other hand, we arrive also at \eqref{e:approxfracplapnD} if we start from \eqref{e:bochnerp}:
\begin{align}
\label{e:vxvx02}
(-\Delta)_p^s u(\vx) & = C_2(n,s,p)\Gamma(-sp/2)\left[\frac{1}{\Gamma(-sp/2)}\int_0^\infty e^{t\Delta}\left[\Phi_p(u(\vx_0) - u(\cdot))\right](\vx)\frac{dt}{t^{1+sp/2}}\right]
	\cr
& = C(n,s,p)\left.\left[\frac{1}{\Gamma(-sp/2)}\int_0^\infty e^{t\Delta}\left[\Phi_p(u(\vx_0) - u(\cdot))\right]\frac{dt}{t^{1+sp/2}}\right]\right|_{\vx = \vx_0},
\end{align}
where we have taken into account \eqref{e:C2Gamma}. Then, the bracketed part in \eqref{e:vxvx02} matches \eqref{e:bochnerp}, so it can also be regarded as the fractional Laplacian of order $sp/2$ of $\Phi_p(u(\vx_0) - u(\cdot))$. Applying Corollary~\ref{coro:n}, to \eqref{e:vxvx02}, and reasoning as above, \eqref{e:approxfracplapnD} follows again.

\end{proof}

\begin{remark} The comments on the scale factor done in Section~\ref{s:fraclap} for the fractional Laplacian are valid also for the fractional $p$-Laplacian. In particular, if all the second-order differentiation matrices are scaled by the same factor $L^2$, then it is enough to divide the numerical approximation of $(-\Delta)_p^su(\vec{x})$ by $(L^2)^{sp/2} = L^{sp}$.
\end{remark}

\subsection{Implementation in MATLAB}

In order to implement numerically \eqref{e:approxfracplapnD}, we can proceed in two different ways: by computing \eqref{e:approxfracplapnD} for all the $N_1\times\ldots\times N_n$ different values that can take $(i_1, \ldots, i_n)$, or by working with $2n$-dimensional arrays, which produces very compact codes, but the memory requirements quickly become prohibitive, even for $n = 2$.

With respect to the first option, we can consider $n$ nested for-loops to generate all the possible $n$-tuples $(i_1, \ldots, i_n)$. However, we proceed as in \cite{cuestadelahozgirona2024}, where we generate them sequentially. More precisely, we use the following algorithm, taken from \cite{cuestadelahozgirona2024}.

\begin{algorithm}
	\caption{Sequential generation of the $n$-tuples $(i_1, \ldots, i_n)$\label{alg:i1i2in}}
	\begin{algorithmic}
		\State $indexmax \gets 1$
		\For{$j\gets 1$ \textbf{to} $n$}
		\State $indexmax \gets N_j\cdot indexmax$
		\State $i_j \gets N_j$
		\EndFor
		\For{$index \gets indexmax$ \textbf{to} $1$ \textbf{step} $-1$}
		\State \textbf{output} $(i_1, \ldots, i_n)$
		\State $k \gets 1$
		\While{$k \le n$}
		\If{$i_k > 1$}
		\State $i_k \gets i_k - 1$
		\State \textbf{break}
		\EndIf
		\State $i_k \gets N_k$
		\State $k \gets k + 1$
		\EndWhile
		\EndFor
	\end{algorithmic}
\end{algorithm}

This algorithm has several advantages: it consists of one single for-loop, it can be used without any change for any number $n$ of dimensions, and it generates the $n$-tuples, indexing them by means of the variable \verb|index|, in such a way that, given an $n$-tuple $(x_{1,i_1}, \ldots, x_{n, i_n})$ and its corresponding value of \verb|index|, accessing $u(x_{1,i_1}, \ldots, x_{n, i_n})$ is reduced to typing \verb|U(index)|. Note that the algorithm generates the $n$-tuples from the last one to the first one, because this order was necessary in \cite{cuestadelahoz2024}, but in our case, the order does not matter, and they could be generated in any other order, provided that none is forgotten.

On the other hand, as said above, it is possible to implement \eqref{e:approxfracplapnD}, without needing to generate explicitly the $n$-tuples $(i_1, \ldots, i_n)$, by taking advantage of MATLAB's syntax. To illustrate this, suppose that we take $n = 1$ (in what follows, we will use the notation of Corollary~\ref{coro:1}). Then, instead of evaluating point-wise $\Phi_p$ at $[\U]_i\Jx - \U\in\CC^{N_x}$, for $1 \le i \le N_x$, where $[\U]_i = u(x_i)$, and $\Jx\in\R^{N_x}$ is the all-one vector of length $N_x$, we evaluate point-wise $\Phi_p$ at $\Jx\cdot\U^T - \U\cdot\Jx^T\in\CC^{N_x\times N_x}$, which in MATLAB is reduced to typing \verb|phip(u.' - u)|, where the function \verb|phip| corresponds to \eqref{e:Phip}. Bearing in mind this, we apply Corollary~\ref{coro:1} simultaneously to all the columns of $\Phi_p(\Jx\cdot\U^T - \U\cdot\Jx^T)\in\CC^{N_x\times N_x}$, taking $sp/2$ instead of $s$, and the result is precisely the diagonal of the resulting matrix multiplied by $C(1, s, p)$:
\begin{equation}
\label{e:fracplap1D}
(-\Delta)_p^s u(\x) \approx C(1, s, p)\diag\left(\Px\cdot\left[[(-\boldsymbol\lambda_x)^{\odot sp/2}\cdot\Jx^T]\odot\left[\Px^{-1}\cdot[\Phi_p(\Jx\cdot\U^T - \U\cdot\Jx^T)]\right]\right]\right).
\end{equation}
Moreover, in the MATLAB implementation, we do not need to multiply the column vector $(-\boldsymbol\lambda_x)^{\odot sp/2}$ by $\Jx^T$. Indeed, if \verb|v| is a column vector and \verb|A| is a matrix whose number of rows coincides with the length of \verb|v|, then \verb|v.*A| multiplies point-wise \verb|v| by all the columns of \verb|A|. Thus, the right-hand side \eqref{e:fracplap1D} can be expressed in MATLAB in a very compact way:
\begin{verbatim}
	C*diag(Px*((-lambdax).^(s*p/2).*(Px\phip(U.'-U))))
\end{verbatim}
where \verb|C| stores $C(1,s,p)$, and \verb|Px|, \verb|lambdax| and \verb|phip| correspond to $\PP_x$, $-\lx$ and $\Phi_p$, respectively. Note that \verb|U.'| is equivalent to typing \verb|permute(U,[2,1])|. Therefore, in order to generalize \eqref{e:fracplap1D} to $n$-dimensions, we observe that, if \verb|U| is an $n$-dimensional array, then all the instances of $[\U]_{i_1\ldots i_n}\Jn - \U$ can be generated simultaneously by typing \verb|permute(U,[n+1:2*n 1:n])-U|, which produces a $2n$-dimensional array of size $N_1\times\ldots\times N_n\times N_1\times\ldots\times N_n$. Then, if \verb|V| is an $n$-dimensional array, and \verb|A| is a $2n$-dimensional array, such that the size of the first $n$ dimensions of \verb|V| coincide with the size of the dimensions of $\V$, then \verb|V.*A| is equivalent to multiplying entry-wise \verb|V| by all the possible $n$-dimensional arrays that arise after giving values to the last $n$-dimensions of \verb|A|. For instance, if we type
\begin{verbatim}
	V=rand(2,3);
	A=rand(2,3,4,5);
	B=V.*A;
\end{verbatim}
the last line produces exactly the same result as if we type
\begin{verbatim}
	B=zeros(size(A));
	for i1=1:4
		for i2=1:5
			B(:,:,i1,i2)=V.*A(:,:,i1,i2);
		end
	end
\end{verbatim}
but \verb|B=V.*A| is faster. Applying this idea to Corollary~\ref{coro:n}, we get the equivalent of \eqref{e:fracplap1D}. Note that, in the case of a $2n$-dimensional array of size $N_1\times\ldots\times N_n\times N_1\times\ldots\times N_n$, with $n > 1$, we cannot compute directly the diagonal of it, as in \eqref{e:fracplap1D}. Therefore, we reshape it to a square matrix of order $N_1\times\ldots\times N_n$, then, extract its diagonal, which is a vector of length $N_1\times\ldots\times N_n$, or, equivalently, if we index the elements of the array from $1$ to $N_1^2\times\ldots\times N_n^2$, then the elements we are interested in are stored precisely at the positions $1 + k(N_1\times\ldots\times N_n + 1)$, for $0\le k\le(N_1\times\ldots\times N_n-1)$. Finally, we reshape the vector of length $N_1\times\ldots\times N_n$ to an $n$-dimensional array of size $N_1\times\ldots\times N_n$.

In Listing~\ref{code:fractionalpLaplacianND}, we offer the MATLAB code that approximates numerically $(-\Delta)_p^se^{-x_1^2 - \ldots - x_n^2}$. Besides, we measure the elapsed time to generate $(-\LL)^{sp/2}$, $\U$ and the matrices $\PP_j$ and $\PP_j^{-1}$, and also the elapsed times to compute the numerical approximation of $(-\Delta)_p^se^{-x_1^2 - \ldots - x_n^2}$ following both a loop-based approach and a loopless approach. Unfortunately, we do not know the exact expression, except when $p = 2$, which is given by \eqref{e:Deltasexp}, but, as a preliminary test to verify the correctness of the implementation, we compare the max norm of the difference between the approximations given by the two approaches, which is infinitesimally small, because, e.g., MATLAB can produce infinitesimally different results if we multiply directly a matrix by a column vector, and if we multiply one by one the rows of the matrix by those of the vector. Moreover, when $p = 2$, we also compare the results with the exact solution in \eqref{e:Deltasexp}, which coincide (again, up to infinitesimally small differences) with those given by Listing~\ref{code:fractionalLaplacianND}.

In general, the loopless approach is faster, provided that the available memory is enough to handle correctly the size of the $2n$-dimensional array; however, if the size is too large, this approach can simply be unfeasible. For instance, when $p = 2$, $s = 0.4$, $n = 1$, $N_1 = 5000$ and $L_1 = 490$, generating $(-\LL)^{sp/2}$, $\U$, $\PP_1$ and $\PP_1^{-1}$ takes 44.17 seconds; then, the loop-based approach takes 44.36 seconds (with an error of $1.0449\times10^{-12}$), whereas the loopless approach takes only 5.08 seconds (with a virtually identical error of $1.0453\times10^{-12}$); moreover, the discrepancy between both approaches is of $4.4409\times10^{-16}$, i.e., negligible. In what regards the computation of the exact calculation, it takes 9.74 seconds. At this point, we remark that $\LL$ and $\U$ only need to be computed once, and then reused for different values of $s$ and $p$ (we have taken $p = 2$, to compare the results with the exact solution, but a different choice of $p$ does not affect the execution time).

On the other hand, when $p = 2$, $s = 0.67$, $n = 2$, $N_1 = 200$, $N_2 = 201$, $L_1 = 18$ and $L_2 = 18.1$, generating $(-\LL)^{sp/2}$, $\U$, $\PP_1$, $\PP_1^{-1}$, $\PP_2$ and $\PP_2^{-1}$ takes only 0.03 seconds; then, the loop-based approach takes 47.65 seconds (with an error of $7.9572\times10^{-13}$), whereas the loopless approach takes 340.70 seconds (with an identical error of $7.9572\times10^{-13}$); moreover, the discrepancy between both approaches is of $6.6613\times10^{-16}$, i.e., virtually equal to the epsilon of the machine. In what regards the computation of the exact calculation, it takes 78.38 seconds. The reason why the loopless approach is much slower is due to the fact that the four-dimensional array of size $200\times201\times200\times201$ occupies $12.92832\times10^9$ bytes of memory, and the total amount of available RAM memory is 32 gigabytes. For the sake of comparison, we have run this example in another computer with 128 Gb of RAM memory, and the loopless approach is almost five times as faster as the loop-based approach.

\begin{lstlisting}[label=code:fractionalpLaplacianND, language=MATLAB, basicstyle=\footnotesize, caption = {Program \texttt{fractionalpLaplacianND.m}, that approximates numerically the fractional $p$-Laplacian of a function in $n$ dimensions}]
clear
close all
tic
p=2; % $p$
s=0.67; % $s$
Njall=[200,201]; % Number of nodes $\{N_1, \ldots, N_n\}$
Ljall=[18,18.1]; % Scale factors $\{L_1, \ldots, L_n\}$
%s=0.4; Njall=5000; Ljall=490;
if length(Njall)~=length(Ljall), error('Error in the data!'); end
n=length(Njall); % Number of dimensions $n$
Pjall=cell(1,n); % Eigenvector matrices $\mathbf P_j$
invPjall=cell(1,n); % $\mathbf P_j^{-1}$
phip=@(t)abs(t).^(p-1).*sign(t); % $\Phi_p(t)$
C=-sqrt(pi)*2^(2*s-s*p-1)*gamma(1-s*p/2)/gamma((p+1)/2)/gamma(1-s); % $C(n,s,p)$
tic % Start of the loop-based approach
for j=1:n
    Dxjxj=create_Dx_Dxx(Njall(j)); % $\mathbf D_{x_jx_j}$
    [Pj,Lambdaj]=eig(Dxjxj); % $\mathbf P_j$ and $\boldsymbol\Lambda_j$
    lambdaj=diag(Lambdaj)/Ljall(j)^2; % $\boldsymbol\lambda_j$
    aux=find(min(abs(lambdaj))==abs(lambdaj));
    lambdaj(aux)=0; % Set to zero the eigenvalue with the smallest modulus.
    Pj(:,aux)=1/sqrt(Njall(j)); % Eigenvector associated to $\lambda_j=0$.
    Pjall{j}=Pj; % Store $\mathbf P_j$.
    invPjall{j}=inv(Pj); % Store $\mathbf P_j^{-1}$.
    xj=Ljall(j)*cot(pi*(1/2+(0:Njall(j)-1)).'/ Njall(j)); % $\mathbf x_j$
    if j==1
        sumX2=xj.^2; % $x_1^2+\ldots+x_n^2$
        Lambda=lambdaj; % $\boldsymbol\Lambda$
    else
        P=[j 2:j-1 1 j+1:n]; % Permute the first and $m$th dimensions
        Lambda=Lambda+permute(lambdaj,P);
        sumX2=sumX2+permute(xj.^2,P);
    end
end
Lsp=(-Lambda).^(s*p/2);
U=exp(-sumX2); % $u=\exp(-x_1^2-\ldots-x_n^2)$
toc,tic
if n==1, fracplap_num=zeros(Njall,1); else fracplap_num=zeros(Njall); end
ii=Njall; % ii contains the $n$-tuples $(i_1, \ldots, i_n)$.
for index=prod(Njall):-1:1 % index contains the position of the $n$-tuple.
    tildeU=phip(U(index)-U); % $\tilde{\mathbf U}$
    for j=1:n, tildeU=multND(invPjall{j},tildeU,j); end
    aux=Lsp.*tildeU;
    for j=1:n, aux=multND(Pjall{j}(ii(j),:),aux,j); end
    fracplap_num(index)=aux;
    if ii(1)>1
        ii(1)=ii(1)-1;
    else
        k=2;
        while k<=n&&ii(k)==1
            k=k+1;
        end
        if k<=n
            ii(k)=ii(k)-1;
            ii(1:k-1)=Njall(1:k-1);
        end
    end
end
fracplap_num=C*fracplap_num; % Numerical approximation of $(-\Delta)_p^su(\vec x)$
toc, tic % Start of the loopless approach
tildeU=phip(permute(U,[n+1:2*n 1:n])-U);
for j=1:n, tildeU=multND(invPjall{j},tildeU,j); end % $\tilde{\mathbf U}$
tildeU=Lsp.*tildeU;
for j=1:n, tildeU=multND(Pjall{j},tildeU,j); end
if n==1 % Numerical approximation of $(-\Delta)_p^su(\vec x)$
    fracplap_num2=C*diag(tildeU);
else
    fracplap_num2=C*reshape(tildeU(1:prod(Njall)+1:prod(Njall)^2),Njall);
end
toc
norm(fracplap_num(:)-fracplap_num2(:),inf) % Discrepancy in max norm
if p==2 % When $p=2$, we know the explicit solution.
    tic
    fracplap=2^(2*s)*gamma(s+n/2)/gamma(n/2)*hypergeom(s+n/2,n/2,-sumX2);
    toc
    norm(fracplap(:)-fracplap_num(:),inf) % Error in max norm
    norm(fracplap(:)-fracplap_num2(:),inf) % Error in max norm    
end
\end{lstlisting}

\section{Numerical experiments for the fractional Laplacian}

\label{s:numericalfracLap}

In what follows, we perform several numerical experiments to see how well our method approximates the fractional Laplacian. To that purpose, we use the explicit formulas given in \cite{sheng2020}, namely \eqref{e:Deltasexp} and
\begin{equation}
\label{e:Deltas1x2}
(-\Delta)^s\left(\frac{1}{(1+x_1^2 + \ldots + x_n^2)^r}\right) = \frac{2^{2s}\Gamma(s+r)\Gamma(s+n/2)}{\Gamma(r)\Gamma(n/2)}{}_2F_1\left(s+r,s+\frac n2; \frac n2; -x_1^2 - \ldots - x_n^2\right),
\end{equation}
where $s \ge 0$, $r \ge 0$ and $n\in\mathbb N$. 

Since we are computing the fractional Laplacian of radial functions in \eqref{e:Deltasexp} and \eqref{e:Deltas1x2}, we take for the sake of simplicity $N = N_1 = \ldots = N_n$ and $L = L_1 = \ldots = L_n$ in the numerical experiments. Observe that, in that case, $\Dxj$ is identical, for all $1\le j\le n$, so it is enough to generate it once.

In order to measure the errors, we compute $\|(-\Delta)^s(\cdot) - (-\Delta)^s_{num}(\cdot)\|_\infty$, where $(-\Delta)^s_{num}$ denotes the numerical approximation of the fractional Laplacian. We have taken $N\in \{16, 32, 64, 128\}$, $s\in\{0.1, 0.2, \ldots, 0.9\}$, $n\in\{1, 2, 3\}$, and, in \eqref{e:Deltas1x2}, $r \in \{1, 2\}$. In what regards $L$, we have taken $L \in\{0.1, 0.2, \ldots, 30\}$, when $n = 1$; $L \in\{0.1, 0.2, \ldots, 20\}$, when $n = 2$; and $L \in\{0.1, 0.2, \ldots, 16\}$, when $n = 3$. In Figure~\ref{f:Deltasexp}, we plot the errors corresponding to \eqref{e:Deltasexp}; in Figure~\ref{f:Deltasr1}, the errors corresponding to \eqref{e:Deltas1x2}, with $r = 1$; and in Figure~\ref{f:Deltasr2}, the errors corresponding to \eqref{e:Deltas1x2}, with $r = 2$. In all cases, a distinct color is used for each value of $s$, and a distinct pattern for each value of $N$. To facilitate the comparison, the errors are plotted on a semilogarithmic scale over the range $[10^{-14}, 1]$.

\begin{figure}
	\centering
\includegraphics[width=0.3333\textwidth]{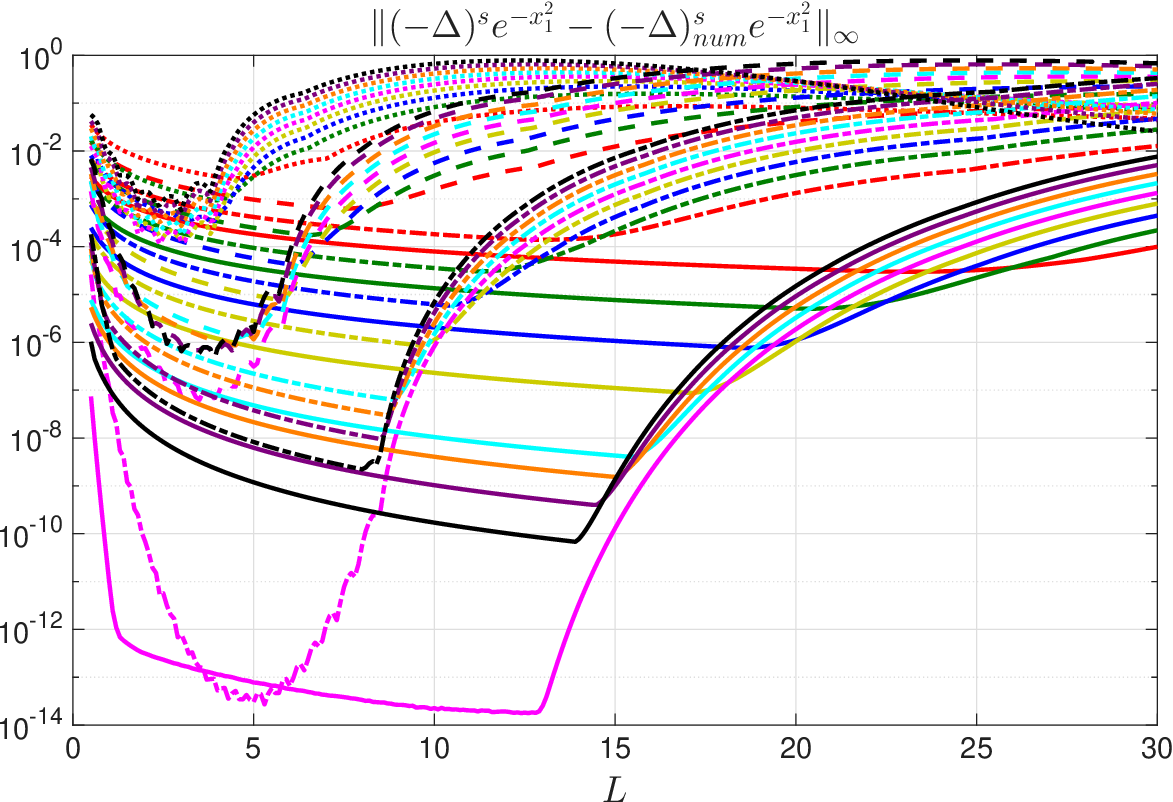}\includegraphics[width=0.3333\textwidth]{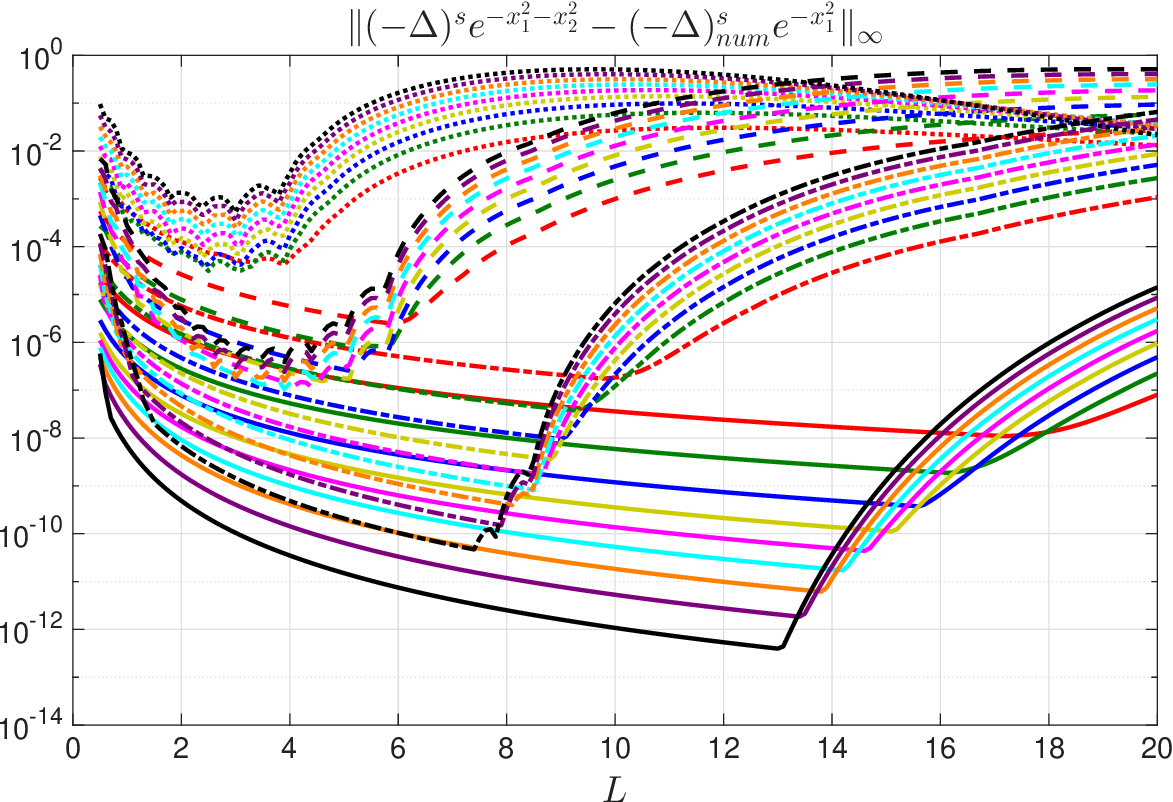}\includegraphics[width=0.3333\textwidth]{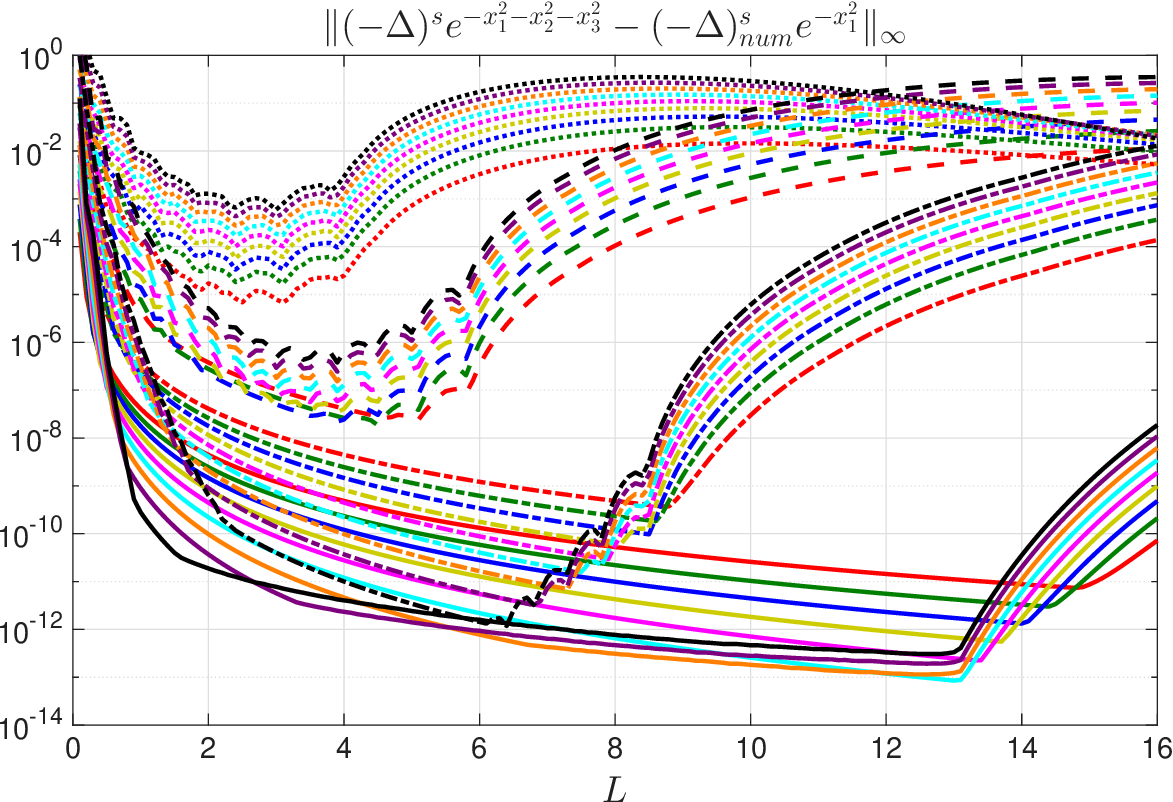}
\includegraphics[width=\textwidth]{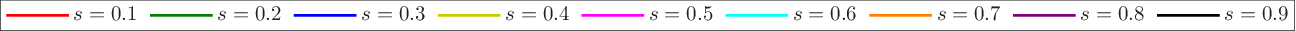}
\includegraphics[width=0.4444\textwidth]{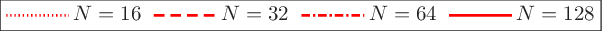}
\caption{Errors in semilogarithmic scale corresponding to \eqref{e:Deltasexp}, in one dimension (left), two dimensions (center), and three dimensions (right).}
\label{f:Deltasexp}
\end{figure}

\begin{figure}
	\centering
	\includegraphics[width=0.3333\textwidth]{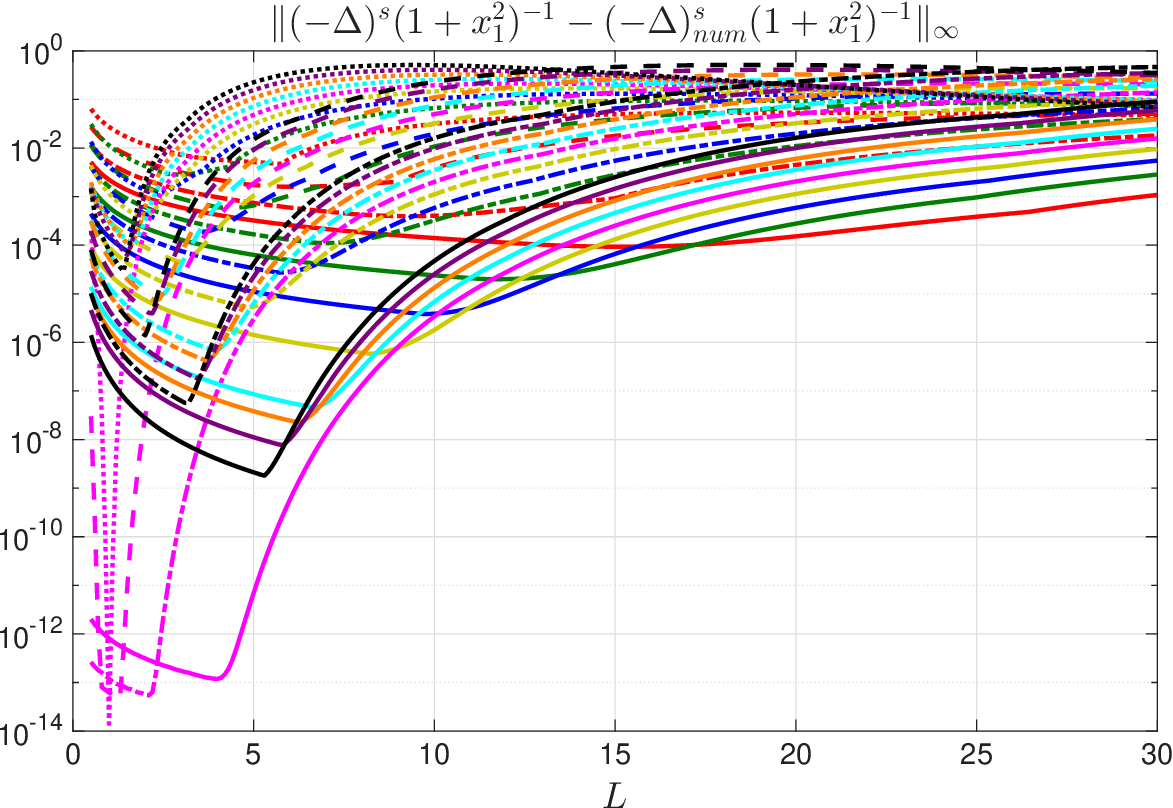}\includegraphics[width=0.3333\textwidth]{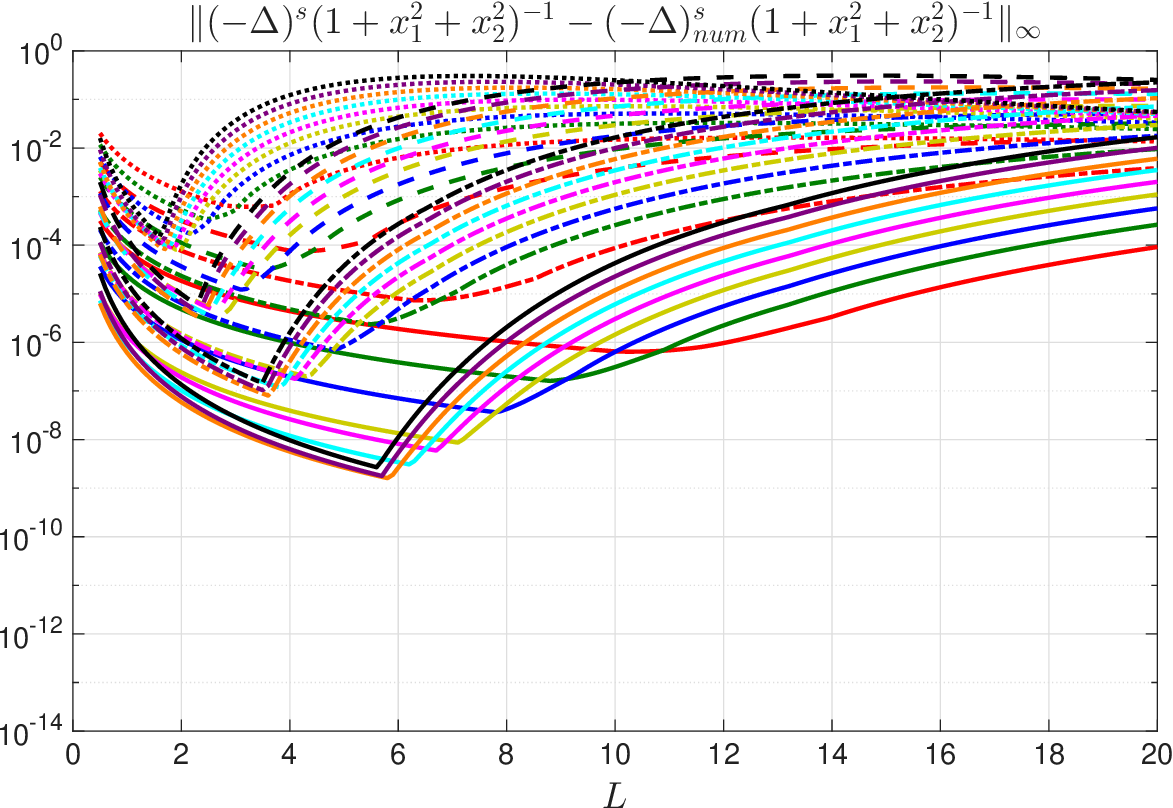}\includegraphics[width=0.3333\textwidth]{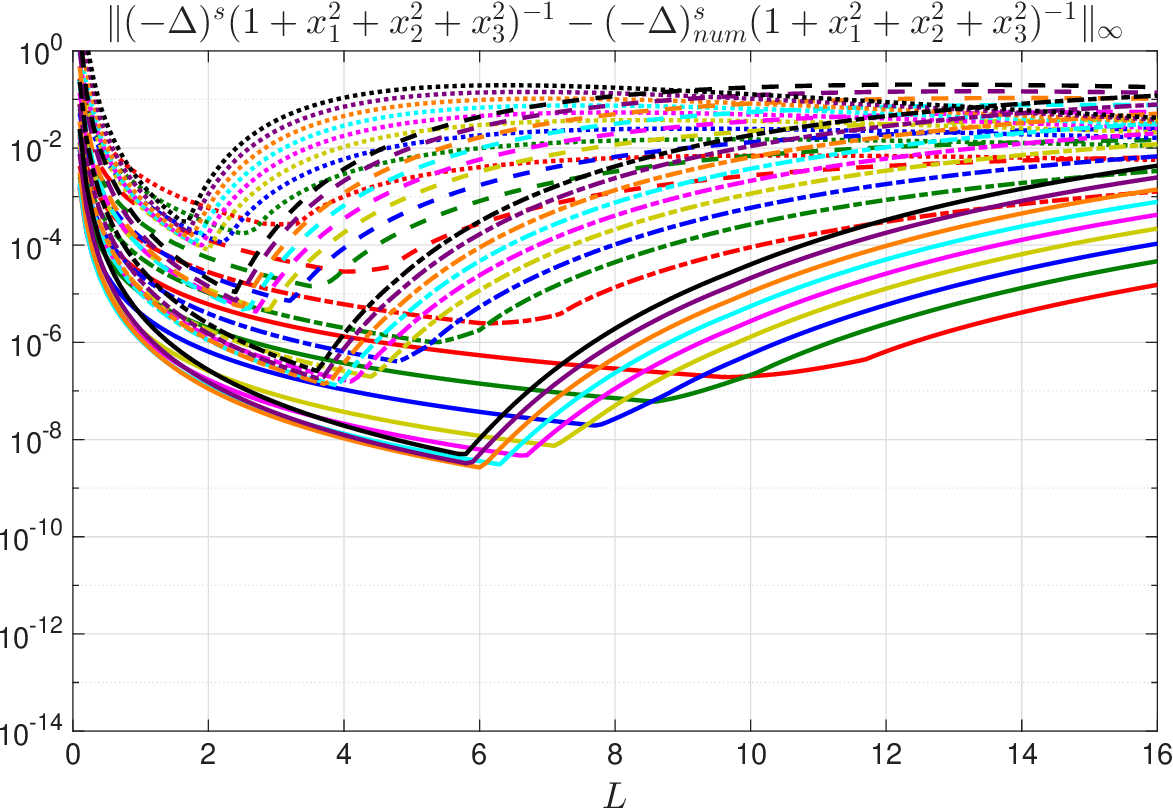}
	\includegraphics[width=\textwidth]{legends.eps}
	\includegraphics[width=0.4444\textwidth]{legendN.eps}
	\caption{Errors in semilogarithmic scale corresponding to \eqref{e:Deltas1x2} with $r = 1$, in one dimension (left), two dimensions (center), and three dimensions (right).}
	\label{f:Deltasr1}
\end{figure}

\begin{figure}
	\centering
	\includegraphics[width=0.3333\textwidth]{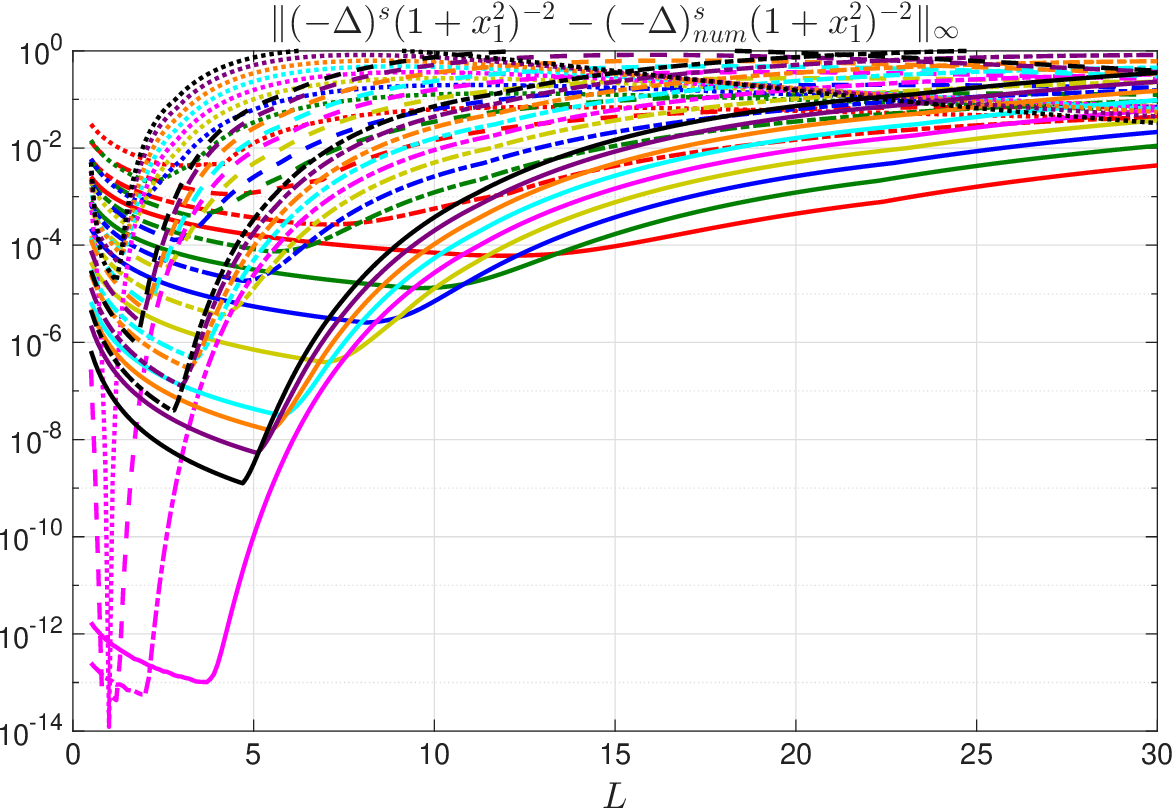}\includegraphics[width=0.3333\textwidth]{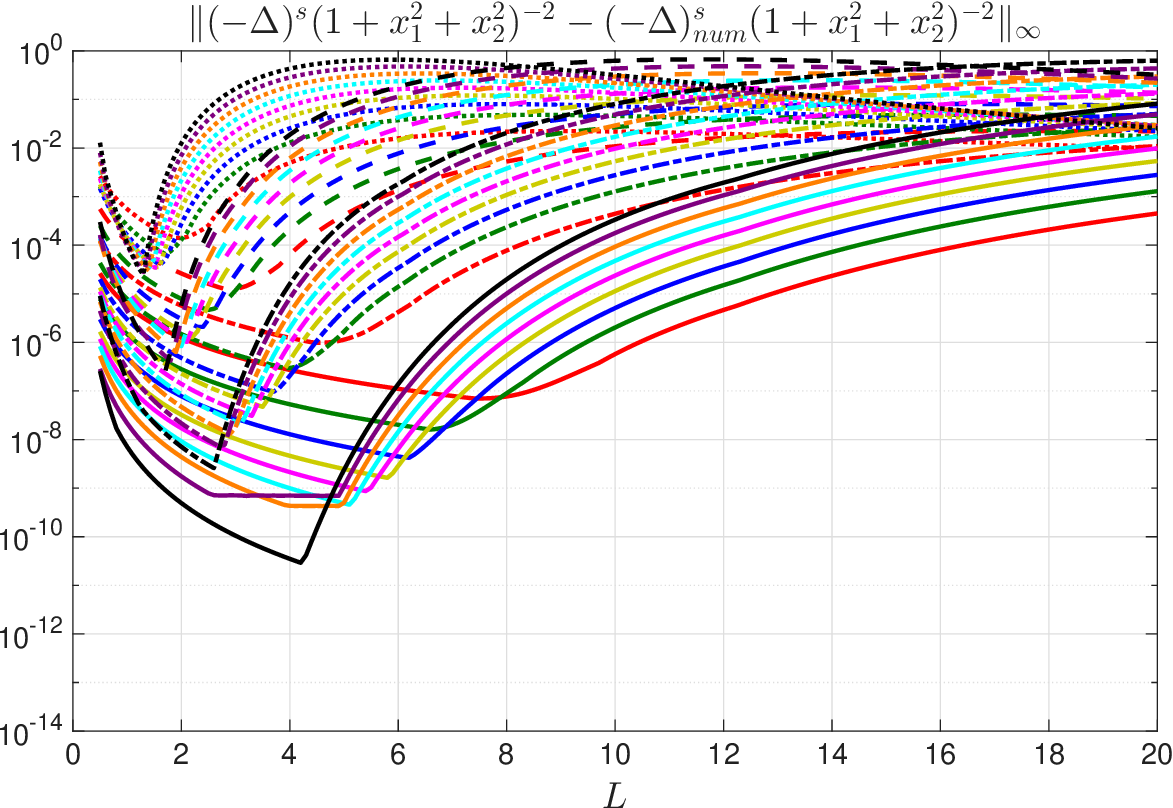}\includegraphics[width=0.3333\textwidth]{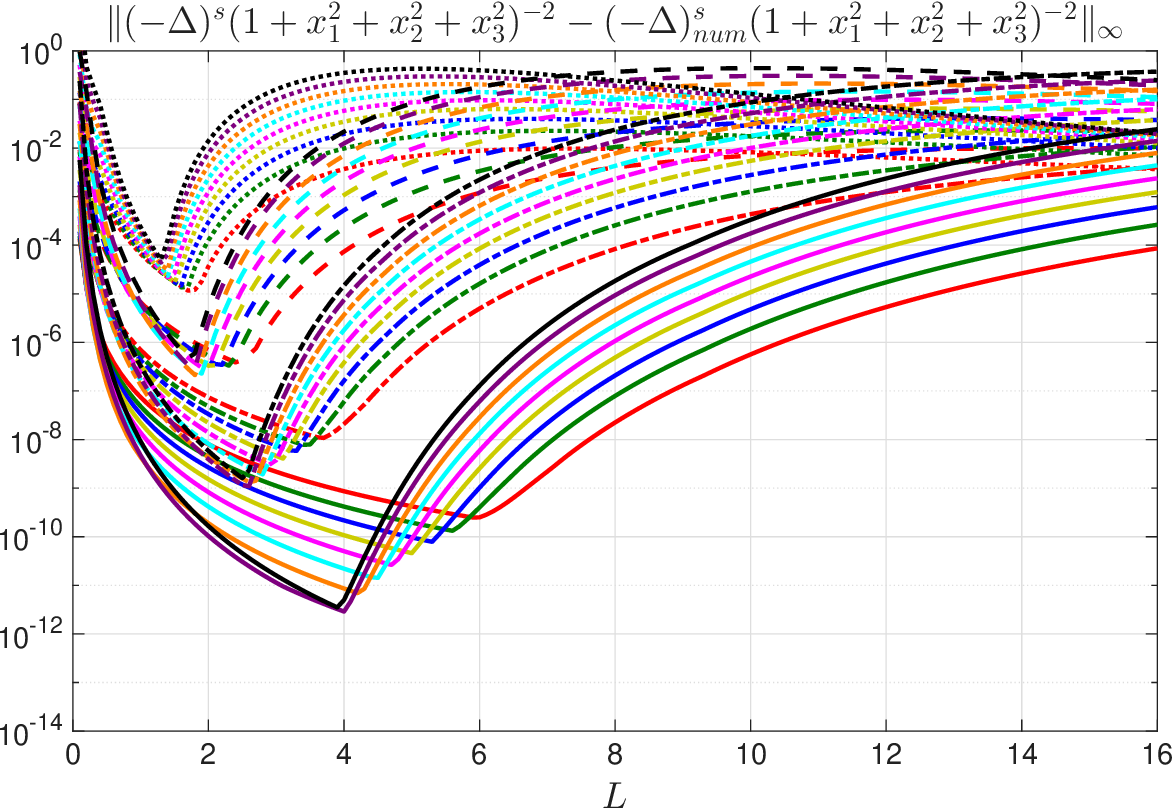}
	\includegraphics[width=\textwidth]{legends.eps}
	\includegraphics[width=0.4444\textwidth]{legendN.eps}
	\caption{Errors in semilogarithmic scale corresponding to \eqref{e:Deltas1x2} with $r = 2$, in one dimension (left), two dimensions (center), and three dimensions (right).}
	\label{f:Deltasr2}
\end{figure}

As can be seen from the numerical experiments, the accuracy of the results improves as $n$ increases, which is in agreement with the results in \cite{sheng2020}. Moreover, in general, it also increases with $s$, except for the special case when $n = 1$ and $s = 1/2$, which corresponds to the half Laplacian, where spectral accuracy is achieved for the three functions considered (with errors of the order of $\mathcal O(10^{-14})$). On the other hand, when $n = 1$ and $s\neq1/2$, the results are rather modest, especially for $s = 0.1$, so, when $n = 1$, it seems advisable to work with \eqref{e:D2aalternative}. Note also that the decay rate of the function $u(x)$ plays a role; more precisely, the faster the decay, the more accurate the results. 

With respect to the dependence of the results on $L$, even if there are some theoretical results \cite{boyd1982}, the experiments confirm what was said in \cite{cayamacuestadelahoz2020,cayamacuestadelahoz2021,cuestadelahozgirona2024,cayamacuestadelahozgarciacervera2025}, i.e., that the problem of determining the best choice for $L$ is a difficult one that depends on many factors (here, $n$, $s$, $N$ and the function itself). However, a good working rule of thumb seems to be that the absolute value of a given function at the extreme grid points is smaller than a given threshold.

\section{Numerical experiments for the fractional $p$-Laplacian}
\label{s:numericalfracpLap}

In what regards the fractional $p$-Laplacian with $p\neq2$, unfortunately, we do not have exact, regular and bounded solutions with which we can compare our numerical approximations. However, since the numerical approximation of the fractional $p$-Laplacian is based on that of the fractional Laplacian, we can expect a similar behavior of the former, i.e., that the accuracy of the results improves with $n$, and, in general, with $s$. At this point, we must recall that, as said in the introduction, the formulas \eqref{e:balakrishnanp} and \eqref{e:bochnerp} on which our numerical method is based are formally equivalent to \eqref{e:Dpa}, when $p\in (1,\infty)$, $s\in (0,1)$ and $sp < 2$. However, this equivalence is no longer clear when $sp \ge 2$. Therefore, when $sp < 2$ is not satisfied, even if our numerical method returns coherent results, it is unclear to which operator they really correspond. To illustrate this, we have approximated numerically $(-\Delta)_p^s e^{-x_1^2}$, for $s = 0.8$, taking $N = 2000$ and $L = 20$. Note that, when $p = 2$, the exact solution is known, and the error is equal to $1.1749\times 10^{-10}$. 

In the left panel of Figure~\ref{f:fracplapexpn1all}, we consider $p\in(1,2/s) = (1,2.5)$, whereas in the right panel of Figure~\ref{f:fracplapexpn1all}, we take $p\in(1, 10)$. When $p \in (1, 2.5)$, which is the range of $p$ where $sp < 2$ is satisfied, we can observe the emergence of two singularities, as $p \to 1^+$ and $p\to 2.5^-$. In contrast, if we consider a larger set of values of $p$, e.g., $p \in (1, 10)$, singularities also appear at $p = 5$, $p = 7.5$, and $p = 10$. However, as noted above, it is not clear which operator the case with $p \in (2.5, 10)$ corresponds to, because $sp > 2$.

\begin{figure}
	\centering
	\includegraphics[width=0.5\textwidth]{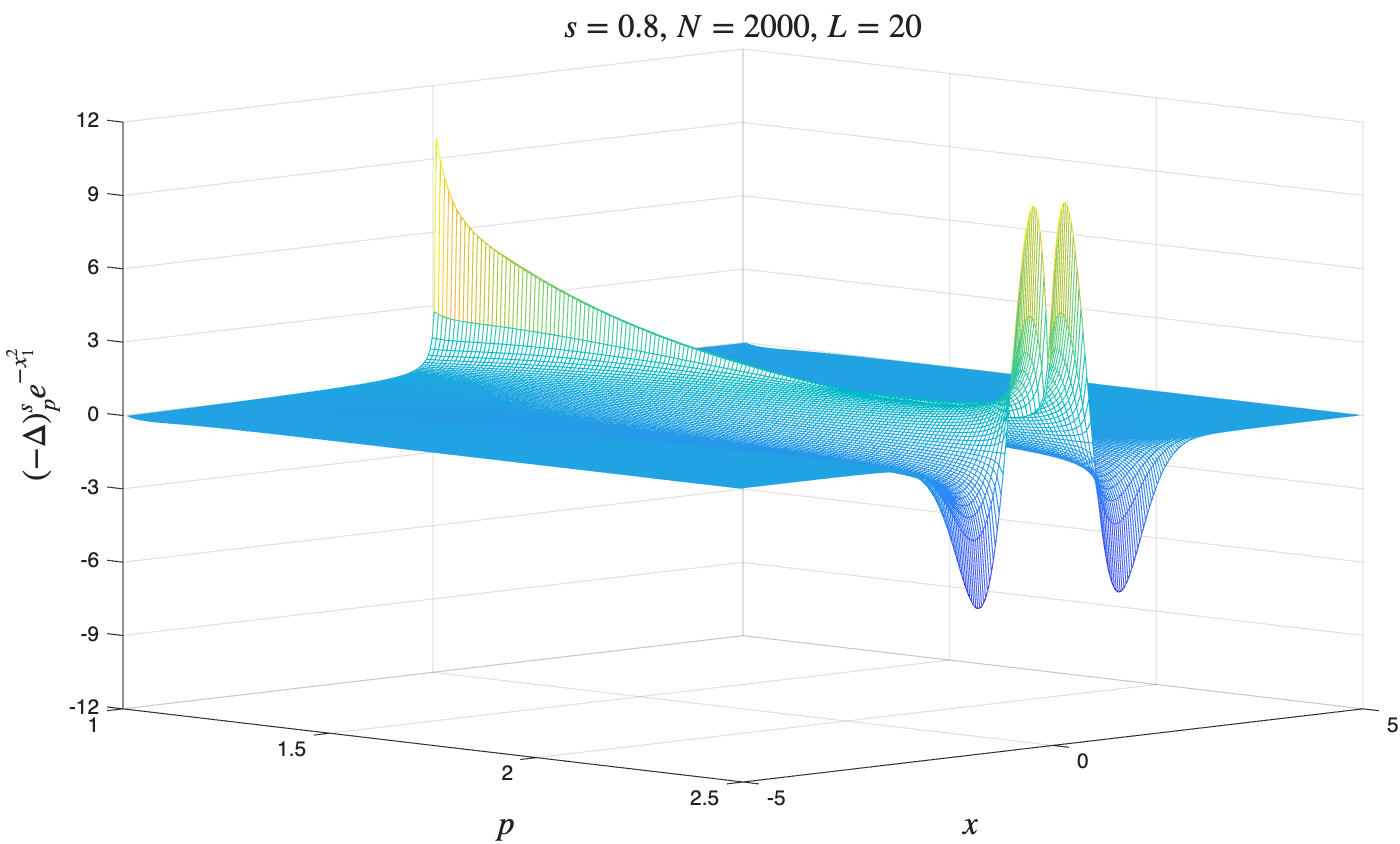}\includegraphics[width=0.5\textwidth]{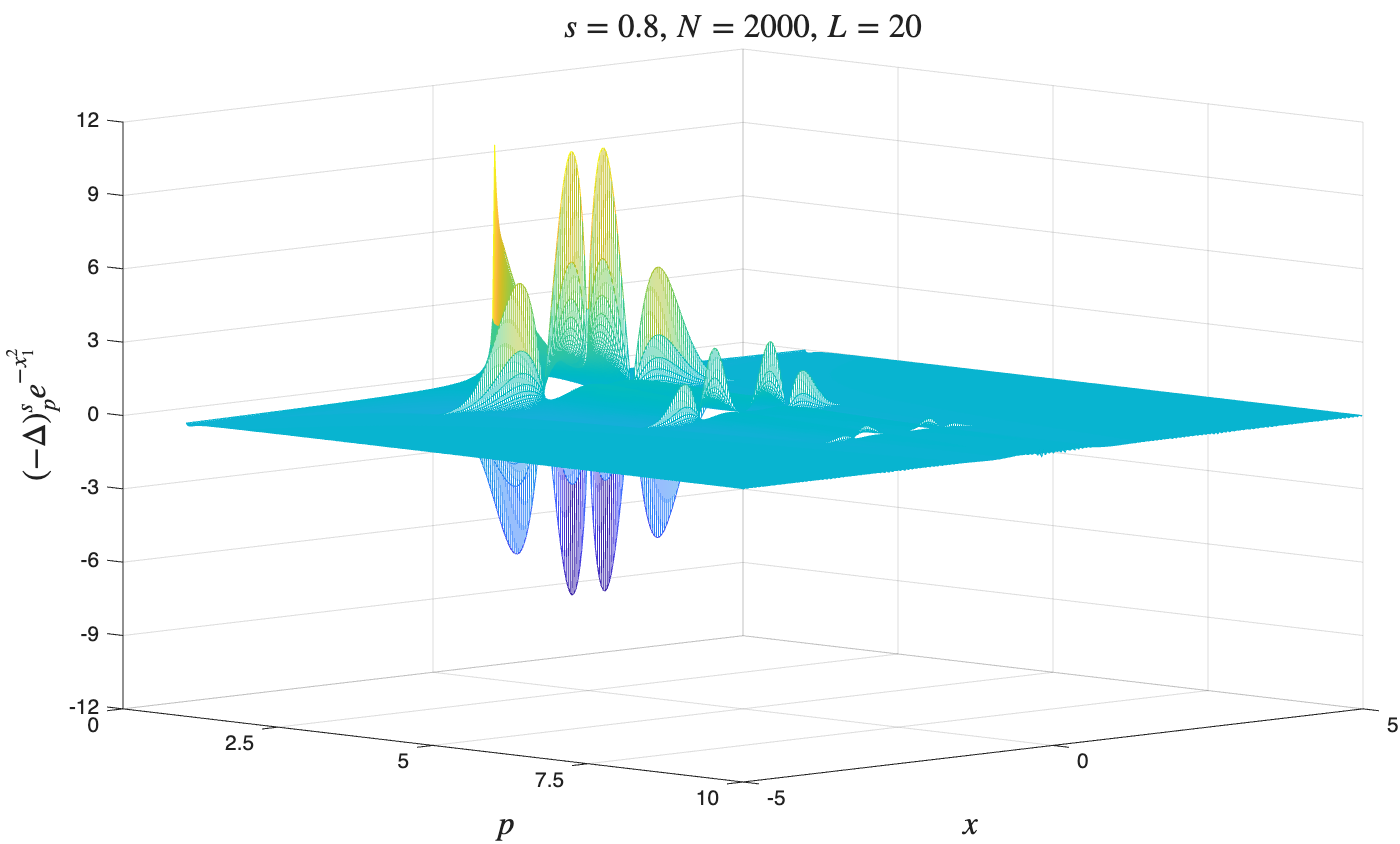}
	\caption{$(-\Delta)_p^s e^{-x_1^2}$, for $s = 0.8$, $N = 2000$, $L = 20$. Left: $p\in\{1.01, 1.02, \ldots, 2.49\}$. Right: $p\in\{1.01, 1.02, \ldots, 9.99\}$.}
	\label{f:fracplapexpn1all}
\end{figure}

On the other hand, when $p\in (1, 2.5)$, two distinct regimes can be observed: one for $p < 2$ and another for $p > 2$. In the former case, the solution exhibits a single peak, whereas in the latter it develops two peaks, as shown in Figure~\ref{f:fracplapexpn1p}. Moreover, this transition from a single peak when $p < 2$ to multiple peaks when $p > 2$ also appears to persist in higher dimensions. This behavior is illustrated in Figure~\ref{f:fracplapexpn2}, which corresponds to the two-dimensional case with $s  = 0.8$, $N = 250$ and $L = 4$. For comparison, in the case with $p = 2$, the exact solution is again known, and the error is equal to $1.2875\times 10^{-10}$.

\begin{figure}
	\centering
	\includegraphics[width=0.5\textwidth]{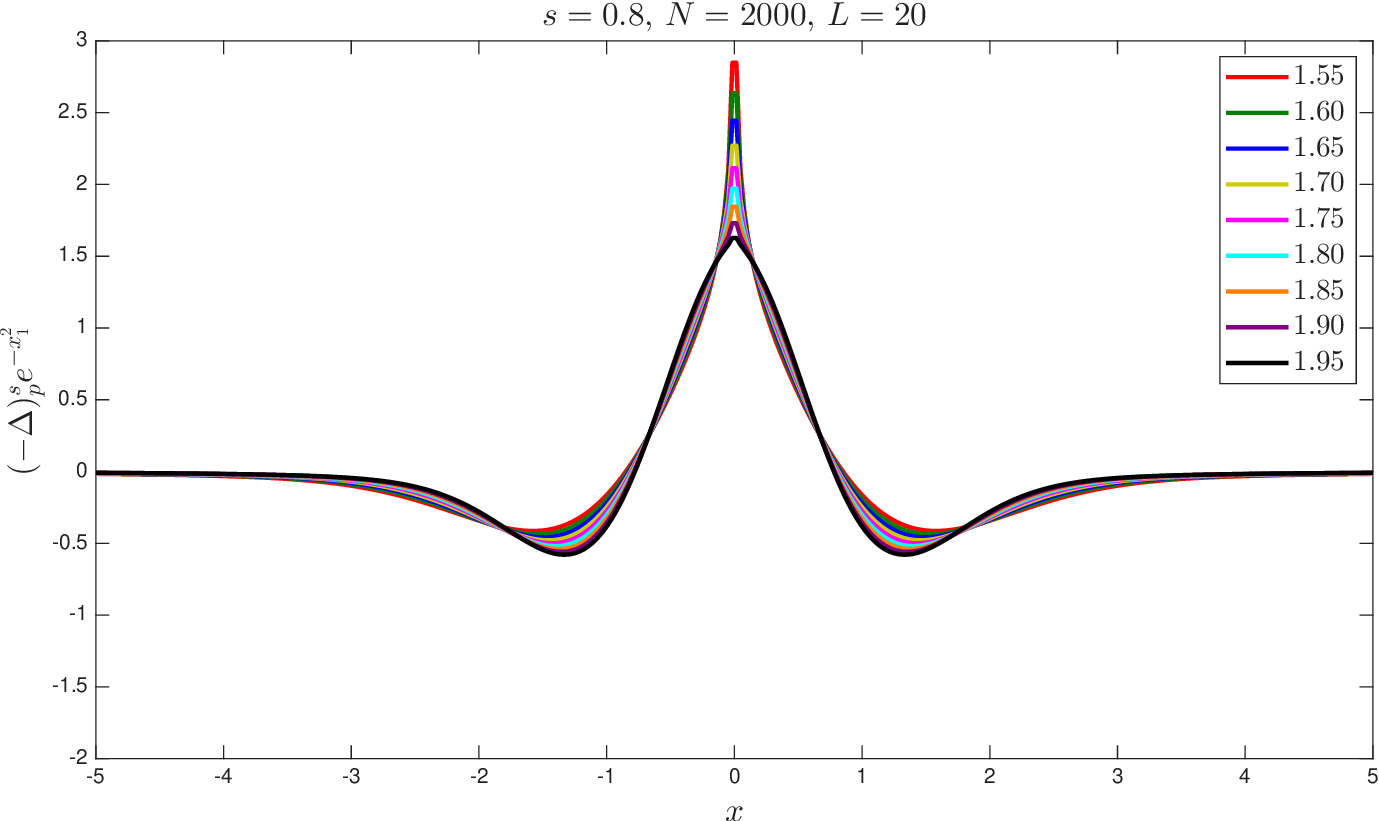}\includegraphics[width=0.5\textwidth]{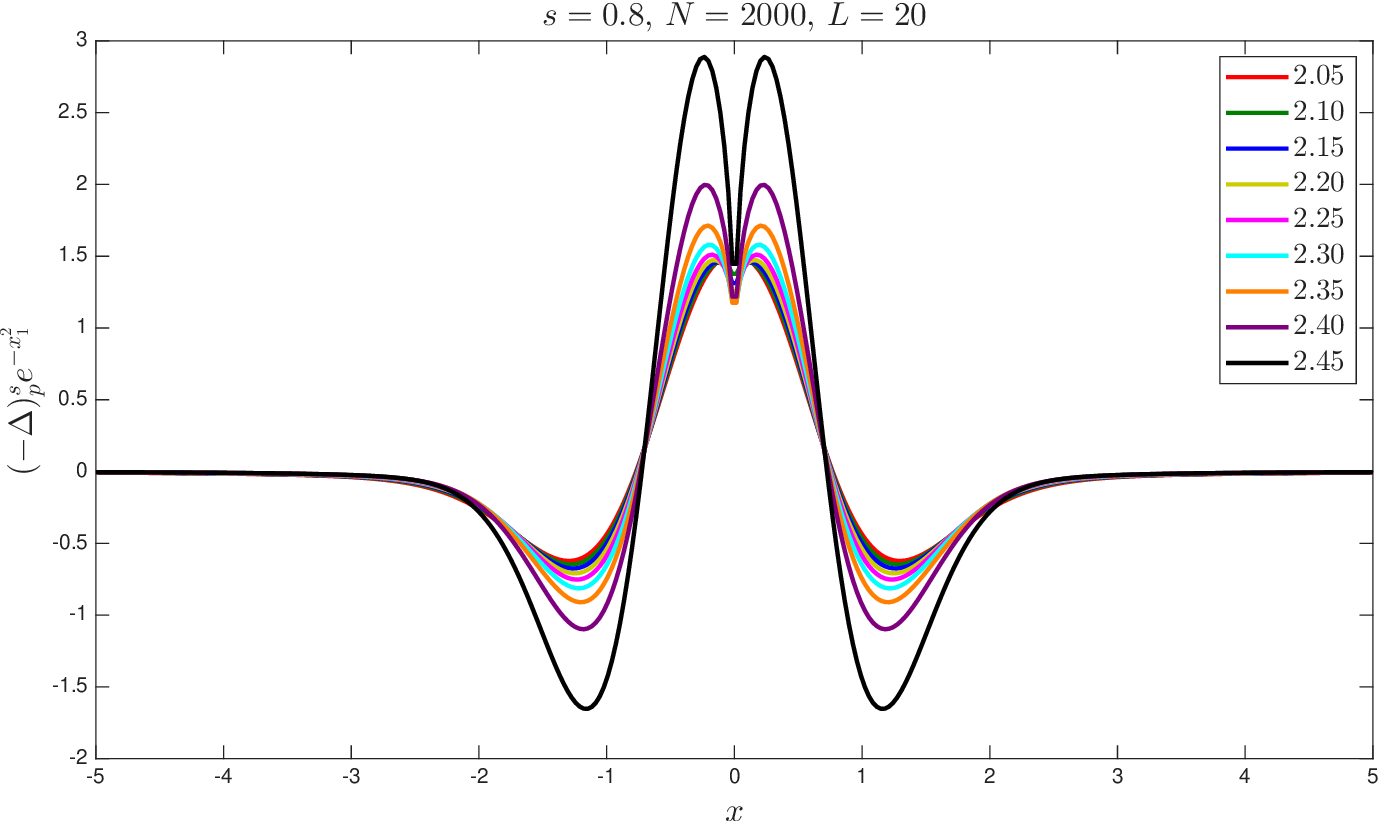}
	\caption{$(-\Delta)_p^s e^{-x_1^2}$, for $s = 0.8$, $N = 2000$, $L = 20$. Left: $p\in\{1.60, 1.65, \ldots, 1.95\}$. Right: $p\in\{2.05, 2.10, \ldots, 2.45\}$.}
	\label{f:fracplapexpn1p}
\end{figure}

\begin{figure}
	\centering
	\includegraphics[width=0.5\textwidth]{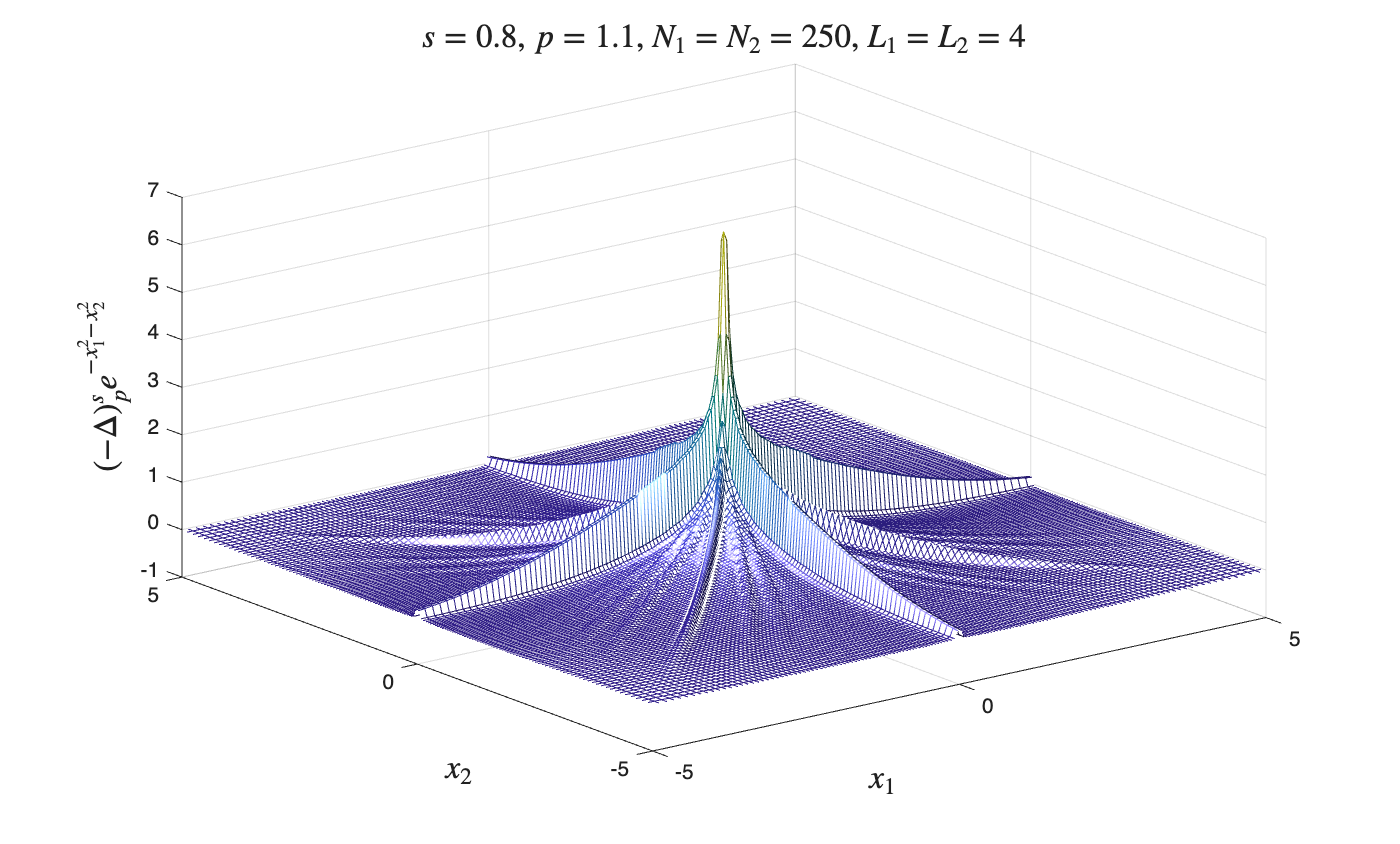}\includegraphics[width=0.5\textwidth]{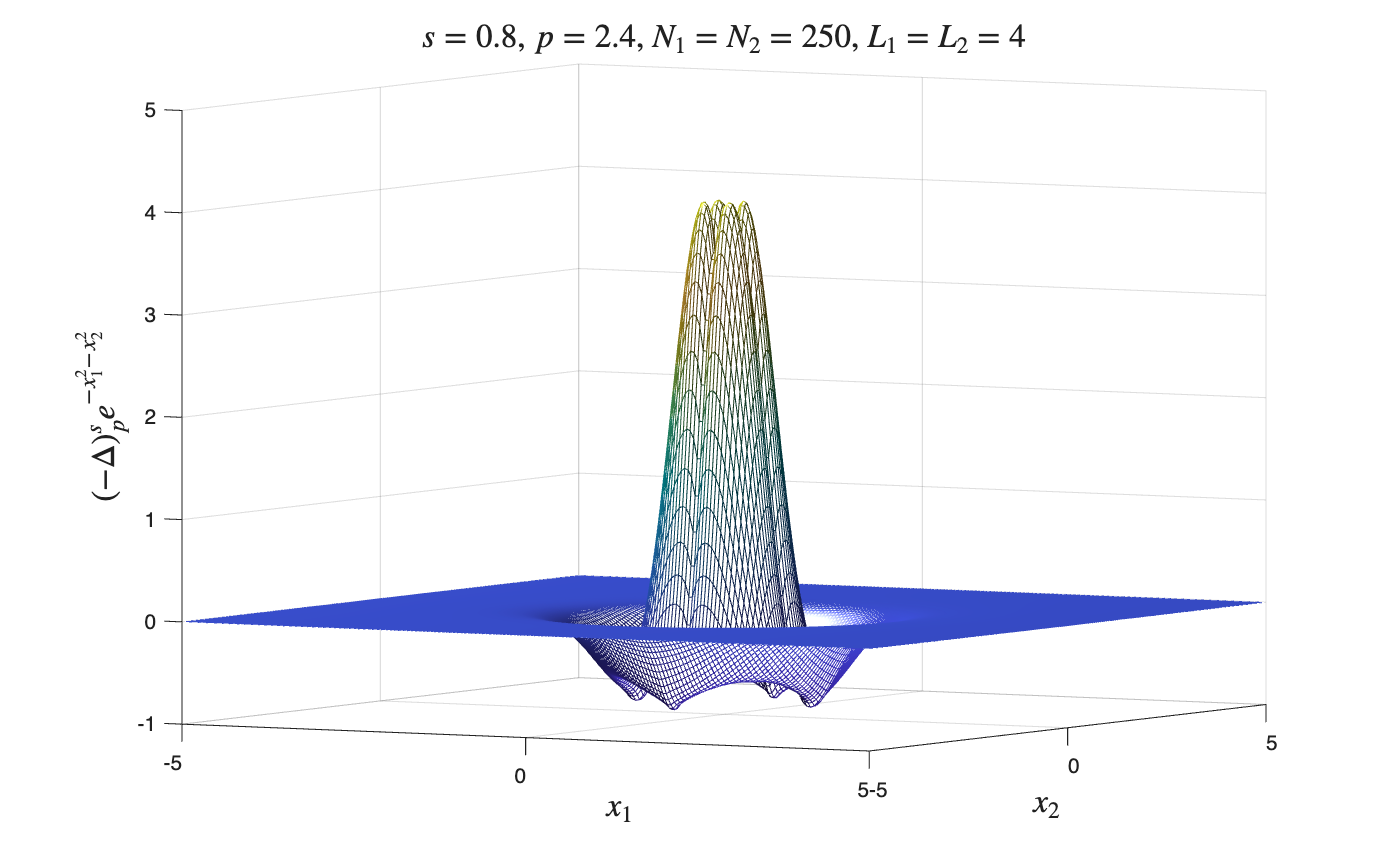}
	\caption{$(-\Delta)_p^se^{-x_1^2-x_2^2}$, for $s = 0.8$, $N = 250$, $L = 4$. Left: $p = 1.1$. Right: $p = 2.4$.}
	\label{f:fracplapexpn2}
\end{figure}

\section{Numerical experiments for the evolution equation of the fractional $p$-Laplacian}
\label{s:evolution}

In this section, we consider the following evolution equation:
\begin{equation}\label{FpL:evolution:eq}
\left\{
\begin{aligned}
& \frac{\partial u}{\partial t} + (-\Delta)^s_p u = 0, \quad t>0,\ \vx\in \R^n,
	\cr
& u(\vx,0) = u_0(\vx), \quad \vx\in \R^n.
\end{aligned}
\right.
\end{equation}
We have approximated $(-\Delta)^s_p$ numerically using our method, taking, as in Section~\ref{s:numericalfracLap}, $N = N_1 = \ldots = N_n$ and $L = L_1 = \ldots = L_n$. Then, we have employed a classical fourth-order Runge--Kutta scheme for the time integration of \eqref{FpL:evolution:eq}. Our goal is to illustrate numerically the convergence of solutions towards the self-similar profile as $t\to\infty$, as established analytically in \cite{Vazquez2020, Vazquez2021, Vazquez2022}.

We summarize briefly what to expect according to these analytical results. Let
\begin{equation*}
p_c = \frac{2n}{n + s}, \quad s\in(0,1).
\end{equation*}
If $p>p_c$, and, initially, $u_0 \in L^1(\R^n)$, then the solution to the initial value problem conserves its initial mass
\begin{equation*}
M = \int_{\R^n} u_0(\vx)d\vx,
\end{equation*}
i.e.,
\begin{equation*}
\int_{\R^n} u(\vx,t)d\vx = M, \quad \forall t > 0,
\end{equation*}
and converges to the self-similar solution $u_M$ in, for instance, the following sense:
\begin{equation}
\label{e:limtuuM}
\lim_{t \to \infty} \|u(\cdot,t) - u_M(\cdot,t)\|_1 = 0,
\end{equation}
where the function $u_M$ has the form
\begin{equation}\label{e:ss}
u_M(\vx,t) = M^{sp\beta} t^{-\alpha}F(\|\vr\|_2), \mbox{ with } \vr(\vx, t) =  M^{(2-p)\beta} t^{-\beta} \vx,\quad \alpha = \beta n,\quad
\beta = \frac{1}{sp - n(2-p)}.
\end{equation}
The existence of $F$, and the convergence \eqref{e:limtuuM}, hold under certain restrictions on $s$ and $p$. Specifically, $F$ is a continuous, positive, radially symmetric, and decreasing function, defined for $p > p_c$, that decays to $0$ as $\|\vr\|_2\to\infty$ according to
\begin{equation}\label{e:norm:fast:diff}
F(\|\vr\|_2) \approx \|\vr\|_2^{-n-sp}, 
\quad \text{when } p > 2 \text{ or } p \in (p_1,2),
\end{equation}
and
\begin{equation}\label{e:veryfast:diff}
F(\|\vr\|_2) \approx \|\vr\|_2^{-sp/(2-p)}, \quad \text{for } p \in (p_c,p_1),
\end{equation}
where $p_1\in(p_c,2)$ is the value for which the exponents in \eqref{e:norm:fast:diff} and \eqref{e:veryfast:diff} coincide, and is defined implicitly as 
\begin{equation*}
s\, p_1 (p_1 - 1) = n(2 - p_1),
\end{equation*}
i.e.,
\begin{equation*}
p_1 = \frac{s - n + \sqrt{n^2 + 6ns + s^2}}{2s}.
\end{equation*}
In what follows, we present numerical examples for $n = 1$ and $n = 2$, and for the three regimes that arise, according to the value of $p$, namely $p\in(p_c,p_1)$, $p\in(p_1,2)$ and $p>2$; the first regime is dictated by \eqref{e:veryfast:diff}, whereas the second and third regimes are dictated by \eqref{e:norm:fast:diff}. Moreover, in our numerical experiments, we have restricted the range of parameters so that $sp < 2$ (see Section~\ref{s:numericalfracpLap}); hence, in the third regime, $p\in(2,2/s)$.

To confirm numerically the self-similar behavior predicted by \eqref{e:limtuuM}--\eqref{e:ss}, we have considered in all the experiments the following initial condition:
\begin{equation}\label{init:evol}
u_0(\vx) = \exp\left(-\sum_{i=1}^n x_i^2\right).
\end{equation}
Observe that, in this case, when $n = 1$, performing the change of variable $x = y^{1/2}$, the mass can be computed explicitly:
\begin{equation*}
M = \int_\R u_0(x)dx = \int_{-\infty}^{\infty}e^{-x^2}dx = 2\int_0^{\infty}e^{-x^2}dx = \int_0^{\infty}y^{1/2-1}e^{-y}dy = \Gamma(1/2) = \sqrt{\pi},
\end{equation*}
and, in general, when $n\in\mathbb N$, applying Fubini,
\begin{equation}
\label{e:M}
M = \int_{\R^n}u_0(\vx)d\vx = \int_{\R^n}e^{-x_1^2 - \ldots - x_n^2}d\vx = \left(\int_{-\infty}^{\infty}e^{-x^2}dx\right)^n = \pi^{n/2}, \quad n\in\mathbb N.
\end{equation}
After computing the numerical solution, we transform it into the self-similar variables given in \eqref{e:ss}. With respect to $M$, instead of using its exactly predicted value \eqref{e:M}, we have approximated it spectrally. More precisely, if $f:\R\to \CC$, performing the change of variable $x = L\cot(\xi)$, 
\begin{equation}
\label{e:int1D}
\int_{-\infty}^{\infty}f(x)dx = L\int_{0}^{\pi}\frac{f(L\cot(\xi))}{\sin^2(\xi)}d\xi \approx \frac{L\pi}{N}\sum_{j=1}^N\frac{f(L\cot(\xi_j))}{\sin^2(\xi_j)},
\end{equation}
where the nodes $\{\xi_j\}$ are given by \eqref{e:xi}, i.e., we have applied the midpoint rule, which gives us spectral accuracy if the $\pi$-periodic function $f(L\cot(\xi_j))/\sin^2(\xi_j)$ is regular. Likewise, if $f:\R^2\to \CC$, performing the changes of variable $x_1 = L\cot(\xi_1)$ and $x_2 = L\cot(\xi_2)$,
\begin{align}
\label{e:int2D}
\int_{-\infty}^{\infty}\int_{-\infty}^{\infty}f(x_1,x_2)dx_1dx_2 & = L^2\int_{0}^{\pi}\int_{0}^{\pi}\frac{f(L\cot(\xi_1),L\cot(\xi_2))}{\sin^2(\xi_1)\sin^2(\xi_2)}d\xi_1d\xi_2
	\cr
& \approx \frac{L^2\pi^2}{N^2}\sum_{j_1=1}^N\sum_{j_2=1}^N\frac{f(L\cot(\xi_{1,j_1}),L\cot(\xi_{2,j_2}))}{\sin^2(\xi_{1,j_1})\sin^2(\xi_{2,j_2})},
\end{align}
where the nodes are defined as in \eqref{e:xi} for each coordinate.

At this point, it is important to note that the choice of $L$ plays an important role in the accuracy of the quadrature formulas in \eqref{e:int1D} (for $n = 1$) and \eqref{e:int2D} (for $n = 2$). To illustrate this, we have computed the absolute value of the difference between $\sqrt\pi$ and the approximation of the integral of  \eqref{init:evol} (for $n = 1$) over $\R$; and  the absolute value of the difference between $\pi$ and the approximation of the integral of \eqref{init:evol} (for $n = 2$) over $\R^2$. On the left-hand side of Figure~\ref{f:errormass}, we have plotted those errors, for $N = 50$, and $L \in \{0.01, 0.02, \ldots, 50\}$; and on the right-hand side of Figure~\ref{f:errormass}, we have done the same, for $N = 1000$, and $L \in\{0.01, 0.02, \ldots, 1000\}$. As can be seen, the set of values of $L$ for which spectral accuracy is achieved is much larger for $N = 1000$ than for $N = 50$, and, in general, when $L$ attains a certain value, the accuracy quickly diminishes. This observation is relevant, because, in order to capture numerically the self-similar behavior of the solutions of \eqref{FpL:evolution:eq}, large values of $L$ may be required, and the simulation may have to be run until large times $t$. Hence, if additionally we want to compute accurately the mass, the computational time may become prohibitive, especially in dimensions $n > 1$. In practice, however, this does not pose a problem, because, in the numerical simulation of the evolution of \eqref{FpL:evolution:eq}, we have found that the approximation of the mass given by \eqref{e:int1D} and \eqref{e:int2D} is well preserved over time, even in those cases where the real value of the mass is not correctly approximated. Therefore, since we do not know an analytical solution for \eqref{FpL:evolution:eq}, and we are only interested in showing that the self-similar behavior in the numerical solutions happens as $t$ grows, taking in \eqref{e:ss} for each time the numerical approximation of the mass given by \eqref{e:int1D} (for $n = 1$) and \eqref{e:int2D} (for $n = 2$) is enough. For clarity, we denote this time-dependent numerical approximation by $M(t)$, to distinguish it from the exact (constant) mass $M$.

\begin{figure}
	\centering
	\includegraphics[width=0.5\textwidth]{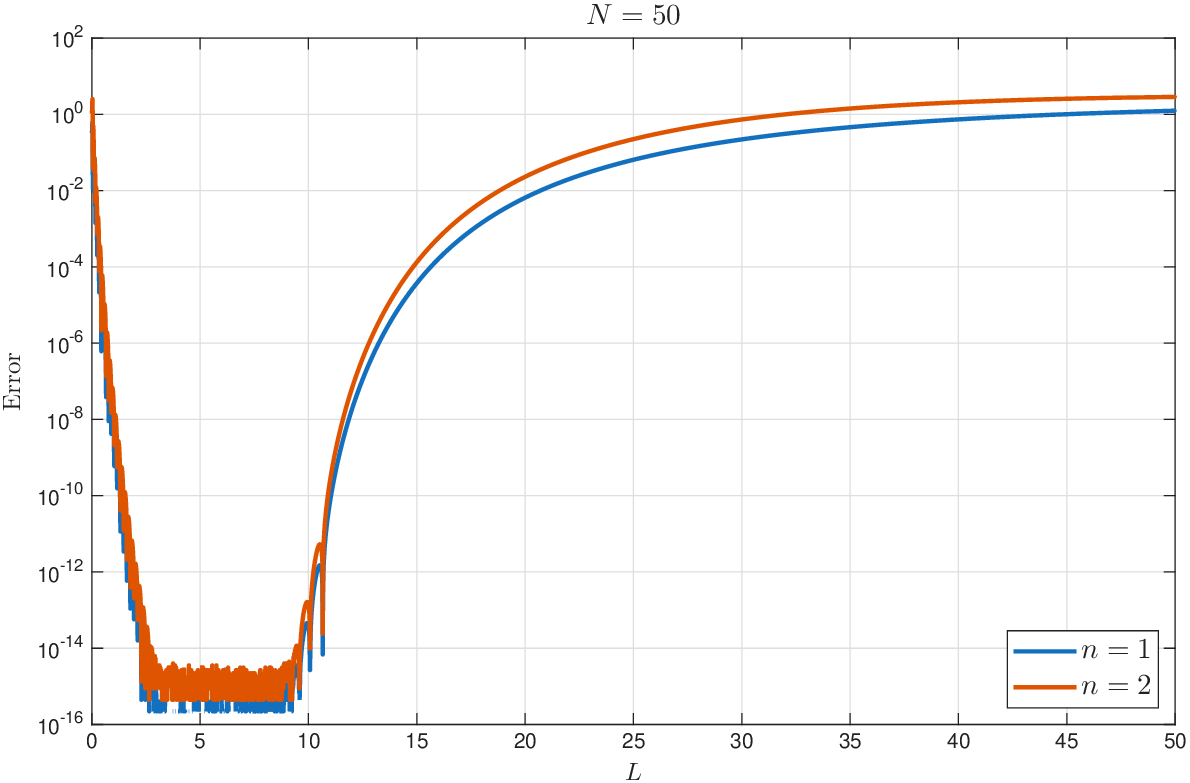}\includegraphics[width=0.5\textwidth]{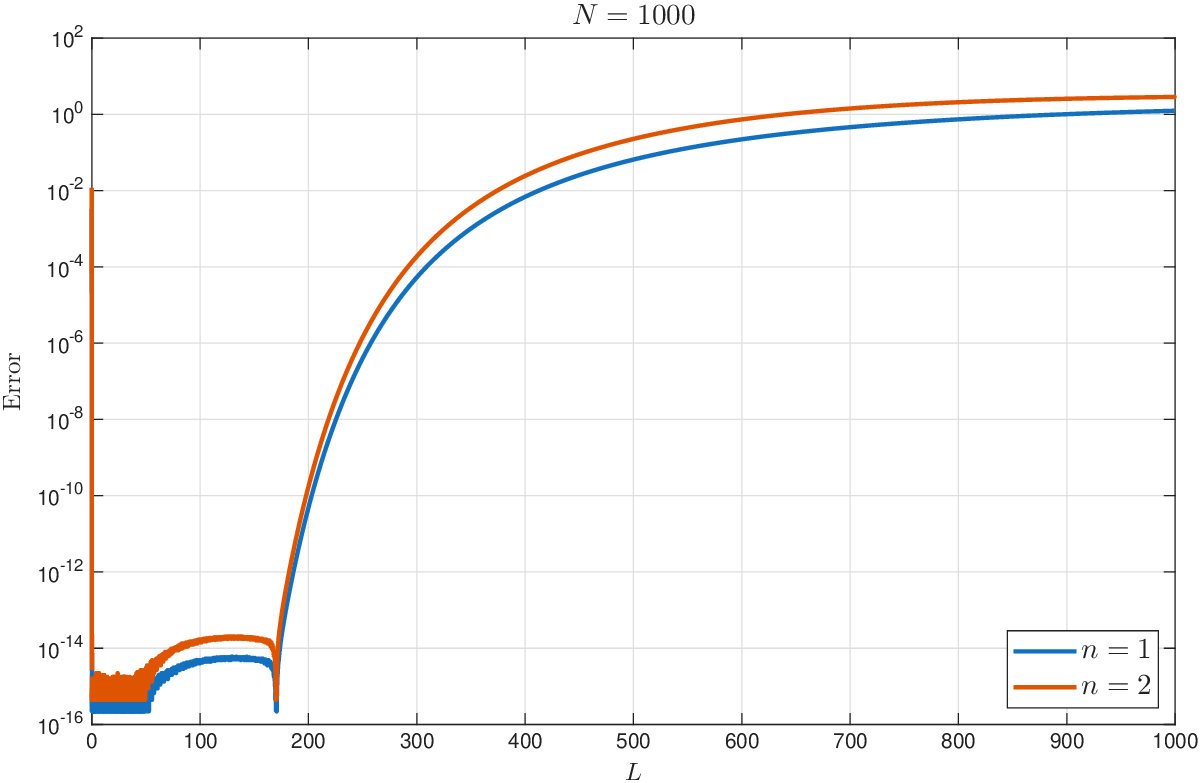}
	\caption{Error in the spectral approximation of the mass $M$ in \eqref{e:M}, computed by means of \eqref{e:int1D} (when $n = 1$) and \eqref{e:int2D} (when $n = 2$). Left: $N = 50$. Right: $N = 1000$.}
	\label{f:errormass}
\end{figure}

Bearing in mind \eqref{e:ss}, we have performed twelve numerical experiments, six for $n = 1$ and six for $n = 2$. In each case, we have considered two values of $s$, one close to zero and one close to $1$, namely $s = 0.2$ and $s = 0.8$. Furthermore, for each setting, we have chosen three values of $p$, satisfying respectively $p\in(p_c, p_1)$, $p\in(p_1, 2)$, and $p\in(2, 2/s)$. In the one-dimensional case, we have plotted $v(x,t) = u(x,t)M(t)^{-sp\beta}t^\alpha$ against $r(x,t) = M(t)^{(2-p)\beta}t^{-\beta}x$. In the two-dimensional case, we have exploited the radial symmetry and restricted the visualization to the section $y = 0$; therefore, we have plotted $v(x,t) = u(x,0,t)M(t)^{-sp\beta}t^\alpha$ versus $r(x,t) = M(t)^{(2-p)\beta}t^{-\beta}x$. The comparison of these profiles, displayed in Figures~\ref{f:n1first}--\ref{f:n2third}, shows that for sufficiently large times they overlap, thereby providing consistent numerical evidence of convergence towards the self-similar profile as $t \to \infty$. Figures~\ref{f:n1first}--\ref{f:n1third} correspond to the one-dimensional case, whereas Figures~\ref{f:n2first}--\ref{f:n2third} correspond to the two-dimensional case. Moreover, Figures~\ref{f:n1first} and~\ref{f:n2first} correspond to the case $p\in(p_c, p_1)$; Figures~\ref{f:n1second} and~\ref{f:n2second} to the case $p\in(p_1, 2)$; and Figures~\ref{f:n1third} and~\ref{f:n2third} to the case $p\in(2, 2/s)$. In all figures, the left panel corresponds to $s = 0.2$ and the right panel to $s = 0.8$. The choice of $N$, $L$, the time step $\Delta t$, and the time range over which the self-similar profiles can be observed varies considerably from one experiment to another. Some cases pose no particular difficulty, whereas others are very challenging. For instance, on the left-hand side of Figure~\ref{f:n1third}, where $n = 1$, $s = 0.2$, and $p\in(2,2/s)$, taking $\Delta t = 1$ is sufficient to reach large times and obtain the self-similar profiles with great accuracy. By contrast, on the left-hand side of Figure~\ref{f:n2first}, where $n = 2$, $s = 0.2$, and $p\in(p_c,p_1)$, capturing the self-similar behavior has required a very large value of $L = 10^4$, which causes the horizontal axis to be anomalously compressed and the vertical axis to be anomalously stretched. We acknowledge that this last figure is visually poor; however, it required more than a full day of computation to generate, and obtaining a better-tuned result would require a significantly larger value of $N$, which would make the computational cost prohibitive. We therefore regard the figures presented here as sufficient evidence of the self-similar behavior, and consider the production of higher-quality figures in the most challenging parameter regimes to lie beyond the scope of this paper.

\begin{figure}
	\centering
	\includegraphics[width=0.5\textwidth]{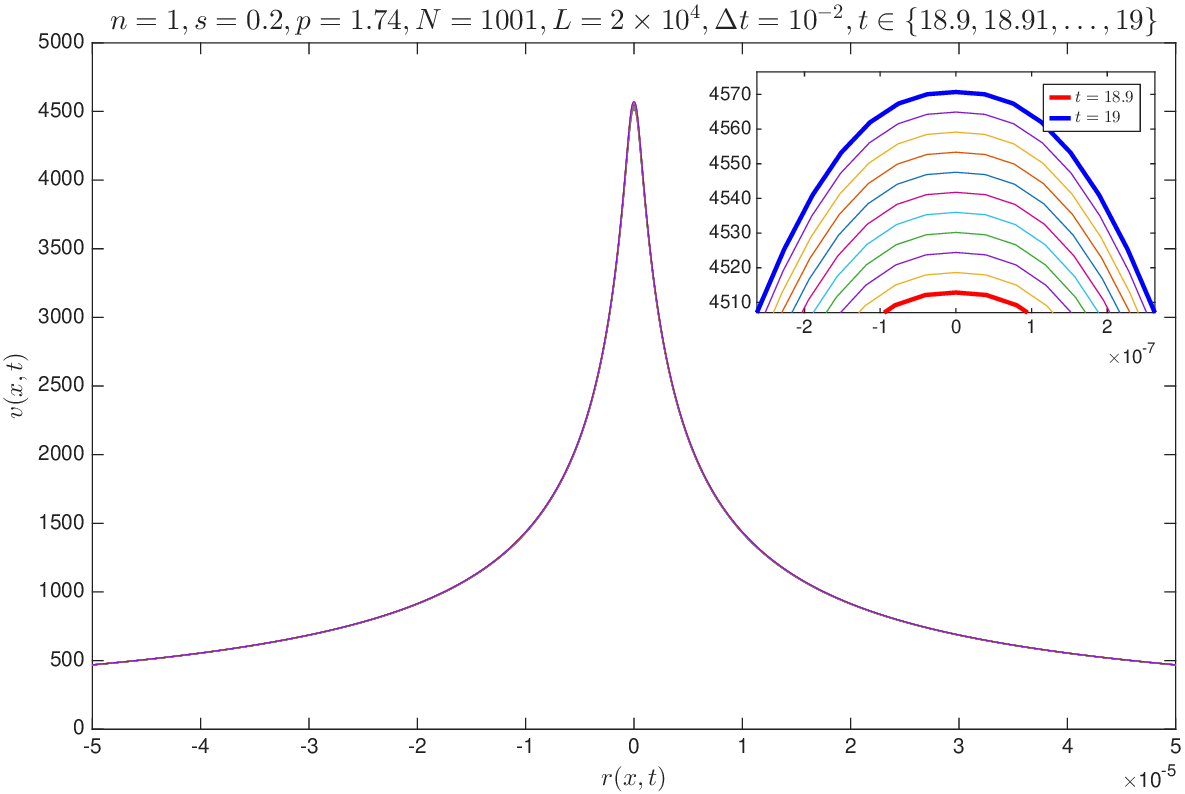}\includegraphics[width=0.5\textwidth]{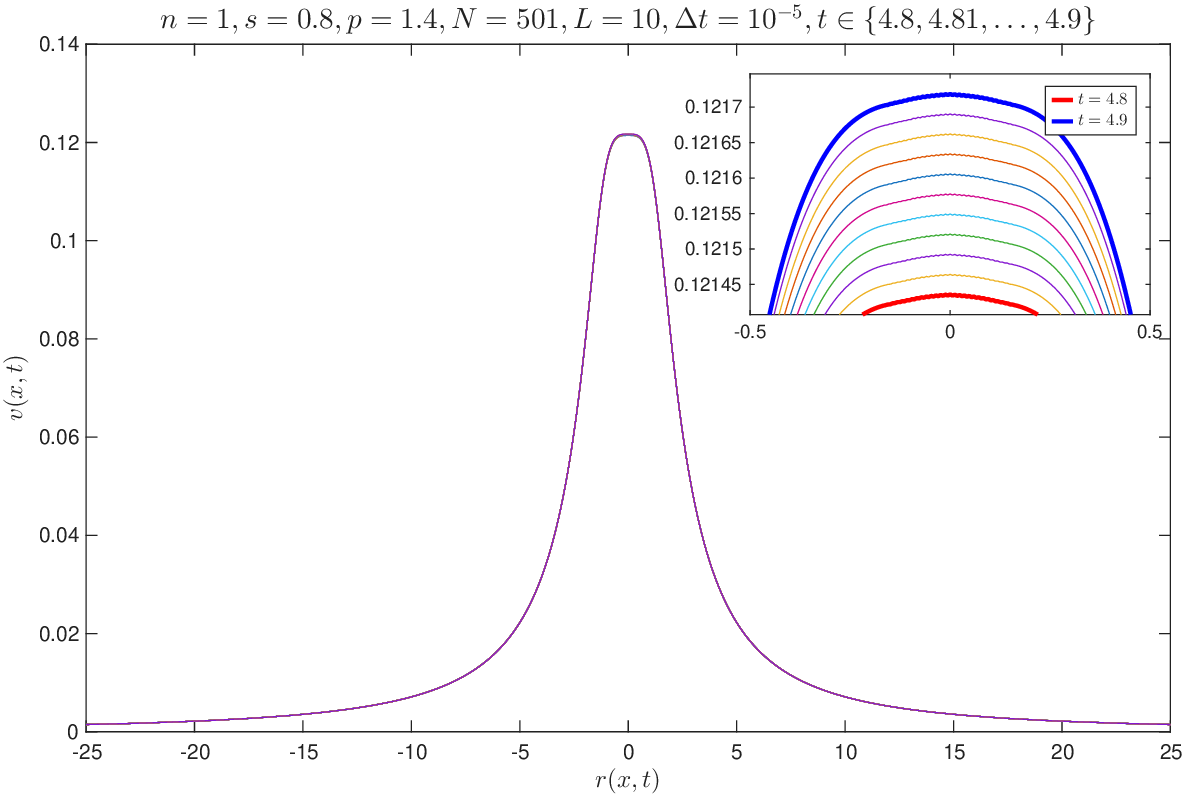}
	\caption{Solutions in the self-similar variables $v(x,t) = u(x,t)M(t)^{-sp\beta}t^\alpha$ and $r(x,t) = M(t)^{(2-p)\beta}t^{-\beta}x$, for $n = 1$ and $p\in(p_c, p_1)$. Left: $s = 0.2$, $p = 1.74$, with $N = 1001$, $L = 2\times10^4$, $\Delta t = 10^{-2}$, $t\in\{18.9, 18.91, \ldots, 19\}$. Right: $s = 0.8$, $p = 1.4$, with $N = 501$, $L = 10$, $\Delta t = 10^{-5}$, $t\in\{4.8, 4.81, \ldots, 4.9\}$.}
	\label{f:n1first}
\end{figure}

\begin{figure}
	\centering
	\includegraphics[width=0.5\textwidth]{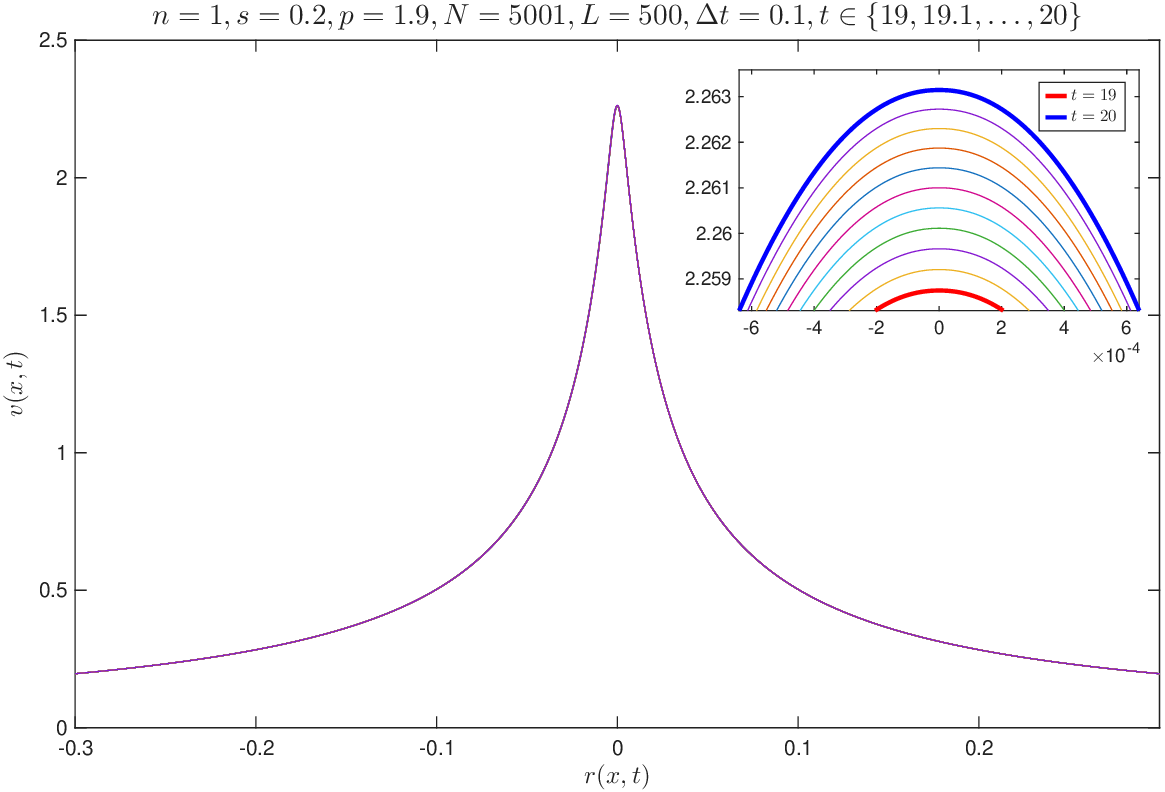}\includegraphics[width=0.5\textwidth]{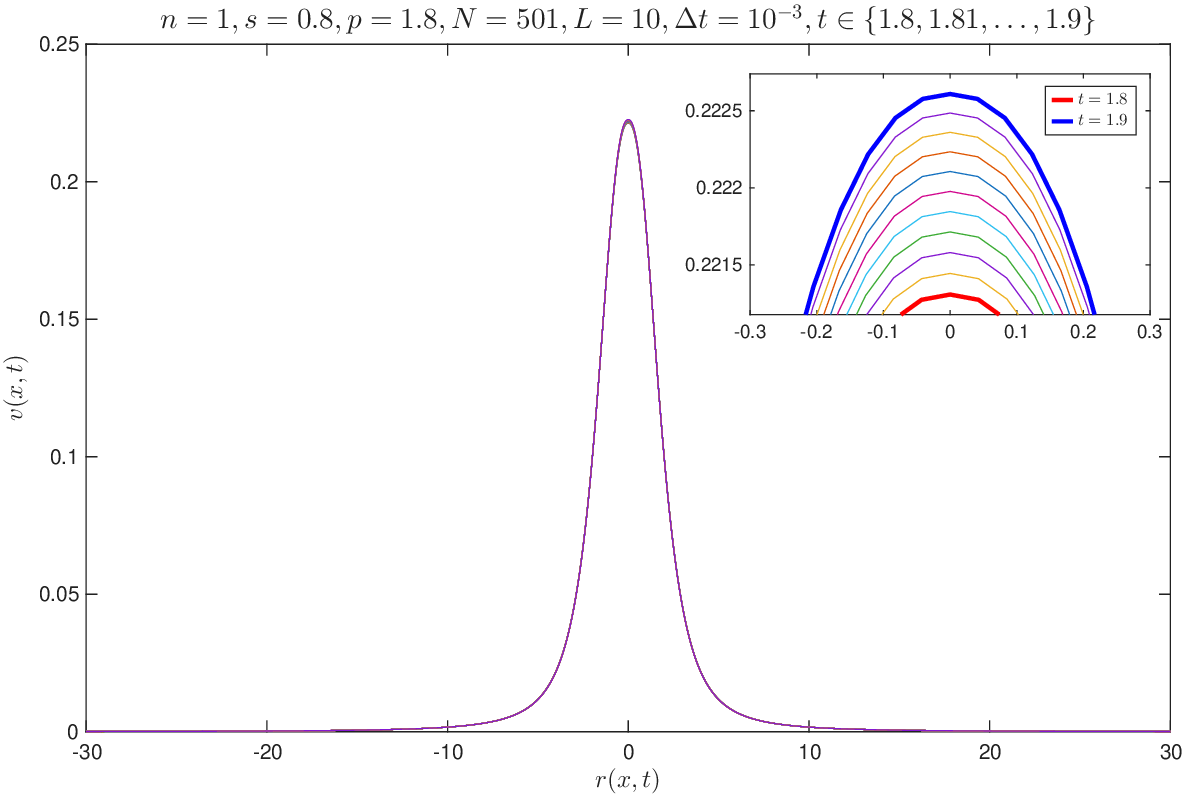}
	\caption{Solutions in the self-similar variables $v(x,t) = u(x,t)M(t)^{-sp\beta}t^\alpha$ and $r(x,t) = M(t)^{(2-p)\beta}t^{-\beta}x$, for $n = 1$ and $p\in(p_1, 2)$. Left: $s = 0.2$, $p = 1.9$, with $N = 5001$, $L = 500$, $\Delta t = 0.1$, $t\in\{19, 19.1, \ldots, 20\}$. Right: $s = 0.8$, $p = 1.8$, with $N = 501$, $L = 10$, $\Delta t = 10^{-3}$, $t\in\{1.8, 1.81, \ldots, 1.9\}$.}
	\label{f:n1second}
\end{figure}

\begin{figure}
	\centering
	\includegraphics[width=0.5\textwidth]{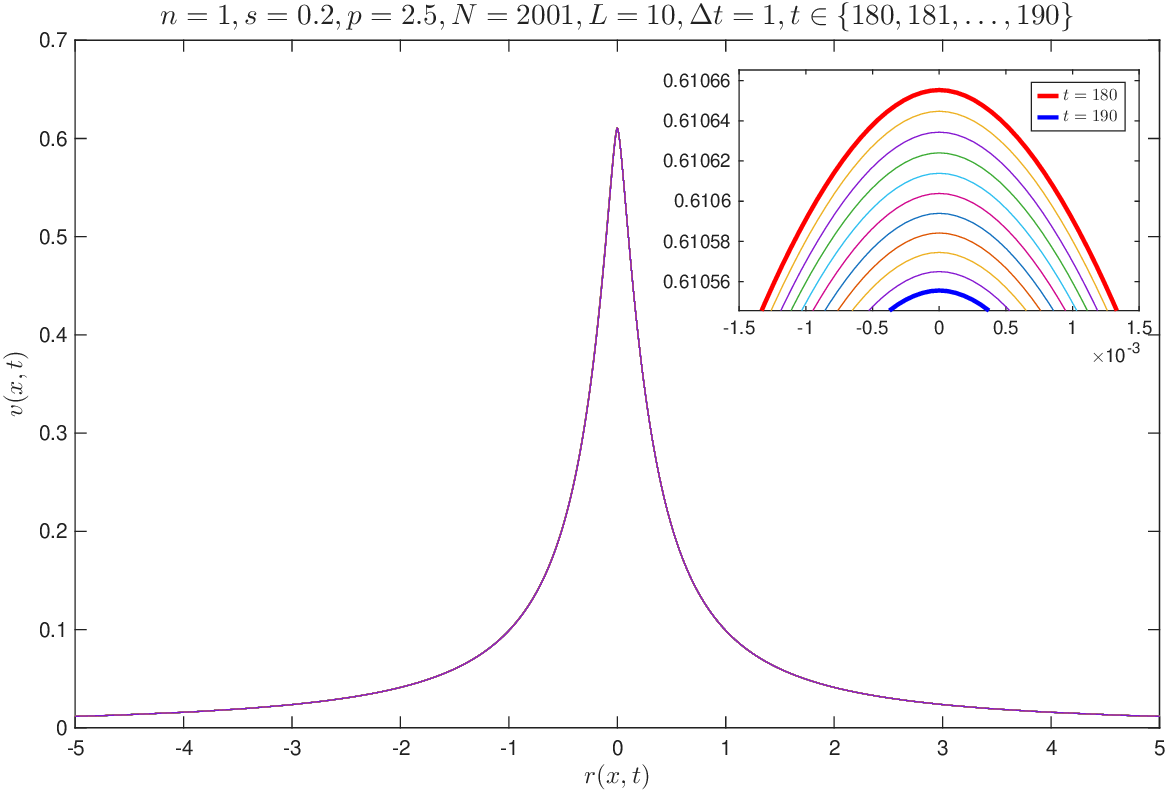}\includegraphics[width=0.5\textwidth]{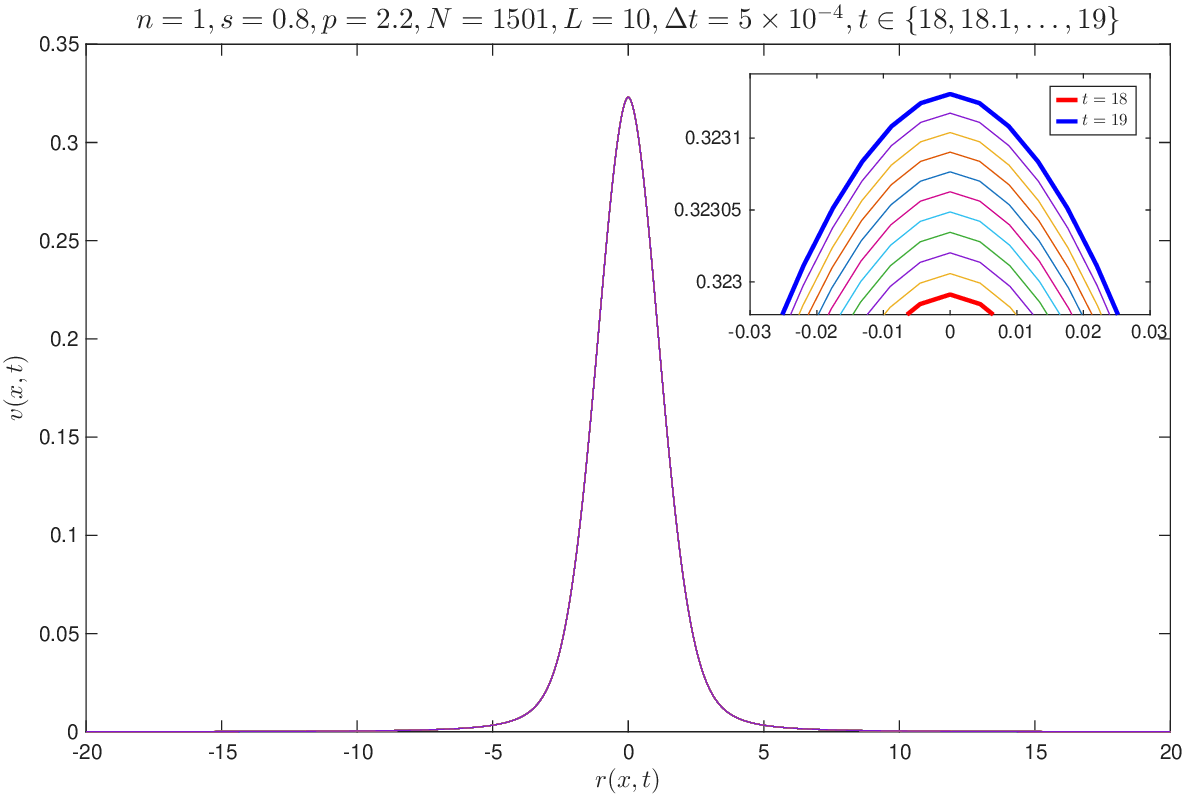}
	\caption{Solutions in the self-similar variables $v(x,t) = u(x,t)M(t)^{-sp\beta}t^\alpha$ and $r(x,t) = M(t)^{(2-p)\beta}t^{-\beta}x$, for $n = 1$ and $p\in(2,2/s)$. Left: $s = 0.2$, $p = 2.5$, with $N = 2001$, $L = 10$, $\Delta t = 1$, $t\in\{180, 181, \ldots, 190\}$. Right: $s = 0.8$, $p = 2.2$, with $N = 1501$, $L = 10$, $\Delta t = 5\times10^{-4}$, $t\in\{18, 18.1, \ldots, 19\}$.}
	\label{f:n1third}
\end{figure}

\begin{figure}
	\centering
	\includegraphics[width=0.5\textwidth]{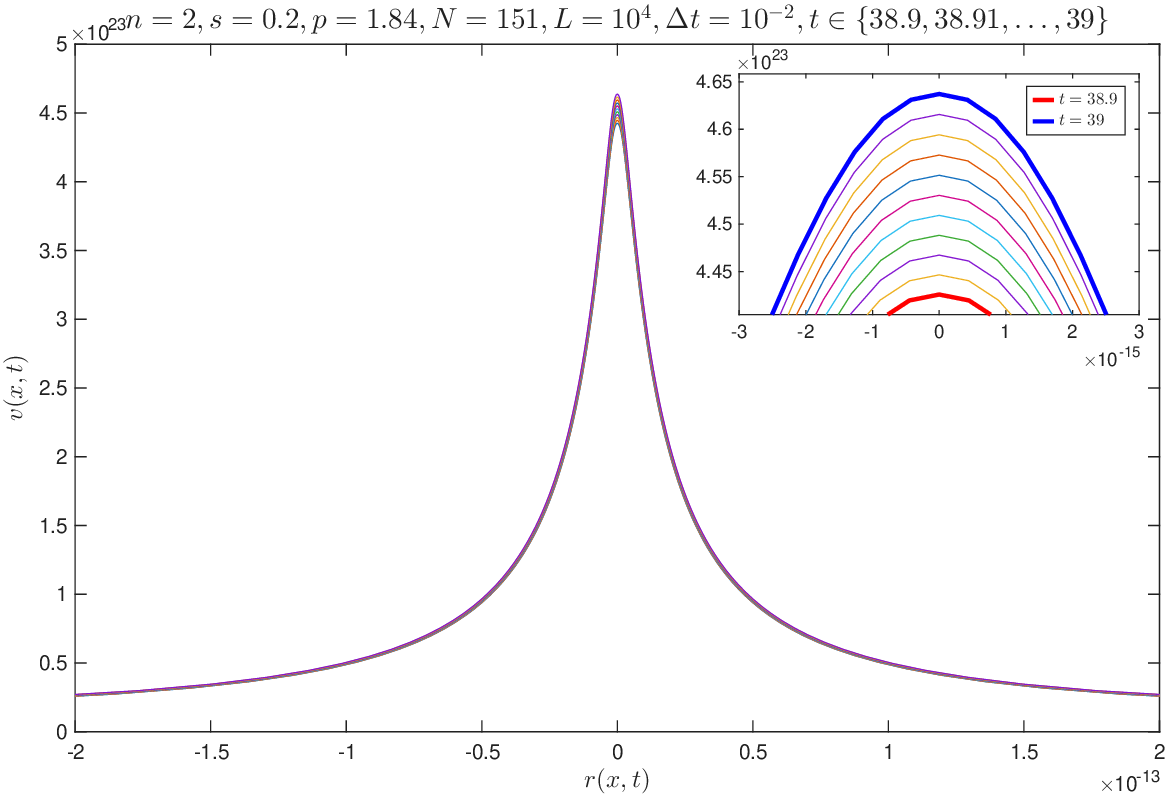}\includegraphics[width=0.5\textwidth]{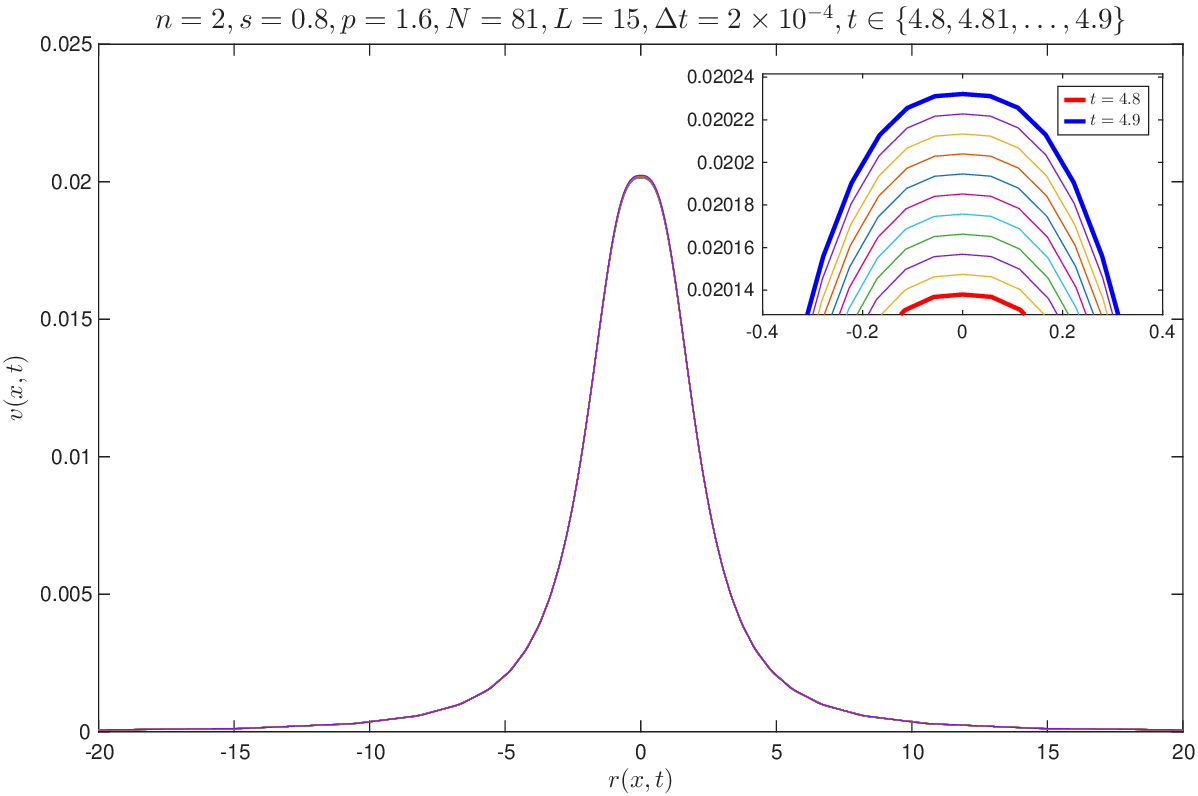}
	\caption{Solutions in the self-similar variables $v(x,t) = u(x,0,t)M(t)^{-sp\beta}t^\alpha$ and $r(x,t) = M(t)^{(2-p)\beta}t^{-\beta}x$, for $n = 2$ and $p\in(p_c, p_1)$. Left: $s = 0.2$, $p = 1.84$, with $N = 151$, $L = 10^4$, $\Delta t = 10^{-2}$, $t\in\{38.9, 38.91, \ldots, 39\}$. Right: $s = 0.8$, $p = 1.6$, with $N = 81$, $L = 15$, $\Delta t = 2\times10^{-4}$, $t\in\{4.8, 4.81, \ldots, 4.9\}$.}
	\label{f:n2first}
\end{figure}

\begin{figure}
	\centering
	\includegraphics[width=0.5\textwidth]{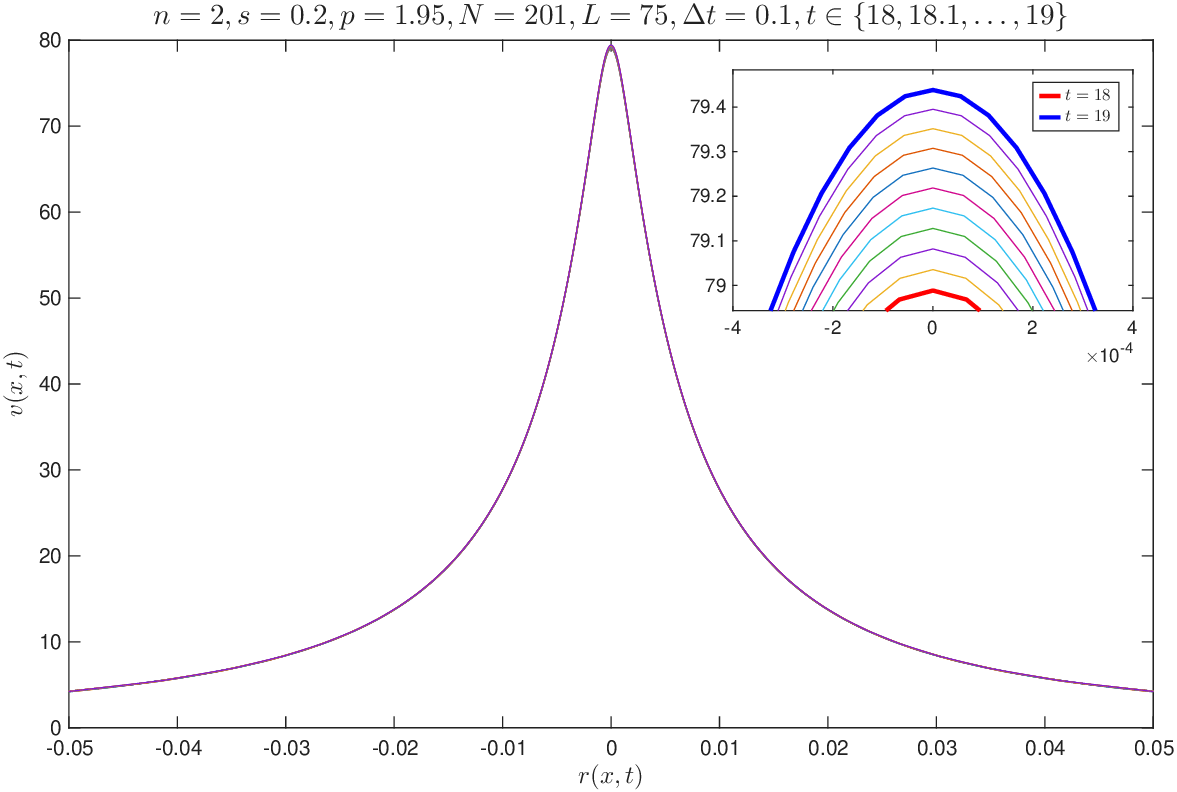}\includegraphics[width=0.5\textwidth]{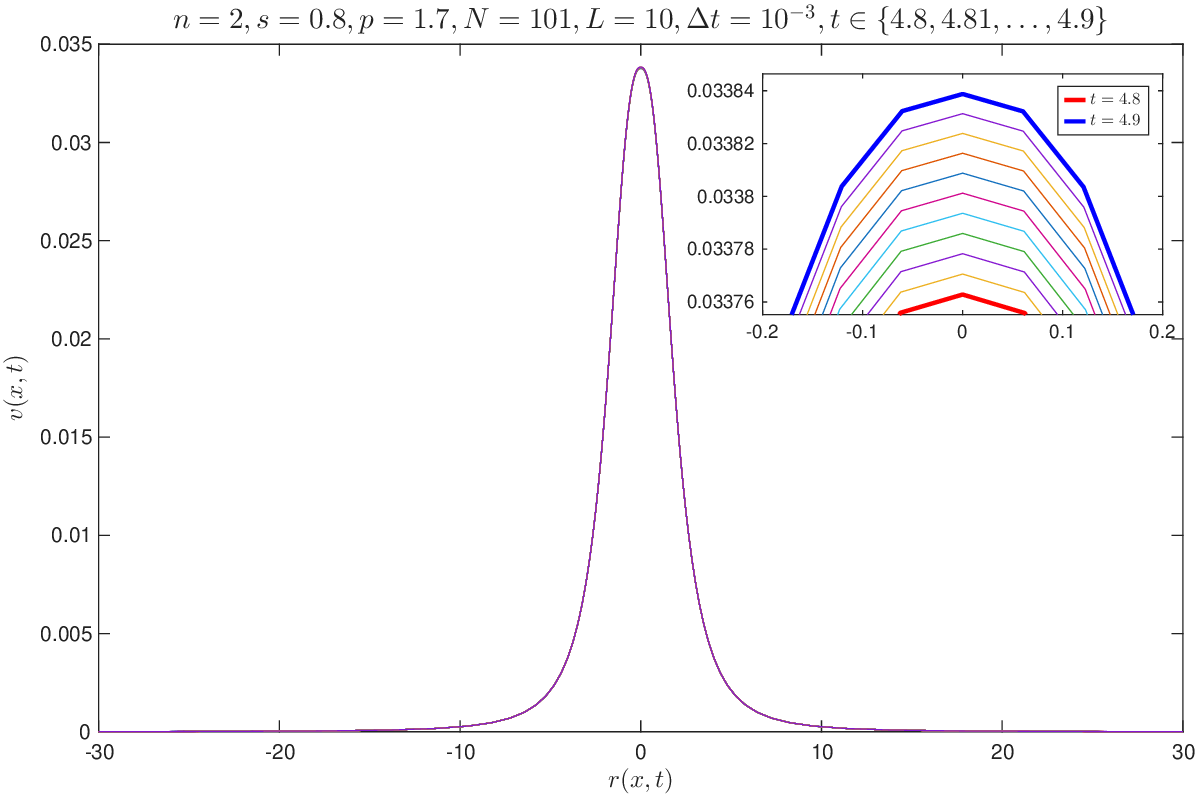}
	\caption{Solutions in the self-similar variables $v(x,t) = u(x,0,t)M(t)^{-sp\beta}t^\alpha$ and $r(x,t) = M(t)^{(2-p)\beta}t^{-\beta}x$, for $n = 2$ and $p\in(p_1, 2)$. Left: $s = 0.2$, $p = 1.95$, with $N = 201$, $L = 75$, $\Delta t = 0.1$, $t\in\{18, 18.1, \ldots, 19\}$. Right: $s = 0.8$, $p = 1.7$, with $N = 101$, $L = 10$, $\Delta t = 0.001$, $t\in\{4.8, 4.81, \ldots, 4.9\}$.}
	\label{f:n2second}
\end{figure}

\begin{figure}
	\centering
	\includegraphics[width=0.5\textwidth]{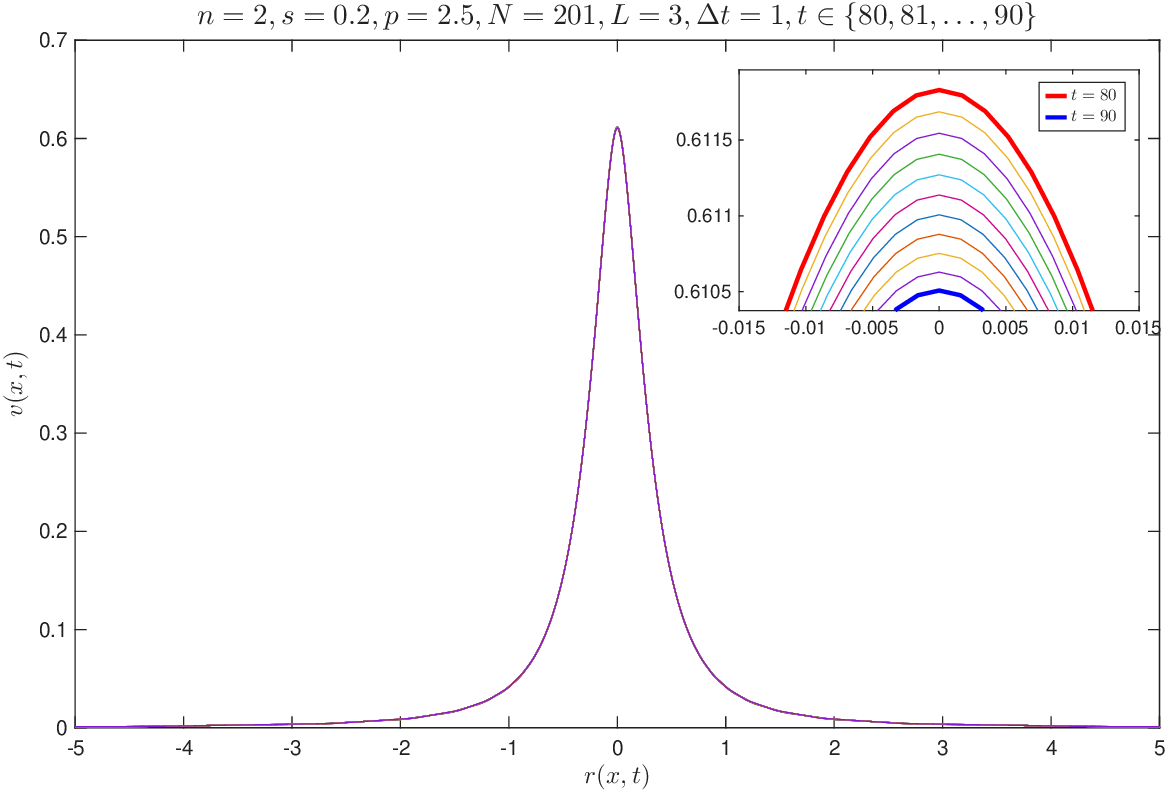}\includegraphics[width=0.5\textwidth]{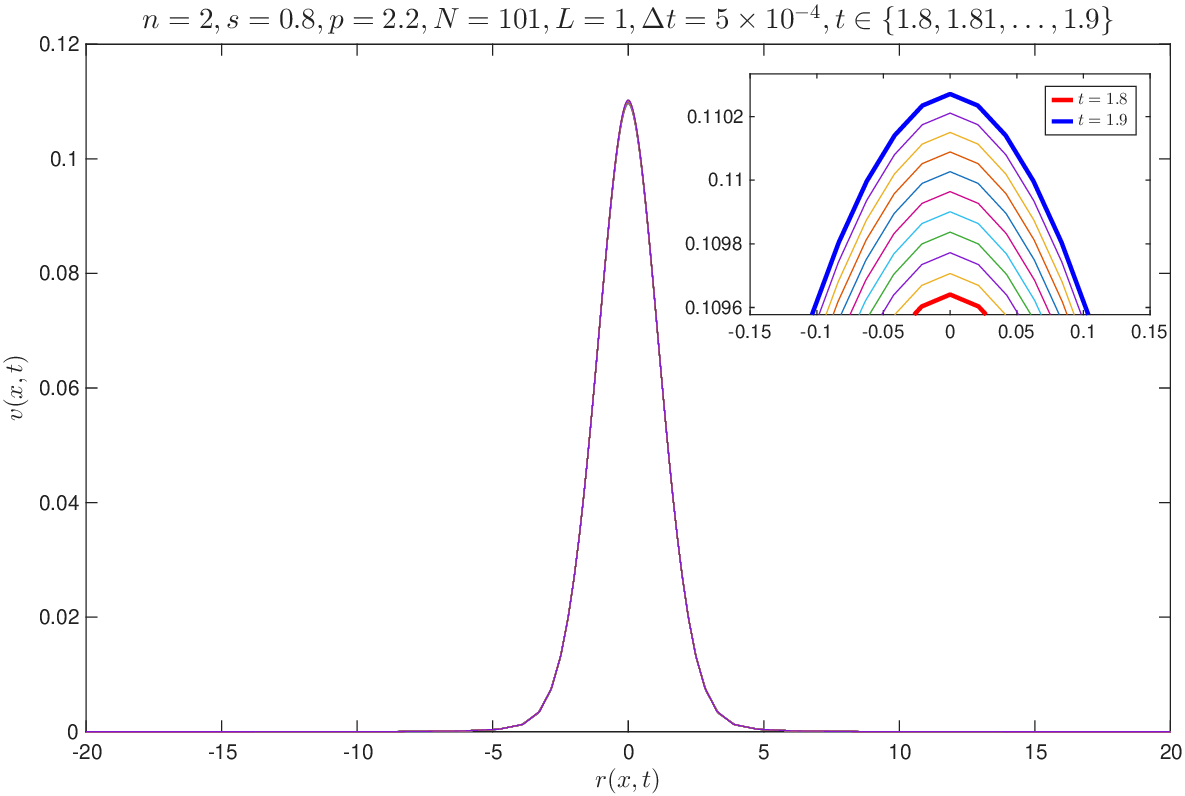}
	\caption{Solutions in the self-similar variables $v(x,t) = u(x,0,t)M(t)^{-sp\beta}t^\alpha$ and $r(x,t) = M(t)^{(2-p)\beta}t^{-\beta}x$, for $n = 2$ and $p\in(2, 2/s)$. Left: $s = 0.2$, $p = 2.5$, with $N = 201$, $L = 3$, $\Delta t = 1$, $t\in\{80, 81, \ldots, 90\}$. Right: $s = 0.8$, $p = 2.2$, with $N = 101$, $L = 1$, $\Delta t = 5\times10^{-4}$, $t\in\{1.8, 1.81, \ldots, 1.9\}$.}
	\label{f:n2third}
\end{figure}

\appendix

\section{Auxiliary lemmas}

\label{a:aux}

\begin{lemma}
\label{lem:I1}
Let $\mu < 0$ and $s\in(0,1)$. Then,
\begin{equation}
\label{e:I1}
I_1(\mu, s) = \int_0^\infty\frac{\mu t^{s-1}}{t-\mu}dt = -\pi\csc(\pi s)(-\mu)^s.
\end{equation}
\end{lemma}

\begin{proof}

We apply first the change of variable $t = -\mu x$, then
\begin{equation*}
I_1(\mu, s) = \int_0^\infty\frac{\mu (-\mu x)^{s-1}}{-\mu x-\mu}(-\mu)dx =  -(-\mu)^s\int_0^\infty\frac{x^{s-1}}{x + 1}dx.
\end{equation*}
Now, we apply the change of variable $x = y / (1 - y) \Longleftrightarrow y = x / (1 + x)$, then
\begin{equation*}
I_1(\mu, s) = -(-\mu)^s\int_0^1\frac{[y/(1-y)]^{s-1}}{y/(1-y) + 1}\frac{dy}{(1 - y)^2} = -(-\mu)^s\int_0^1y^{s-1}(1-y)^{1-s-1}dy = -(-\mu)^sB(s, 1 - s),
\end{equation*}
where $B(\cdot,\cdot)$ denotes Euler's beta function:
\begin{equation*}
B(p, q) = \int_0^1y^{p-1}(1-y)^{q-1}dy = \frac{\Gamma(p)\Gamma(q)}{\Gamma(p+q)}.
\end{equation*}
Hence,
\begin{equation*}
I_1(\mu, s) = -(-\mu)^s\frac{\Gamma(s)\Gamma(1-s)}{\Gamma(s+1-s)} = -(-\mu)^s\pi\csc(\pi s),
\end{equation*}
where we have used that $\Gamma(s)\Gamma(1- s) = \pi \csc(\pi s)$. This concludes the proof.

\end{proof}

\begin{lemma}
\label{lem:I2}
	Let $\mu < 0$ and $s\in(0,1)$. Then,
	\begin{equation}
	\label{e:I2}
	I_2(\mu, s) = \int_0^\infty \frac{e^{\mu t} - 1}{t^{1 + s}}dt = \Gamma(-s)(-\mu)^s.
	\end{equation}
\end{lemma}

\begin{proof}

Let $t > 0$ and $s > 0$, then
\begin{align*}
\int_0^\infty x^s e^{-tx}dx  & = \int_0^\infty (y/t)^s e^{-y}dy/t = \frac{1}{t^{1 + s}}\int_0^\infty y^{s+1-1} e^{-y}dy = \frac{\Gamma(1+s)}{t^{1 + s}}
	\cr
& \quad \Longleftrightarrow \frac{1}{t^{1 + s}} = \frac{1}{\Gamma(1+s)}\int_0^\infty x^s e^{-tx}dx.
\end{align*}
where we have performed the change of variable $y = tx$. Introducing the expression for $1 / t^{1 + s}$ into the definition of $I_2(\mu, s)$ in \eqref{e:I2}:
\begin{equation}
\label{e:I2b}
I_2(\mu, s) = \int_0^\infty \left[\frac{1}{\Gamma(1+s)}\int_0^\infty x^s e^{-tx}dx\right](e^{\mu t} - 1)dt
 = \frac{1}{\Gamma(1+s)}\int_0^\infty x^s \left[\int_0^\infty e^{-xt}(e^{\mu t} - 1)dt\right]dx,
\end{equation}
where we have changed the order of integration. The inner integral is absolutely convergent and can be calculated immediately:
\begin{equation*}
\int_0^\infty e^{-xt}(e^{\mu t} - 1)dt = \int_0^\infty (e^{(\mu - x)t} - e^{-xt})dt = \left.\frac{e^{(\mu - x)t}}{\mu - x} - \frac{e^{-xt}}{-x}\right|_{t = 0}^{t = \infty} = \frac{1}{x - \mu} - \frac{1}{x} = \frac{\mu}{x(x-\mu)}.
\end{equation*}
Introducing it into \eqref{e:I2b},
\begin{align*}
I_2(\mu, s) & = \frac{1}{\Gamma(1+s)}\int_0^\infty\frac{\mu x^{s-1}}{x-\mu}dx = \frac{1}{\Gamma(1+s)}I_1(\mu,s),
\end{align*}
where $I_1(\mu,s)$ is given by \eqref{e:I1}. Therefore,
\begin{align*}
I_2(\mu, s) = \frac{1}{\Gamma(1+s)}(-\pi)\csc(\pi s)(-\mu)^s = -(-\mu)^s\frac{\Gamma(s)\Gamma(1 - s)}{\Gamma(1+s)} = -(-\mu)^s\frac{\Gamma(s)(-s)\Gamma(-s)}{s\Gamma(s)} = \Gamma(-s)(-\mu)^s,
\end{align*}
where we have used that $\Gamma(s)\Gamma(1- s) = \pi \csc(\pi s)$.

\end{proof}

\section*{Data availability}

No data were used for the research described in this article.

\section*{Acknowledgments}

This work was partially supported by the research group grant IT1615-22 funded by the Basque Government, and by the projects PID2021-126813NB-I00 and PID2024-158099NB-I00 funded by MICIU/AEI/10.13039/501100011033 and by ERDF, EU.

\end{document}